\documentclass[12pt]{elsart}

\usepackage{amssymb}

\usepackage{color}
\usepackage{graphicx}\usepackage{subfigure}
\usepackage{amsmath}
\usepackage{extarrows}
\usepackage{amsfonts}
\usepackage{listings}
\usepackage{grffile}
\usepackage{float}
\usepackage{caption}
\usepackage[title]{appendix}
\allowdisplaybreaks
\usepackage{psfrag}

\usepackage{paralist}
\usepackage[colorlinks=true]{hyperref}

\usepackage{multirow,booktabs}

\DeclareMathOperator\sign{sgn}
\newtheorem{theorem}{Theorem}[section]
\numberwithin{equation}{section}
\numberwithin{figure}{section}
\numberwithin{table}{section}

\renewcommand{\vec}[1]{\mbox{\boldmath \small $#1$}}
\renewcommand{\bar}[1]{\mkern 1.5mu\overline{\mkern-1.5mu#1\mkern-1.5mu}\mkern 1.5mu}

\renewcommand{\qed}{\hfill \nobreak \ifvmode \relax \else
      \ifdim\lastskip<1.5em \hskip-\lastskip
      \hskip1.5em plus0em minus0.5em \fi \nobreak
      \vrule height0.75em width0.5em depth0.25em\fi}

\newtheorem{example}{Example}[section]

\newtheorem{remark}{Remark}[section]
\numberwithin{equation}{section}
\numberwithin{figure}{section}
\numberwithin{table}{section}

\newenvironment{proof}[1][Proof]{\begin{trivlist}
\item[\hskip \labelsep {\bfseries #1}]}{\end{trivlist}}
\renewcommand{\qed}{\hfill \nobreak \ifvmode \relax \else
      \ifdim\lastskip<1.5em \hskip-\lastskip
      \hskip1.5em plus0em minus0.5em \fi \nobreak
      \vrule height0.75em width0.5em depth0.25em\fi}
\begin{document}

\begin{frontmatter}
 \title{Second-order accurate BGK schemes for the special
 relativistic hydrodynamics with the Synge equation of state}

 \author{Yaping Chen}
 \ead{ypchen@nwpu.edu.cn}
  \address{
  NPU-UoG International Cooperative Lab for Computation \& Application in Cardiology, Northwestern Polytechnical University, Xi'an 710129, Shaanxi Province, P.R. China}
 \author{Yangyu Kuang},
 \ead{986254703@qq.com}
 \author[label2]{Huazhong Tang}
 \thanks[label2]{Corresponding author} 
 \ead{hztang@pku.edu.cn}
 \address{Center for Applied Physics and Technology, HEDPS and LMAM,
    School of Mathematical Sciences, Peking University, Beijing 100871,
    P.R.China}
 \date{\today{}}

 \maketitle

 \begin{abstract}
This paper extends the second-order accurate BGK finite
volume schemes for the ultra-relativistic flow simulations \cite{CHEN2017} to the 1D and 2D special relativistic hydrodynamics  with the Synge
 equation of state.
%
It is shown that such 2D schemes
 are very time-consuming  due to the moment integrals (triple integrals) so that they are no longer  practical.
 In view of this,  the
simplified BGK (sBGK) schemes are presented 
 by removing some terms in the approximate nonequilibrium distribution at the
 cell interface for the BGK scheme without loss of accuracy.
 They are practical because the moment integrals of the approximate distribution can be reduced to the single integrals by  some  coordinate transformations.
The 
relations between the left and right states of the
shock wave, rarefaction wave, and contact discontinuity are also discussed, so that the exact solution of the 1D Riemann problem could be derived and used for the numerical comparisons.
Several numerical experiments are conducted to demonstrate that the proposed {gas-kinetic  schemes} are accurate and stable.
A comparison of the sBGK schemes with the BGK scheme in one dimension shows that the former performs almost the same as the latter in terms of the accuracy and resolution, but is much more efficiency. 
 \end{abstract}

\end{frontmatter}


\section{Introduction}
\label{sec:intro}
 In many flow problems of astrophysical interest, the fluid moves at extremely high velocities near the speed of light, so that the relativistic effects become important. Relativistic hydrodynamics (RHD) plays a major role in astrophysics, plasma physics and nuclear physics etc.,
but
the dynamics of the relativistic system requires solving highly nonlinear governing equations, rendering the analytic treatment of practical problems extremely difficult. The numerical simulation is the primary and powerful way to study and understand the RHDs.

The pioneering numerical work in the field of numerical RHDs may date back to
the finite difference code via artificial viscosity  for the spherically symmetric general RHD equations in the Lagrangian coordinate \cite{May1966,May1967} and  for multi-dimensional RHD equations in the Eulerian coordinate \cite{R.Wilson1972}.
Since 1990s, 
various modern shock-capturing methods with an exact or approximate Riemann solver have been developed for the RHD equations. Some examples are the local characteristic approach \cite{Mart1991Numerical}, the two-shock approximation solvers \cite{balsara1994riemann,Dai1997}, the Roe solver \cite{F.Eulderink1995}, the flux corrected transport method \cite{duncan1994}, the flux-splitting method based on the spectral decomposition \cite{Donat1998}, the piecewise parabolic method \cite{Marti1996,Mignone2005}, the HLL (Harten-Lax-van Leer) method \cite{schneider1993new}, the HLLC (HLL-Contact) method \cite{mignone2005hllc} and the Steger-Warming flux vector splitting method \cite{Zhao2014Steger}.
The analytical solution of the  Riemann problem in relativistic hydrodynamics was studied in \cite{mart1994}.
Some other higher-order accurate methods have also been well studied in the literature, e.g. the ENO (essentially non-oscillatory) and weighted ENO (WENO) methods \cite{dolezal1995relativistic,del2002efficient,Tchekhovskoy2007},
the discontinuous Galerkin (DG) method \cite{RezzollaDG2011}, the adaptive moving mesh methods \cite{he2012adaptive1,he2012adaptive2},
the Runge-Kutta DG {methods} with WENO limiter \cite{zhao2013runge,ZhaoTang-CiCP2017,ZhaoTang-JCP2017},
the direct Eulerian GRP {schemes} \cite{yang2011direct,yang2012direct,wu2014third},
the local evolution Galerkin method \cite{wu2014finite},
and the two-stage fourth-order accurate time
  discretizations \cite{Yuan2020JCM}.
Recently some physical-constraints-preserving (PCP) schemes were developed for the special RHDs and relativistic magnetohydrodynamics (RMHD). They are the high-order accurate PCP finite difference WENO schemes and discontinuous Galerkin (DG) methods proposed
in \cite{wu2015high,wu2016physical,qin2016bound,Ling2019JCP,wu2017m3as,wu2018zamp}.
The entropy-stable schemes were also developed for the special RHD or RMHD
equations \cite{DUAN-JCP2020A,DUAN-AAMM2020,DUAN-JCP2020B}.
The readers are also referred to the early review articles \cite{marti2003review,marti2015review,Font2008} as well as references therein.


It is noted that most of those methods are based on macroscopic continuum description and the most commonly used EOS is designed for the gas with constant ratio of specific heats. However, such EOS is a poor approximation for most relativistic astrophysical flows and is essentially valid only for the  either sub-relativistic or ultra-relativistic gases. Later, several more accurate EOSs were proposed in the literature for the numerical RHDs, see e.g. \cite{Falle1996,Sokolov2001,Mignone2005,Mathews1971,Ryu2006}.
The existing results suggest that employing a correct EOS is important for getting quantitatively correct results in problems involving a transition from the non-relativistic temperature to the relativistic temperature or vice versa. For the single-component perfect gas in the relativistic regime, the ``exact'' EOS is derived by Synge \cite{synge1957} from the relativistic kinetic theory, which goes back to 1911 when an equilibrium distribution function was derived for a relativistic gas \cite{Juttner1911}. Therefore, it seems convenient and meaningful to construct a gas-kinetic scheme (GKS) for the RHDs with such ``exact'' EOS.
The GKS presents a gas evolution process from a kinetic scale to a
hydrodynamic scale, where the fluxes are recovered from the moments
of a single time-dependent gas distribution function.
The development of GKS, such as the kinetic flux
vector splitting  (KFVS) and Bhatnagar-Gross-Krook (BGK) schemes,
has attracted much attention and significant progress has been
made in the non-relativistic hydrodynamics \cite{Xu2001,LiuTang2014,XuSimplification,May2007ImprovedGKS}.
They utilize the well-known connection that the macroscopic governing equations are the moments of the Boltzmann equation whenever the distribution function is at equilibrium.
The  kinetic  beam  scheme was first proposed for the relativistic gas dynamics in \cite{yang1997kinetic}. After that,
the kinetic schemes including the KFVS and BGK-type scheme for  the ultra-relativistic Euler equations  were developed in  \cite{Kunik2003ultra,kunik2003second,kunik2004bgktype}.
 For special relativistic Euler equations, the kinetic schemes were developed in \cite{Kunik2004,QamarRMHD2005,SRHDKFVS2005}. The BGK method \cite{Xu2001} seems to give good simulations for classical gas dynamics and has been  successfully extended to the ultra-relativistic RHDs \cite{CHEN2017}. Extension of the above method to the special RHDs seems to be feasible.
 However, unlike the ultra-relativistic case, the difficulty and complexity involved in the GKS for the special RHDs are obviously increased due to the Lorentz factor  and the Maxwell-J\"uttner distribution which make the RHD equations highly nonlinear.

 This paper will extend  the BGK scheme of the ultra-RHDs \cite{CHEN2017} to the special RHDs with the Synge EOS. Unfortunately, such scheme seems no longer practical because the triple moment integrals  should be numerically calculated  for numerical fluxes at each time step  with very high computational cost.
 Therefore, the simplification of the BGK scheme is necessary to improve the efficiency of the BGK schemes  while preserving the accuracy.
%
The paper is organized as follows.
 Section \ref{sec:basic} introduces the special relativistic Boltzmann
 equation from the kinetic theory.
 Section \ref{sec:GovernEqns} presents the special-relativistic Euler equations and proves the boundness of the speed of sound.
  Section \ref{sec:scheme}  develops the second-order accurate {gas-kinetic schemes} and their simplified version for the special-relativistic Euler equations.
  Section \ref{sec:moments}
 presents the moment integrals of the approximate distribution appearing in the above gas kinetic schemes.
 Section \ref{sec:test} gives several numerical experiments to demonstrate accuracy, robustness and effectiveness of
 the proposed schemes in simulating special-relativistic fluid flows.
 Section \ref{sec:conclusion} concludes the paper.

\section{Preliminaries and notations}
\label{sec:basic}
A microscopic gas particle in the special relativistic kinetic theory of gases \cite{rbebook} is characterized by the four-dimensional
space-time coordinates $(x^{\alpha}) = (x^0,\vec{x})$ and momentum four-vector $(p^{\alpha}) = (p^0,\vec{p})$, where $x^0 = ct$, $c$,  $t$ and $\vec{x}$ are the speed of light in vacuum, the time and three-dimensional (3D) spatial coordinates, respectively, and the Greek index
$\alpha$  runs from $0, 1, 2, 3$.
Besides the contravariant notation (e.g. $p^{\alpha}$), the covariant notation such as $p_{\alpha}$ will also be used in the following,
while both notations $p^\alpha$ and $p_{\alpha}$ are related by
$p_\alpha = g_{\alpha\beta}p^{\beta}$ and $p^{\alpha} = g^{\alpha\beta}p_{\beta}$,
with 
 the Minkowski space-time metric tensor chosen as
$(g^{\alpha\beta}) = \text{diag}\{1, -1, -1, -1\}$ and its
 inverse $(g_{\alpha\beta})$.
To be more specific, the contravariant components of the momentum four-vector $p^{\alpha}$ are defined by
$ (p^{\alpha}) = m\gamma(\vec{v})(c,\vec{v})$ with
$\gamma(\vec{v}) = (1-c^{-2}|\vec{v}|^2)^{-\frac{1}{2}}$,
where $\vec{v}$ is the particle velocity   and $m$ is the mass of each structure-less particle which is assumed to be the same for all particles. The expression of $p^{\alpha}$ shows that the scalar product of the momentum four-vector with itself is
 $p^{\alpha}p_{\alpha} = m^2c^2$.
The one-particle distribution function $f(x^{\alpha},p^{\alpha})$, defined in terms of the space-time and momentum coordinates, can be taken equal to $f(\vec{x}, t, \vec{p})$ since $p^0 = \sqrt{|\vec{p}|^2+m^2c^2}$. The distribution function is defined as a
scalar invariant such that $f(\vec{x}, t, \vec{p})d^3 \vec{x}d^3 \vec{p}$ gives at time $t$ the number of particles in the volume element $d^3 \vec{x}$ about $\vec{x}$ with momenta in a range $d^3 \vec{p}$ about $\vec{p}$.
The special relativistic Boltzmann equation describes the time evolution of one-particle distribution function and reads
\begin{equation*}
 p^{\alpha}\frac{\partial f}{\partial x^{\alpha}} = Q(f,f),
\end{equation*}
where the collision term $Q(f,f)$   depends on the product of distribution functions of two particles at collision. There exist some simpler collision
models in the literature. The Anderson-Witting (AW) model \cite{AW}
\begin{equation}\label{AWM}
 p^{\alpha}\frac{\partial f}{\partial x^{\alpha}} = -\frac{U_\alpha p^{\alpha}}{\tau c^2}(f-g),
\end{equation}
will be considered in this paper, where $\tau$ is the relaxation time and
 the hydrodynamic four-velocities  $U_{\alpha}$ are defined according to {the} Landau-Lifshitz decomposition, by
\begin{equation}\label{EQ:Landau-Lifshitz frame}
 U_{\beta}T^{\alpha\beta} = \varepsilon g^{\alpha\beta}U_{\alpha},
\end{equation}
which implies that $(\varepsilon, U_{\alpha})$ is a generalized eigenpair of $(T^{\alpha\beta},g^{\alpha\beta})$, here $\varepsilon$ and $ T^{\alpha\beta}$ are
the energy density and  energy-momentum tensor, respectively.
The local-equilibrium distribution $g=g(\vec{x},t,\vec{p})$ in  \eqref{AWM}  is given by
\begin{equation}\label{juttner}
 g=\frac{\rho \zeta}{4\pi m^4c^3K_2(\zeta)}\exp\left(-m^{-1}c^{-2}\zeta U_{\alpha}p^{\alpha}\right),
\end{equation}
which is the so-called Maxwell-J$\ddot{\text{u}}$ttner equilibrium (or relativistic Maxwellian) distribution
and obeys the common prescription
that the mass density $\rho$ and the energy density $\varepsilon$ are completely determined by  $g$ alone, where
$\zeta = \frac{mc^2}{kT}$ is the ratio between the rest energy of a particle $mc^2$ and $kT$, $T$ denotes thermodynamic temperature, $k$ is the Boltzmann's constant, and $K_{{\nu}}(\zeta)$  is the modified Bessel function of the second kind defined by
\[
 K_{{\nu}}(\zeta):=\int_{0}^{\infty}\cosh({\nu}\vartheta)\exp(-\zeta\cosh\vartheta)d\vartheta, \ \ {\nu}\geq 0,
\]
satisfying the recurrence relation
\[
  K_{\nu+1}(\zeta) = K_{\nu-1}(\zeta) + 2\nu\zeta^{-1}K_{\nu}(\zeta).
\]
The particles behave as non-relativistic (resp. ultra-relativistic) for $\zeta\gg1$ (resp. $\zeta \ll1$).
For the collision invariants 1 and $p^{\alpha}$, the collision term in  \eqref{AWM}
satisfies the   identities
\begin{equation}\label{convc}
 \int_{\mathbb{R}^3}\frac{U_\alpha p^{\alpha}}{\tau c^2}(f-g)\vec{\Psi} \frac{d^3\vec{p}}{p^0} = 0,
 \ \
 \vec{\Psi} = (1, p^{i}, p^0)^T,
\end{equation}
which imply the following conservation laws
\begin{equation}\label{conv}
 \partial_{\alpha}N^{\alpha} = 0,\quad \partial_{\beta}T^{\alpha\beta} = 0,
\end{equation}
where
the particle four-flow $N^{\alpha}$ and the energy-momentum tensor $ T^{\alpha\beta}$ can be expressed as
\begin{align}
 N^{\alpha} = c\int_{\mathbb{R}^3}p^{\alpha}f\frac{d^3\vec{p}}{p^0},
 \ \
 \label{intTT}
 T^{\alpha\beta}=c\int_{\mathbb{R}^3}p^{\alpha}p^{\beta}f\frac{d^3\vec{p}}{p^0}.
\end{align}
In the Landau and Lifshitz decomposition they can be decomposed with respect to the four-velocity $U^{\alpha}$ by
\begin{align*}
  N^{\alpha} &= m^{-1}\rho U^{\alpha} + n^{\alpha},\ \
  T^{\alpha\beta} = c^{-2}\varepsilon U^{\alpha}U^{\beta} - \Delta^{\alpha\beta}(p+\varPi) + \pi^{\alpha\beta},
\end{align*}
where 
\begin{equation*}
  \Delta^{\alpha\beta} = g^{\alpha\beta}-c^{-2}U^{\alpha}U^{\beta},
\end{equation*}
which is {a} symmetric  projector onto the 3D subspace  orthogonal to $U_{\alpha}$,
i.e.   $\Delta^{\alpha\beta}U_{\beta} = 0$.
With the help of $N^{\alpha}$ and $T^{\alpha\beta}$,  one can calculate the mass density $\rho$,   the particle-diffusion current $n^\alpha$,
the energy density $\varepsilon$, and the shear-stress tensor {$\pi^{\alpha\beta}$}
of the gas
 by
  \begin{align*}
       \rho=& mc^{-2} U_{\alpha}N^{\alpha} = mc^{-1} \int_{\mathbb{R}^3}Efd\varXi,
       \ \
    n^{\alpha} = \Delta^{\alpha}_{\beta}N^{\beta} = c\int_{\mathbb{R}^3}p^{<\alpha>}fd\varXi,
\\
    \varepsilon =& c^{-2}U_{\alpha}U_{\beta}T^{\alpha\beta} = c^{-1}\int_{\mathbb{R}^3}E^2fd\varXi,
  \ \ 
    \pi^{\alpha\beta} = \Delta^{\alpha\beta}_{\mu\nu}T^{\mu\nu} = c\int_{\mathbb{R}^3}p^{<\alpha\beta>}fd\varXi,
    \end{align*}
and the sum of thermodynamic pressure $p$ and bulk viscous pressure $\varPi$ by
  \begin{equation*}
    p + \varPi = -\frac{1}{3}\Delta_{\alpha\beta}T^{\alpha\beta} = \frac{1}{3c}\int_{\mathbb{R}^3}(E^2-m^2c^4)fd\varXi,
  \end{equation*}
where $d\varXi := \frac{d^3\vec{p}}{p^0}$ is the volume element which is invariant with respect to Lorentz transformations,
$E=U_{\alpha}p^{\alpha}$, $p^{<\alpha>}=\Delta^{{\alpha}}_{\gamma}p^{\gamma}$, $p^{<\alpha\beta>}=\Delta^{\alpha\beta}_{\gamma\delta}p^{\gamma}p^{{\delta}}$, and
  \begin{equation*}
      \Delta^{\alpha\beta}_{\mu\nu}
      =\frac{1}{2}(\Delta^{\alpha}_{\mu}\Delta^{\beta}_{\nu} + \Delta^{\beta}_{\mu}\Delta^{\alpha}_{\nu})
      -\frac{1}{3}\Delta_{\mu\nu}\Delta^{\alpha\beta}.
  \end{equation*}

\begin{remark}
The mass density $\rho$ and energy density $\varepsilon$ are completely determined by the local-equilibrium distribution $g$ alone, i.e.
\begin{equation}\label{rhoepsilon}
\begin{array}{l}{\rho=m c^{-1} \int_{\mathbb{R}^{3}} E g d \Xi}, \\
{\varepsilon=c^{-1} \int_{\mathbb{R}^{3}} E^{2} g d \Xi=\rho c^{2}\left(G(\zeta)-\zeta^{-1}\right)},\ \  G(\zeta)=\frac{K_3(\zeta)}{K_2(\zeta)}. \end{array}
\end{equation}
\end{remark}

\begin{remark}
The quantities $n^{\alpha}, \varPi$, and $\pi^{\alpha\beta}$ become zero at the local thermodynamic equilibrium while $f=g$, i.e.
\begin{equation}\label{npip}
\begin{array} {l}
n^{\alpha} =c \int_{\mathbb{R}^{3}} p^{<\alpha>} g d \Xi = 0, \\
\pi^{\alpha \beta} =c \int_{\mathbb{R}^{3}} p^{<\alpha \beta>} g d \Xi = 0,\\
p=\frac{1}{3 c} \int_{\mathbb{R}^{3}}\left(E^{2}-m^{2} c^{4}\right) g d \Xi =
n k T=\rho c^{2} \zeta^{-1}.\\
\end{array}
\end{equation}
\end{remark}

\begin{remark}
The recovery procedure for the admissible primitive variables $\rho$,
$u$, and $T$ from the nonnegative distribution $f$ or $g$, see
 \cite[Theorem 2.2]{kuang20163D}, will be used in  our {gas-kinetic schemes}.  The readers are also referred to \cite[Appendix A]{kuang20161D}  for its 1D version and a similar result in \cite[Theorem 2.1]{CHEN2017}.
\end{remark}

\section{Special relativistic Euler equations}
\label{sec:GovernEqns}
This section gives the special-relativistic Euler equations using the Maxwell-J$\ddot{\text{u}}$ttner equilibrium distribution $g$. The macroscopic variables $\rho, \varepsilon$ and $p$ are determined by $g$ according to \eqref{rhoepsilon} and \eqref{npip},
while the specific internal energy $e$ and the specific enthalpy $h$ are calculated by
\begin{equation}\label{EQ:EOS}
 e=\rho^{-1}\varepsilon-c^2=c^2(G(\zeta)-\zeta^{-1}-1), \quad h= \rho^{-1}(\varepsilon+p) =c^2 G(\zeta),
\end{equation}
 which is the equation of state for a single-component relativistic perfect gas \cite{synge1957}.
The speed of sound $c_s$ can be obtained from \eqref{EQ:EOS} by
\begin{align}\label{eq:soundspeed}
 \frac{c_s}{c} &= \sqrt{\frac{\frac{p}{\rho^2}-(\frac{\partial e}{\partial \rho})_p}{h(\frac{\partial e}{\partial p})_{\rho}}}
 =\sqrt{\frac{G'(\zeta)/G(\zeta)}{\zeta^{-1}+\zeta G'(\zeta)}}
 = \frac{\sqrt{  5G(\zeta) + \zeta-G(\zeta)^2\zeta}}{\sqrt{G(\zeta)(  5G(\zeta)\zeta + \zeta^2 -1-G(\zeta)^2\zeta^2)}},
 \end{align}
where $G'(\zeta)=G^2-5G\zeta^{-1}-1$. In the ultra-relativistic limit ($\zeta\ll1$),   $G(\zeta)\approx4\zeta^{-1}$ and $c_s\approx c/\sqrt{3}$, which are   consistent with our earlier research \cite{CHEN2017}.
For the sake of convenience,  units in which the speed of light, the mass of each structure-less particle
and the Boltzmann's constant are equal to one will be used hereafter.
\begin{theorem} If the primitive variables $\rho$, $u$, and $T$ are admissible in physics, then the speed of sound $c_s$ in \eqref{eq:soundspeed} satisfies $0<c_s<1$.
\end{theorem}
\begin{proof}
 (i) 
  Let us show the right-hand side of \eqref{eq:soundspeed} is
  well-defined, equivalently,
  prove $\Phi_1(G(\theta^{-1}),\theta):=  5G(\theta^{-1}) + \theta^{-1}-G(\theta^{-1})^2\theta^{-1}>0$ and $\Phi_2(G(\theta^{-1}),\theta):=G(\theta^{-1})( 5G(\theta^{-1})\theta^{-1} + \theta^{-2} -1-G(\theta^{-1})^2\theta^{-2})>0$ for all $\theta := \zeta^{-1}>0$.

The proof of Theorem 2.2 in \cite{kuang20163D} shows
 \begin{align*}
   & 0<\int_{\mathbb{R}^3} \frac{E-1}{E+1}gd\varXi
  = -\rho\theta\left((3\theta+2)G(\theta^{-1})-2(6\theta^2+4\theta+1)\right),            \\
   & 0<\int_{\mathbb{R}^3} (E-1)^2gd\varXi = \rho\left(2G(\theta^{-1})-5\theta-2\right),
 \end{align*}
 and $\frac{5}{2}\theta + 1<\frac{2(6\theta^2+4\theta+1)}{3\theta+2}$ for all $\theta>0$,
 thus for all $\rho,\theta>0$ one has
 \begin{equation*}
 G\in \left(\frac{5}{2}\theta + 1,\frac{2(6\theta^2+4\theta+1)}{3\theta+2}\right)=:{\mathcal I}_G.
 \end{equation*}

The function $\Phi_1$ is a concave quadratic function of $G(\theta^{-1})$ because $\partial^2_G \Phi_1=-\theta^{-1}<0$,
and
 \begin{align*}
   & \Phi_1\left(\frac{5}{2}\theta + 1,\theta\right) = \frac{25\theta}{4} > 0,                                                  \\
   & \Phi_1\left(\frac{2(6\theta^2+4\theta+1)}{3\theta+2},\theta\right) = \frac{\theta(36\theta^2+48\theta+7)}{(3\theta+2)^2}>0,
 \end{align*}
so $\Phi_1(G(\theta^{-1}),\theta)>0$ for all $\theta>0$.

Similarly,   $\Phi_2$ is a {concave} function of $G$ in the interval ${\mathcal I}_G$ because $\partial^2_G \Phi_2=-6\theta^{-2}(G-5\theta/3)<0$, and
 \begin{align*}
   & \Phi_2\left(\frac{5}{2}\theta + 1,\theta\right) = \frac{105\theta}{8} + \frac{21}{4}>0,                                                              \\
   & \Phi_2\left(\frac{2(6\theta^2+4\theta+1)}{3\theta+2},\theta\right) = \frac{324\theta^4 + 648\theta^3 + 378\theta^2 + 96\theta + 6}{(3\theta+2)^3}>0,
 \end{align*}
thus  $\Phi_{2}\left(G\left(\theta^{-1}\right), \theta\right)>0$ for all $\theta >0$.

 (ii) Let us prove   $c_s<1$, 
  which is equivalent to
 \[\Phi_3(G(\theta^{-1}),\theta):=-\theta^{-2}(G(\theta^{-1})^3-6G(\theta^{-1})^2\theta+6G(\theta^{-1})\theta^2-G(\theta^{-1})+\theta)>0.\]
It is easy to show that $\Phi_3$ is also a {concave} function of $G$ in the interval ${\mathcal I}_G$ because $\partial^2_G \Phi_3=-6\theta^{-2}(G-2\theta)<0$, and
%
 \begin{align*}
   & \Phi_3\left(\frac{5}{2}\theta + 1,\theta\right) = \frac{55}{8}\theta + \frac{21}{4}>0,                                                       \\
   & \Phi_3\left(\frac{2(6\theta^2+4\theta+1)}{3\theta+2},\theta\right) = \frac{216\theta^4+432\theta^3+261\theta^2+82\theta+6}{(3\theta+2)^3}>0.
 \end{align*}
Hence, $\Phi_{3}\left(G\left(\theta^{-1}\right), \theta\right)>0$ for all $\theta >0$ and equivalently  $0<c_s<1$.\qed
\end{proof}

At the local thermodynamic equilibrium where $f=g$, one has
\[N^{\alpha} = \rho U^{\alpha}, T^{\alpha\beta}=\rho h U^{\alpha}U^{\beta} - g^{\alpha\beta}p, \]
and the quantities $n^{\alpha}, \varPi$, and $\pi^{\alpha\beta}$ become zero,
so that the special RHD equations \eqref{conv}
can be written into a time-dependent system of conservation laws in the laboratory frame as follows
\begin{equation}\label{eq:Euler}
 \frac{\partial \vec{W}}{\partial t} + \sum^{d}_{k=1}\frac{\partial\vec F^{k}(\vec{W})}{\partial x_k} = 0,
\end{equation}
where
\begin{align*}
 \vec{W} =& \left(N^0, T^{01},\cdots, T^{0d},  T^{00}
 \right)^T
 = \left(
 \rho  U^0,
 \rho  h U^0U^1,\cdots,  \rho  h U^0U^d,
 \rho  h U^0U^0 - p
 \right)^T,
\\
 \vec{F^k(W)} =& \left( N^k,
 T^{k1},\cdots,T^{kd},
 T^{k0}
 \right)^T
 = \left(
 \rho  U^k,
 \rho  h U^kU^1 + p\delta^{k1},\cdots,  \rho  h U^kU^d + p\delta^{kd},
 \rho  h U^kU^0
 \right)^T.
\end{align*}
 Under \eqref{EQ:EOS}, 
the  Jacobian matrix
$\vec{A}^k=\partial \vec{F}^k/\partial \vec{W}$ is 
 diagonalizable with $(d+2)$ real eigenvalues
\begin{align*}
  & \lambda_k^{(1)} = \frac{U^i(1-c_s^2)-c_s\sqrt{1-u^2_i-(|\vec{u}|^2-u^2_i)c^2_s}}
  {U^0(1-|\vec{u}|^2c^2_s)},    \\
  & \lambda_k^{(2)} = \cdots = \lambda_i^{(d+1)} = u_i,                                          \\
  & \lambda_k^{(d+2)} = \frac{U^i(1-c_s^2)+c_s\sqrt{1-u^2_i-(|\vec{u}|^2-u^2_i)c^2_s}}{U^0(1-|\vec{u}|^2c^2_s)},
\end{align*}
where $c_s$ is the local sound speed calculated by \eqref{eq:soundspeed}.
In comparison with the non-relativistic and ultra-relativistic Euler equations, the conservative variables (the particle four-flow $N^0$ and the energy-momentum tensor $T^{0\alpha}$) in \eqref{eq:Euler} are strongly coupled through the Lorentz factor, so that it impossible to obtain the primitive variables (the mass density $\rho$, the fluid velocity $v_i$ and the energy density $\varepsilon$) or the flux $\vec{F^k}$ from the conservative variables
by any explicit form. Thus, in practical computations, the primitive variable
vector $\vec{V}=(\rho,u_1,\cdots, u_d,p)^T$ has to be first recovered from the known conservative vector
$\vec{U} = (N^0,T^{01},\cdots,T^{0d},T^{00})$ at each time step by numerically solving a nonlinear equation for the pressure $p$ such as
\begin{equation}\label{eq:con2pri}
 T^{00} + p = N^0U^0G(\zeta),
\end{equation}
where $U^0 = (1-\sum_{i=1}^d(T^{0i})^2/(T^{00}+p)^2)^{-1/2}$ and $\zeta=\frac{N^0}{U^0 p}$.
Any standard root-finding algorithm, e.g. Newton's iteration, may be used to solve \eqref{eq:con2pri}
to get the pressure $p$, and then $U^0, \rho, \zeta, h$ and $u_1,\cdots,u_d$ in order.

\section{BGK finite volume methods}\label{sec:scheme}
This section  introduces our 2D BGK finite volume method on the rectangular mesh, whose details can be found in \cite{CHEN2017}.
The starting point of such gas kinetic scheme is
the 2D AW model
\begin{align}\label{eq:2DAW}
    p^0\partial_tf + p^1\partial_x f + p^2\partial_y f = \frac{E(g-f)}{p^0\tau},
\end{align}
whose analytical solution can be given by
\begin{align}\label{eq:AWsolution}
\nonumber f(x,y,t,\vec{p})=&\int_{0}^{t}g(x',y', t',\vec{p})\exp\left(-\int_{t'}^{t}
\frac{U_{\alpha}(x'',y'',t'')p^{\alpha}}{p^{0}\tau}dt''\right)
\frac{U_{\alpha}(x',y',t')p^{\alpha}}{p^{0}\tau}dt'\\
&+\exp\left(-\int_{0}^{t}
\frac{U_{\alpha}(x',y',t')p^{\alpha}}{\tau p^{0}}dt'\right)f_{0}(x-v_1t,y-v_2t,\vec{p}),
\end{align}
where $v_1=p^1/p^0$ and $v_2=p^2/p^0$ are the particle velocities in $x$  and $y$ directions respectively, $\{x'=x-v_1(t-t'), y'=y-v_2(t-t')\}$  and  $\{x''=x-v_1(t-t''), y''=y-v_2(t-t'')\}$ are the particle trajectories,
and $f_{0}(x,y,\vec{p})$ is the initial particle velocity distribution function, i.e. $f(x,y,0,\vec{p})=f_{0}(x,y,\vec{p})$.


Divide the spatial domain $\Omega$ into a rectangular mesh with the cell $I_{i,j}=\{(x,y)|x_{i-\frac{1}{2}}<x<x_{i+\frac{1}{2}}, y_{j-\frac{1}{2}}<y<y_{j+\frac{1}{2}}\}$, where
$x_{i+\frac{1}{2}} = \frac{1}{2}(x_{i}+x_{i+1}), y_{j+\frac{1}{2}}=\frac{1}{2}(y_{j}+y_{j+1})$,  $x_i=i\Delta x, y_j = j\Delta y$, and $i,j\in\mathbb{Z}$. The time interval $[0,T]$
is also partitioned into a (non-uniform) mesh ${t_{n+1}=t_n+\Delta t_n, t_0=0, n\geqslant0}$, where  the time step size $\Delta t_n$ is determined by
\begin{equation}\label{TimeStep2D}
  \Delta t_n = \frac{C \min\{\Delta x, \Delta y\}}{\max\limits_{ij}\{\bar{\varrho}^1_{i,j},\bar{\varrho}^2_{i,j}\}},
\end{equation}
 here   $C$  and $\bar{\varrho}^k_{i,j}$ denote  the CFL number and the  approximation of the spectral radius of $\vec A^k(\vec{W})$ over the cell  $I_{i,j}$, $k=1,2$, at time $t_n$, respectively.

Taking the moments of \eqref{eq:2DAW}
and integrating them over the time-space control volume $I_{i,j}\times[t_n,t_{n+1})$ {yield} the 2D finite volume
scheme
\begin{equation}
\label{eq:2DEulerDisc}
  \bar{\vec{W}}^{n+1}_{i,j} = \vec{\bar{W}}^{n}_{i,j} - \frac{\Delta t_n}{\Delta x}(\hat{\vec{F}}^{1,n}_{i+\frac{1}{2},j} - \hat{\vec{F}}^{1,n}_{i-\frac{1}{2},j}) - \frac{\Delta t_n}{\Delta y}(\hat{\vec{F}}^{2,n}_{i,j+\frac{1}{2}} - \hat{\vec{F}}^{2,n}_{i,j-\frac{1}{2}}),
\end{equation}
where $\vec{\bar{W}}^{n}_{i,j}$ is the cell average approximation of   conservative vector $\vec{W}(x,y,t)$ over the cell $I_{i,j}$ at time $t_n$, i.e.
\begin{equation*}
  \vec{\bar{W}}^{n}_{i,j} \approx \frac{1}{\Delta x{\Delta y}}\int_{I_{i,j}}\vec{W}(x,y,t_n)dx{dy},
\end{equation*}
and $\hat{\vec{F}}^{1,n}_{i+\frac{1}{2},j}$ and $\hat{\vec{F}}^{2,n}_{i,j+\frac{1}{2}}$ approximate the fluxes along the interface of the cell $I_{i,j}$ as
\begin{align*}
  \hat{\vec{F}}^{1,n}_{i+\frac{1}{2},j}&\approx\frac{1}{\Delta t_n\Delta y}\int_{t_n}^{t_{n+1}}\int_{y_{j-\frac{1}{2}}}^{y_{j+\frac{1}{2}}}\vec{F}^{1}(x_{i+\frac{1}{2}},y,t) dy{dt},\\
  \hat{\vec{F}}^{2,n}_{i,j+\frac{1}{2}}&\approx\frac{1}{\Delta t_n\Delta x}\int_{t_n}^{t_{n+1}}\int_{x_{i-\frac{1}{2}}}^{x_{i+\frac{1}{2}}}\vec{F}^{2}(x,y_{j+\frac{1}{2}},t) dx{dt}.
\end{align*}
In the BGK  scheme, the numerical fluxes will be obtained by expanding $f$ in $x,y$ and $t$  and then using the conservation constraints \eqref{convc} and the moments \eqref{intTT} so that the surface and time integral are  analytically calculated.
 The fundamental task is thus to construct {\color{red}an} approximate distribution $\hat{f}$ on the cell interface which can be obtained from the analytical solution \eqref{eq:AWsolution} depending on the approximated initial distribution $f_{h,0}$ and the equilibrium $g_h$.

Let us  focus on the derivation of the numerical flux $\hat{\vec{F}}^{1,n}_{i+\frac{1}{2},j}$, 
because the numerical flux in $y$-direction can be similarly derived.
For our second order method,
\begin{equation}\label{eq:2DF}
  \hat{\vec{F}}^{1,n}_{i+\frac{1}{2},j}=\frac{1}{\Delta t_n}\int_{t_n}^{t_{n+1}}\int_{\mathbb{R}}p^1\hat{f}(x_{i+\frac{1}{2}},y_j,t,\vec{p})d\varXi dt\approx\frac{1}{\Delta t_n}\int_{t_n}^{t_{n+1}}\vec{F}^{1}(x_{i+\frac{1}{2}},y_j,t){dt},
\end{equation}
where $\hat{f}(x_{i+\frac{1}{2}},y_j,t,\vec{p})\approx f(x_{i+\frac{1}{2}},y_j,t,\vec{p})$ is derived  with the help of \eqref{eq:AWsolution} as follows
\begin{align}\label{eq:2DAWsoluapprox}
\nonumber \hat{f}(x_{i+\frac{1}{2}},y_j,t,\vec{p})=&\int_{t_n}^{t}g_h(x',y', t',\vec{p})\exp\left(-\int_{t'}^{t}
\frac{U_{\alpha}(x'',y'',t'')p^{\alpha}}{p^{0}\tau}dt''\right)
\frac{U_{\alpha}(x',y',t')p^{\alpha}}{p^{0}\tau}dt'\\
&+\exp\left(-\int_{t_n}^{t}
\frac{U_{\alpha}(x',y',t')p^{\alpha}}{\tau p^{0}}dt'\right)f_{h,0}(x_{i+\frac{1}{2}}-v_1\tilde{t},y_j-v_2\tilde{t},\vec{p}),
\end{align}
here $\tilde{t}=t-t_n$, $x'=x_{i+\frac{1}{2}}-v_1(t-t'), y'=y_j-v_2(t-t')$ and $x''=x_{i+\frac{1}{2}}-v_1(t-t''), y''=y_j-v_2(t-t'')$, $f_{h,0}(x_{i+\frac{1}{2}}-v_1\tilde{t},y_j-v_2\tilde{t},\vec{p})$ and $g_h(x',y',t',\vec{p})$ are (approximate)   initial   distribution function and equilibrium distribution function, respectively, which will be presented in the following with the ``simplified'' notations $(x_{i+\frac12},y_j)=(0,0)$ and $t_n=0$.
 In order to avoid getting $U_{\alpha}(x'',y'',t'')$ or $U_{\alpha}(x',y',t')$ along the particle trajectory, they may be taken as a constant $U_{\alpha,{i+\frac{1}{2},j}}^n$ instead.

\subsection{Calculation of $f_{h,0}$}\label{section-4.1}
For a nonequilibrium distribution $f$, the first-order Chapman-Enskog expansion of the AW model is
\begin{small}
\begin{equation*}
  f(x,y,t,\vec{p})=g-\frac{\tau}{E}\left(p^0{g_t+p^1g_x+p^2g_y}\right)+{\mathcal O}(\tau^2)
  =:g\left(1-\frac{\tau}{E}\left(p^0A+p^1a+p^2b\right)\right)+{\mathcal O}(\tau^2),
\end{equation*}
\end{small}
where the slopes $a,b$ and $A$  are related to $g$ by
\begin{equation}\label{eq:coefabA}
  a = g_x/g,\quad b = g_y/g, \quad A = g_t/g,
\end{equation}
which 
have a unique correspondence with the slopes of the conservative variables. The term $-\tau E^{-1}(p^0A+p^1a+p^2b)$ accounts for the deviation of a distribution function away from relativistic Maxwellian distribution. Using the conservation constraints \eqref{convc} gives
\begin{equation}
\label{eq:aAcons2D}
\int_{\mathbb{R}^{3}}\vec{\Psi}(p^{0}A+p^{1}a+p^2b)gd\varXi=\int_{\mathbb{R}^{3}}\vec{\Psi}(p^{0}g_{t}+p^{1}g_{x}+p^2g_y)d\varXi
=\frac{1}{\tau}\int_{\mathbb{R}^{3}}\vec{\Psi}E(g-f)d\varXi=0.
\end{equation}

The initial state $f_{h,0}(0,0)$ is assumed to be a nonequilibrium and discontinuous at the cell interface as follows
\begin{equation*}
    f_{h,0}(0,0)=\begin{cases}
        g^l\left(1-\frac{\tau}{E}\left(p^0A^l+p^1a^l+p^2b^l\right)\right),&x>0,\\
        g^r\left(1-\frac{\tau}{E}\left(p^0A^r+p^1a^r+p^2b^r\right)\right),&x<0,
    \end{cases}
\end{equation*}
where the distributions $g^l$ and $g^r$ are the left and right limits of the
Maxwell-J\"uttner distribution at the cell interface and can be obtained
by the reconstructed conservative variables, and
the coefficients $(a^l,b^l,A^l,a^r,b^r,A^r)$ depend on the momentum four-vector $p^{\alpha}$ and the conservative variables as follows
\[\Lambda = \Lambda_3 + \Lambda_1p^1 + \Lambda_2p^2 + \Lambda_0p^0, \quad \Lambda=a^l,b^l,A^l,a^r,b^r,A^r.\]
Using the Taylor series at the point $(0,0)$ 
further gives 
\begin{equation}\label{eq:initialstate}
    f_{h,0}(x,y)=\begin{cases}
        g^l\left(1+a^lx +b^ly-\frac{\tau}{E}\left(p^0A^l+p^1a^l+p^2b^l\right)\right),&x>0,\\
        g^r\left(1+a^rx+b^ry-\frac{\tau}{E}\left(p^0A^r+p^1a^r+p^2b^r\right)
        \right), &x<0.
    \end{cases}
\end{equation}

At $t=t_n$, we may reconstruct the piecewisely (discontinuous) linear polynomial $\vec{W}_h(x,y)$ in the $x$-direction by their cell averages $\{\vec{\bar{W}}^{n}_{i,j}\}$, e.g.
\begin{equation*}
   \vec{W}_h(x,y)=\vec{\bar{W}}^{n}_{i,j}+(\vec{\bar{W}}_x)^{n}_{i,j}(x-x_i)
   +(\vec{\bar{W}}_y)^{n}_{i,j}(y-y_j),\
  (x,y)\in I_{i,j}.
\end{equation*}
 Denote the left and right hand limits of $\vec{W}_h(x,y)$ and their partial derivatives in the $x,y$-directions at the interface center $(x_{i+\frac12},y_j)$ by $\vec W^l$, $\vec W^r$, $\vec W^l_{x}, \vec W^r_{x}$, $\vec W^l_{y}$, and  $\vec W^r_{y}$, respectively. Using the relations between the gas distribution function $f$ and the macroscopic variables yields the linear systems for $a^\omega$ and $b^\omega$ as
\begin{align*}
<a^{\omega}{p^0}>=\vec{W}^{\omega}_x,\ \
<b^{\omega}{p^0}>=\vec{W}^{\omega}_y, \ \omega=l,r,
\end{align*}
where
\[
{ <a^{\omega}>:=\int_{\mathbb{R}^{3}}a^{\omega}g^{\omega} \vec{\Psi}d\varXi},\ \ \omega=l,r.
\]
Their matrix forms
are
\begin{align*}
    M_{0}^{\omega}{\vec{a}}^\omega=\vec{W}^{\omega}_{x},\quad
    M_{0}^{\omega}{\vec{b}}^\omega=\vec{W}^{\omega}_{y},
\end{align*}
where $\vec{a}^\omega:=(a^{\omega}_0, a^{\omega}_1, a^{\omega}_2, a^{\omega}_3)^T$, $\vec{b}^\omega:=(b^{\omega}_0, b^{\omega}_1, b^{\omega}_2, b^{\omega}_3)^T$, and  the coefficient matrix is defined by
\[
  M_{0}^{\omega}=\int_{\mathbb{R}^{3}}p^{0}g^{\omega}\vec{\Psi}\vec{\Psi}^Td\varXi,
\]
 presented in Section \ref{sec:moments} in detail.
After having the values of $a^l,b^l$ and $a^r,b^r$, substituting them into the conservation constraints \eqref{eq:aAcons2D} gives the linear system for $A^\omega$  as
\[
<a^{\omega}p^{1}+b^{\omega}p^{2}+A^{\omega}p^{0}>=0,\ \omega=l,r,
\]
or
\begin{equation}\label{eq:lA2D}
M_{0}^{\omega}\vec{A}^{\omega}=-M_{1}^{\omega}\vec{a}^{\omega}-M_{2}^{\omega}\vec{b}^{\omega},
\end{equation}
with
\begin{align*}
M_{1}^{\omega}=\int_{\mathbb{R}^{3}}p^{1}g^{\omega}\vec{\Psi}\vec{\Psi}^Td\varXi,\ \
M_{2}^{\omega}=\int_{\mathbb{R}^{3}}p^{2}g^{\omega}\vec{\Psi}\vec{\Psi}^Td\varXi, \ \omega=l,r.
\end{align*}
All  elements of the  matrices $M_1^{\omega}$ and $M_2^{\omega}$ can be explicitly calculated by using some coordinate transformation, see Section \ref{sec:moments}.  Up to now, the coefficients $(a^l,b^l,A^l,a^r,b^r,A^r)$ in \eqref{eq:initialstate} have been calculated from the reconstructed derivatives $\vec W^l_{x}, \vec W^r_{x}$, $\vec W^l_{y}$, and  $\vec W^r_{y}$ so that the initial distribution $f_{h,0}$ is determined.

\subsection{Calculation of $g_h$}
The equilibrium distribution $g$ in the neighbourhood of $(x_{i+\frac12},y_j,t_n)=(0,0,0)$ is approximated by
\begin{equation}\label{eq:TaylorEqui}
    g_h(x,y,t) = g^0(1+a^0x+b^0y+A^0t),
\end{equation}
where $g^0$ is a local Maxwell-J$\ddot{\text{u}}$ttner equilibrium
located at $(x_{i+\frac12},y_j)=(0,0)$ and $a^0,b^0,A^0$ are related to the space and time derivatives of $g$ at the point $(0,0)$, see \eqref{eq:coefabA}.
For the ideal gases, the distributions at both sides of a cell interface are the Maxwell-J$\ddot{\text{u}}$ttner equilibrium, thus we can determine the particle four-flow $N^{\alpha}$ and the energy-momentum tensor $T^{\alpha\beta}$ at the cell interface by
\begin{align*}
   &N^{\alpha}_0 = \int_{\mathbb{R}^{3}}p^{\alpha}g^0d\varXi= \int_{\mathbb{R}^{3}\cap{p^{1}>0}} p^{\alpha}g^ld\varXi + \int_{\mathbb{R}^{3}\cap{p^{1}<0}} p^{\alpha}g^rd\varXi, \\
   &T^{\alpha,\beta}_0 = \int_{\mathbb{R}^{3}}p^{\alpha}p^{\beta}g^0d\varXi= \int_{\mathbb{R}^{3}\cap{p^{1}>0}} p^{\alpha}p^{\beta}g^ld\varXi + \int_{\mathbb{R}^{3}\cap{p^{1}<0}} p^{\alpha}p^{\beta}g^rd\varXi.
\end{align*}
Using those and Theorem \ref{thm:NT} calculates the macroscopic quantities $\rho^0, T^0$ and $U_{\alpha}^0$,
and then gives the Maxwell-J\"uttner distribution function by
\[ g^0=\frac{\rho^0 \zeta^0}{4\pi K_2(\zeta^0)}\exp\left(\zeta^0 U_{\alpha}^0p^{\alpha}\right),\]
which reflects the modeling of the collision process leading to the equilibrium.
Using the cell interface values $\{\vec{W}^0_{i+\frac12,j}\}$ 
reconstructs the following approximate derivatives at the point $(x_{i+\frac12},y_j)=(0,0)$
\begin{align*}
    \vec{W}_{x}^0 = \frac{\bar{\vec{W}}_{i+1,j}-\bar{\vec{W}}_{i,j}}{\Delta x},\quad  \vec{W}^0_{y} = \frac{\vec{W}^0_{i+\frac12,j+1}-\vec{W}^0_{i+\frac12,j-1}}{2\Delta y},
\end{align*}
and then the coefficients in  \eqref{eq:TaylorEqui} are determined by solving the   linear systems for $a^0, b^0$ and $A^0$
\[
<a^{0}{p^0}>=\vec{W}^0_{x},\quad <b^{0}{p^0}>=\vec{W}^{0}_{y},\quad <A^{0}p^{0}+a^{0}p^{1}+b^{0}p^{2}>=0,
\]
or
\begin{equation*}
M_{0}^{0}\vec{a}^{0}=\vec{W}^{0}_{x},\quad M_{0}^{0}\vec{b}^{0}=\vec{W}^{0}_{y}, \quad M_{0}^{0}\vec{A}^{0}= -M_{1}^{0}\vec{a}^{0} -M_{2}^{0}\vec{b}^{0},
\end{equation*}
 where the elements of $M_0^0, M_1^0$ and $M_2^0$ will be given later, see Section \ref{sec:moments}.

\subsection{Derivation of $\hat{f}$ and its simplification}
Up to now, all of the coefficients in the initial gas distribution function $f_{h,0}$ and the equilibrium state $g_h$ have been given at $t^n=0$. Substituting \eqref{eq:initialstate} and \eqref{eq:TaylorEqui} into \eqref{eq:2DAWsoluapprox}
 gives
\begin{small}
\begin{align}\label{eq:finalstate1}
 &\hat{f}(x_{i+\frac{1}{2}},y_j,t,\vec{p})
 =g^0\left(1-\exp\left(-\frac{U_{\alpha}^0p^{\alpha}}{p^0\tau}t\right)\right)
+g^0a^0v_1\left(\left(t+\frac{p^0\tau}{U_{\alpha}^0p^{\alpha}}\right)\exp\left(-\frac{U_{\alpha}^0p^{\alpha}}{p^0\tau}t\right)-\frac{p^0\tau}{U_{\alpha}^0p^{\alpha}}\right)\notag\\
&+g^0b^0v_2\left(\left(t+\frac{p^0\tau}{U_{\alpha}^0p^{\alpha}}\right)\exp\left(-\frac{U_{\alpha}^0p^{\alpha}}{p^0\tau}t\right)-\frac{p^0\tau}{U_{\alpha}^0p^{\alpha}}\right)
+g^0A^0\left(t-\frac{p^0\tau}{U_{\alpha}^0p^{\alpha}}\left(1-\exp\left(-\frac{U_{\alpha}^0p^{\alpha}}{p^0\tau}t\right)\right)\right)\notag\\
&+H[v_1]g^l\left(1-\frac{\tau}{U_{\alpha}^lp^{\alpha}}(p^0A^l+p^1a^l+p^2b^l)-a^lv_1t-b^lv_2t\right)\exp\left(-\frac{U_{\alpha}^0p^{\alpha}}{p^0\tau}t\right)\notag\\
&+(1-H[v_1])g^r\left(1-\frac{\tau}{U_{\alpha}^rp^{\alpha}}(p^0A^r+p^1a^r+p^2a^r)-a^rv_1t-b^rv_2t\right)\exp\left(-\frac{U_{\alpha}^0p^{\alpha}}{p^0\tau}t\right),
\end{align}
\end{small}%
where $H(x)$ is the Heaviside function with $H(x)=1$ for $x\leq0$ and $H(x)=0$ otherwise.
Combining \eqref{eq:finalstate1} with \eqref{eq:2DF} yields the numerical flux $\hat{\vec{F}}^{1,n}_{i+\frac{1}{2}}$, while the numerical flux $\hat{\vec{F}}^{2,n}_{j+\frac{1}{2}}$ can be obtained in a similar procedure.

Rearranging  the terms in \eqref{eq:finalstate1} gives
\begin{align}\label{eq:finalstate}
 \hat{f}(x_{i+\frac{1}{2}},&y_j,t,\vec{p})
 =g^0-\frac{\tau}{U_{\alpha}^0p^{\alpha}}\left(p^1g^0_x+p^2g^0_y+p^0g^0_t\right) + g^0_tt\notag\\
 &- \exp\left(-\frac{{U}_{\alpha}^0p^{\alpha}}{p^0\tau} t \right) \left({g}^0-\frac{\tau}{{U}_{\alpha}^0p^{\alpha}}(p^0{g}^0_t+p^1{g}^0_x+p^2g^0_y)-{g}^0_xv_1t - {g}^0_yv_2t\right)\notag\\
 &+ \exp\left(-\frac{U^0_{\alpha}p^{\alpha}}{p^0\tau} t \right) \left(\hat{g}-\frac{\tau}{\hat{U}_{\alpha}p^{\alpha}}(p^0\hat{g}_t+p^1\hat{g}_x+p^2\hat{g}_y)-\hat{g}_xv_1t - \hat{g}_yv_2t\right),
\end{align}
where
    $\hat{g}=g^lH(v_1) + g^r(1-H(v_1))$, $
    \hat{g_x} = g^l_xH(v_1) + g^r_x(1-H(v_1))$,
and the notations for the time derivative and space derivative in $y$-direction are similar. The term $-\frac{\tau}{U_{\alpha}^0p^{\alpha}}\left(p^0g^0_x+p^1g^0_y+p^0g^0_t\right)$ is exactly the nonequilibrium state derived from the Chapman-Enskog expansion of the AW model, while $g_t^0t$ is the time evolution part of the gas distribution function. In addition, the equilibrium terms with the exponential factor has an expression similar to the initial state, but with opposite signs.

For the inviscid fluid flows,   the particles are always in equilibrium, i.e. the particle collision time $\tau=0$ so that the dominant part in $\hat{f}(x_{i+\frac{1}{2}},y_j,t,\vec{p})$ is $g^0+g^0_tt$ since the nonequilibrium state  (depending on $\tau$) and the exponential terms disappear. The term $g^0+g^0_tt$ gives the second order approximation in time for the distribution on cell interface, while the second order in space is accomplished by the reconstruction.
It is well known that the width of the shock wave is proportional to the mean free path of the particle,  which is the product of the average collision time $\tau$ between the particles and the average velocity \cite{XuBook}. For the invisid flow, the particle collision time $\tau = 0$ leads to the width of shock is also zero, thus the solution may have a discontinuity. However, the shock structure cannot be resolved exactly with the limited mesh cells for the gas kinetic scheme. Consequently, the shock thickness is enlarged to the mesh size thickness from
the mean free path scale. In practice, there is not a unique theory for the construction of numerical collision time $\tau_n$, here we use the formula in \cite{LiuTang2014}
\begin{equation*}
\tau_n=C_{1}\Delta t^{\alpha_1}_n+C_{2}\Delta t^{\alpha_2}_{n}\frac{|p^{l}-p^{r}|}{p^{l}+p^{r}},
\end{equation*}
where $p^l$ and $p^r$ are the left and right-hand limits of the pressure at the cell interface, respectively, $C_1$, $C_2$ and $\alpha_1, \alpha_2$ are constants.
In smooth region, $\tau_n$ is small since the left and right-hand limits $\vec{W}^l$ and $\vec{W}^r$ are approximately equal, so that the dominant part in   $\hat{f}(x_{i+\frac{1}{2}}, y_j,t,\vec{p})$
is $g^0+g^0_t\tilde{t}$. However, the pressure jump in the shock structure causes an increase in $\tau_n$, which is equivalent to an increase in the shock width, thereby suppressing the numerical oscillation. From the numerical results in \cite{LiuTang2014}, we know that $\tau_n$ will affect the accuracy of the gas kinetic scheme. The  convergence rate can reach the theoretical value when the value of $\alpha_1, \alpha_2$ matches the order of the scheme. 

The above BGK scheme with \eqref{eq:finalstate} or \eqref{eq:finalstate1}
is very time-consuming thanks to the quadratures for the many three-dimensional  moment integrals of the non-equilibrium part in the approximate distribution $\hat{f}$, so that it is no longer practical. In view of this, it
is necessary to present a simplified/ecomonic BGK scheme.
To do that,  let us first investigate the contribution of the terms in the distribution $\hat{f}(x_{i+\frac{1}{2}},y_j,t,\vec{p})$ in \eqref{eq:finalstate} to the numerical flux $\hat{\vec{F}}^{1,n}_{i+\frac{1}{2},j}$.
If assuming that the primitive variables are linearly reconstructed  in the cell $I_{i,j}$ as
\begin{equation}\label{eq:reconstruct_poly}
  \vec{p}_{i,j}(x,y)=\bar{\vec{V}}_{i,j}+\frac{\bar{\vec{V}}_{i+1,j}-\bar{\vec{V}}_{i,j}}{\Delta x} (x-x_i)
  + \frac{\bar{\vec{V}}_{i,j+1}-\bar{\vec{V}}_{i,j}}{\Delta y}(y-y_j),
\end{equation}
then the left hand limit of $\vec{V}$ on the cell interface
is $\vec{V}_{i+1/2,j,L}=\vec{p}_{i,j}(x_{i+\frac12},y_j)$,
and
\begin{equation}\label{eq:goal}
  \vec{V}_{i+1/2,j,L}-\vec{V}(x_{i+\frac12},y_j)=\bar{\vec{V}}_{i,j}+\frac{\bar{\vec{V}}_{i+1,j}-\bar{\vec{V}}_{i,j}}{2} - \vec{V}(x_{i+\frac12},y_j).
\end{equation}
To simplify the notation, here and hereafter, $y_j$ and subscript $j$ will be omitted as
$\vec{V}(x_{i+\frac12}):=\vec{V}(x_{i+\frac12},y_j), \vec{V}_{i+1/2,L}:=\vec{V}_{i+1/2,j,L}$.
Denote the primitive function of $\vec{V}(x,y_j)$   by 
$
  \tilde{\vec{V}}(x)=\int^x_{x_0}\vec{V}(\xi)d\xi$,
where  $x_0$ is not important in our discussion, and
calculate its point value as 
  $\tilde{\vec{V}}(x_{i+\frac12})=\sum_{k=-i_0}^i\bar{\vec{V}}_k\Delta x$.
Using the point values of $\tilde{\vec{V}}(x)$ at $x_{i-\frac12},x_{i+\frac12},x_{i+\frac32}$ interpolate uniquely a quadratic polynomial $\vec{P}(x)$ as 
\begin{equation}\label{eq:lagrange_Vtilde}
  \vec{P}(x)=\sum^2_{m=0}\tilde{\vec{V}}(x_{i+m-\frac12})\prod_{\substack{l=0\\l\neq m}}^2\frac{x-x_{i+l-\frac12}}{x_{i+m-\frac12}-x_{i+l-\frac12}}.
\end{equation}
%
Subtracting $\tilde{\vec{V}}(x_{i-\frac12})$ from both ends of   \eqref{eq:lagrange_Vtilde} and using the following equation
\[ \sum^2_{m=0}\prod_{\substack{l=0\\l\neq m}}^2\frac{x-x_{i+l-\frac12}}{x_{i+m-\frac12}-x_{i+l-\frac12}}=1, \]
gives
\begin{equation*}
  \vec{P}(x)-\tilde{\vec{V}}(x_{i-\frac12})=\bar{\vec{V}}_i\frac{(x-x_{i-\frac12})(x-x_{i+\frac32})}{-\Delta x}
  +(\bar{\vec{V}}_i+\bar{\vec{V}}_{i+1})\frac{(x-x_{i-\frac12})(x-x_{i+\frac12})}{2\Delta x}.
\end{equation*}
Taking the derivative of the above equation gives
the reconstructed polynomial \eqref{eq:reconstruct_poly} at $y=y_j$.
It means that the
interpolation polynomial $\vec{P}(x)$ of $\tilde{\vec{V}}(x)$ can be used to analyze \eqref{eq:goal}.
Again taking the derivative of \eqref{eq:lagrange_Vtilde} gives
\begin{equation}\label{eq:Pprime}
  \vec{p}(x) = \sum^2_{m=0}\tilde{\vec{V}}(x_{i+m-\frac12}) \cdot 
  \Big({\sum\limits_{\substack{l=0\\l\neq m}}^2 \prod\limits_{\substack{q=0\\q\neq l,m}}^2 x-x_{i+q-\frac12}}\Big)
  \Big({\prod\limits_{\substack{l=0\\l\neq m}}^2x_{i+m-\frac12}-x_{i+l-\frac12}}\Big)^{-1}.
\end{equation}
Using the Taylor expansion of $\tilde{\vec{V}}(x_{i+\frac12\pm 1})$
at $x_{i+\frac12}$
\begin{align*}
  \tilde{\vec{V}}(x_{i+\frac12\pm 1}) = \tilde{\vec{V}}(x_{i+\frac12}) \pm \tilde{\vec{V}}'(x_{i+\frac12})\Delta x
  + \frac12\tilde{\vec{V}}''(x_{i+\frac12})\Delta x^2\pm \frac16\tilde{\vec{V}}'''(x_{i+\frac12})\Delta x^3 + {\mathcal O}(\Delta x^4),
\end{align*}
yields
  \begin{align*}
      \vec{p}(x_{i+\frac12}) =& \frac{1}{2\Delta x}(\tilde{\vec{V}}(x_{i+\frac32}) - \tilde{\vec{V}}(x_{i-\frac12}))
                             = \vec{V}(x_{i+\frac12}) + \frac16 \tilde{\vec{V}}'''(x_{i+\frac12})\Delta x^2 + {\mathcal O}(\Delta x^3).
  \end{align*}
Thus the right hand side of \eqref{eq:goal} is equal to
that $\frac16 \tilde{\vec{V}}'''(x_{i+\frac12})\Delta x^2 + {\mathcal O}(\Delta x^3)$.
Similarly, one has $\vec{V}_{i+1/2,R}-\vec{V}(x_{i+\frac12}) = -\frac13 \tilde{\vec{V}}'''(x_{i+\frac12})\Delta x^2 + {\mathcal O}(\Delta x^3)$  by using the interpolation polynomial in the cell $I_{i+1,j}$.

The above discussion will be used to investigate the difference quotient
\begin{equation}\label{eq:fluxdiff}
  \frac{1}{\Delta x}(\hat{f}(x_{i+1/2},y_j,t,\vec{p})-\hat{f}(x_{i-1/2},y_j,t,\vec{p})),
\end{equation}
with the distribution function $\hat{f}$ in \eqref{eq:finalstate}.
%
According to the calculation of $g^0(\vec V_L, \vec V_R)$,
 one has $g^0(\vec{V},\vec{V})=g(\vec{V})$ and $(\nabla_{\vec{V}_L}g^0 + \nabla_{\vec{V}_R}g^0)(\vec{V},\vec{V})=(\nabla_{\vec{V}}g)(\vec{V})$
 when $\vec{V}_L=\vec{V}_R=\vec{V}$. Thus the first part of the expansion of \eqref{eq:fluxdiff} is
  \begin{align*}
     &g^0(\vec{V}_{L},\vec{V}_{R})|_{i+\frac12}-g^0(\vec{V}_{L},
     \vec{V}_{R})|_{i-\frac12}\nonumber\\
    =&g_{i+\frac12} + (\nabla_{\vec{V_L}}g)(\vec{V}_{i+\frac12})
    (\vec{V}_{i+\frac12,L}-\vec{V}_{i+\frac12})
    +(\nabla_{\vec{V_R}}g)(\vec{V}_{i+\frac12})(\vec{V}_{i+1/2,R}-\vec{V}_{i+\frac12}) \nonumber\\
     & - g_{i-\frac12} - (\nabla_{\vec{V_L}}g)(\vec{V}_{i-\frac12})( \vec{V}_{i-\frac12,L}-\vec{V}_{i-\frac12}) -(\nabla_{\vec{V_R}}g)(\vec{V}_{i-\frac12})(\vec{V}_{i-1/2,R}
     -\vec{V}_{i-\frac12}) +{\mathcal O}(\Delta x^4),
  \end{align*}
where $\vec{V}_{i\pm\frac12}=\vec{V}(x_{i\pm\frac12}), g_{i\pm\frac12}=g(\vec{V}_{i\pm\frac12})$ and the Taylor expansion has been used. It is easy to check that
$g_{i+\frac12}-g_{i-\frac12}=g_{x,i}\Delta x + {\mathcal O}(\Delta x^3)$, while
\begin{align}\label{eq:g0diff}
    & (\nabla_{\vec{V_L}}g)(\vec{V}_{i+\frac12})(\vec{V}_{i+\frac12,L}
    -\vec{V}_{i+\frac12}) - (\nabla_{\vec{V_L}}g)(\vec{V}_{i-\frac12})( \vec{V}_{i-\frac12,L}-\vec{V}_{i-\frac12} )
    \nonumber\\
  =&((\nabla_{\vec{V_L}}g)(\vec{V}_{i+\frac12}) - (\nabla_{\vec{V_L}}g)(\vec{V}_{i-\frac12}))(\vec{V}_{i+\frac12,L}
  -\vec{V}_{i+\frac12})\nonumber\\
    & + (\nabla_{\vec{V_L}}g)(\vec{V}_{i-\frac12})
    (\vec{V}_{i+\frac12,L}-\vec{V}_{i+\frac12}-  \vec{V}_{i-\frac12,L}+\vec{V}_{i-\frac12})={\mathcal O}(\Delta x^3),
\end{align}
 since $(\nabla_{\vec{V_L}}g)(\vec{V}_{i+\frac12}) - (\nabla_{\vec{V_L}}g)(\vec{V}_{i-\frac12})=(\nabla_{\vec{V_L}}g)_{x,i}\Delta x+ {\mathcal O}(\Delta x^3)$ and
$  \vec{V}_{i+\frac12,L}-\vec{V}(x_{i+\frac12})-  \vec{V}_{i-\frac12,L}+\vec{V}(x_{i-\frac12})
  = \frac16 \tilde{\vec{V}}'''(x_{i+\frac12})\Delta x^2
   - \frac16 \tilde{\vec{V}}'''(x_{i-\frac12})\Delta x^2 + {\mathcal O}(\Delta x^3)
  = {\mathcal O}(\Delta x^3).
$
Similarly, $\vec{V}_{i+\frac12,R}-\vec{V}(x_{i+\frac12})-  \vec{V}_{i-\frac12,R}+\vec{V}(x_{i-\frac12})={\mathcal O}(\Delta x^3)$.
Therefore, the first part of the expansion of \eqref{eq:fluxdiff} is
\begin{equation}
  g^0(\vec{V}_{L},\vec{V}_{R})|_{i+\frac12}-g^0(\vec{V}_{L},\vec{V}_{R})|_{i-\frac12} = g_{x,i}\Delta x + {\mathcal O}(\Delta x^3).
\end{equation}
Similarly, the second part of the expansion for \eqref{eq:fluxdiff}, $(g^0_{t,i+\frac12}-g^0_{t,i-\frac12})\tilde{t}$, is  equal to $(g_{tx,i}\Delta x + {\mathcal O}(\Delta x^3))\tilde{t}$.
For the smooth problems, $\tau$ is taken as ${\mathcal O}(\Delta x^2)$ in our second-order BGK scheme, so the third part of the expansion for \eqref{eq:fluxdiff}
\begin{align*}
  &\frac{\tau}{{U}_{\alpha}^0p^{\alpha}}(p^0{g}^0_t+p^1{g}^0_x+p^2g^0_y)\bigg|_{i-\frac12}
  - \frac{\tau}{{U}_{\alpha}^0p^{\alpha}}(p^0{g}^0_t+p^1{g}^0_x+p^2g^0_y)\bigg|_{i+\frac12}\\
  &=(\frac{\tau}{{U}_{\alpha,i-\frac12}^0p^{\alpha}} - \frac{\tau}{{U}_{\alpha,i+\frac12}^0p^{\alpha}} )(p^0{g}^0_t+p^1{g}^0_x+p^2g^0_y)\bigg|_{i-\frac12}\\
  &+ \frac{\tau}{{U}_{\alpha,i+\frac12}^0p^{\alpha}}\left( (p^0{g}^0_t+p^1{g}^0_x+p^2g^0_y)\bigg|_{i-\frac12}
  - (p^0{g}^0_t+p^1{g}^0_x+p^2g^0_y)\bigg|_{i+\frac12}
  \right)
\end{align*}
is ${\mathcal O}(\Delta x^3)$,
and similarly, the forth part of the expansion of \eqref{eq:fluxdiff}
    \begin{align*}
      &\exp\left(-\frac{U^0_{\alpha}p^{\alpha}}{p^0\tau}\tilde{t}\right)
      \left[\hat{g}-\frac{\tau}{\hat{U}_{\alpha}p^{\alpha}}(p^0\hat{g}_t+p^1\hat{g}_x+p^2\hat{g}_y)
      -\hat{g}_xv_1\tilde{t} - \hat{g}_yv_2\tilde{t} \right.\nonumber\\
      & \left.-\left({g}^0-\frac{\tau}{{U}_{\alpha}^0p^{\alpha}}(p^0{g}^0_t+p^1{g}^0_x+p^2g^0_y)
          -{g}^0_xv_1\tilde{t} - {g}^0_yv_2\tilde{t}\right)\right]\bigg|_{i+\frac12}\nonumber\\
      &-\exp\left(-\frac{U^0_{\alpha}p^{\alpha}}{p^0\tau}\tilde{t}\right)
      \left[\hat{g}-\frac{\tau}{\hat{U}_{\alpha}p^{\alpha}}(p^0\hat{g}_t+p^1\hat{g}_x+p^2\hat{g}_y)
      -\hat{g}_xv_1\tilde{t} - \hat{g}_yv_2\tilde{t} \right.\nonumber\\
      & \left. -\left({g}^0-\frac{\tau}{{U}_{\alpha}^0p^{\alpha}}(p^0{g}^0_t+p^1{g}^0_x+p^2g^0_y)
          -{g}^0_xv_1\tilde{t} - {g}^0_yv_2\tilde{t}\right)\right]\bigg|_{i-\frac12},
    \end{align*}
is also ${\mathcal O}(\Delta x^3)$.
In summary, one has
\begin{align*}
  \frac{\hat{f}(x_{i+\frac{1}{2}},y_j,t,\vec{p})-\hat{f}(x_{i-\frac{1}{2}},y_j,t,\vec{p}) }{\Delta x}&=g_{x,i} + g_{xt,i}\tilde{t} + O(\Delta x^2)
  =g^0_{x,i} + g^0_{xt,i}\tilde{t} + {\mathcal O}(\Delta x^2).
\end{align*}
It can be seen that the term affecting accuracy in the distribution function $\hat{f}(x_{i+\frac{1}{2}},y_j,t,\vec{p})$ is $g^0+g^0_t\tilde{t}$, while the remainder $\hat{f}(x_{i+\frac{1}{2} },y_j,t,\vec{p})-g^0-g^0_t\tilde{t}$ only works near the discontinuity.
Thus a simplified gas kinetic scheme may be obtained by removing the terms which are not easy to be calculated in the moment integrals.
If the conservation variables are reconstructed, the same conclusion will be obtained in a similar way.

Based on the above discussion, we will {simplify} the distribution function $\hat{f}(x_{i+\frac{1}{2}},y_j,t,\vec{p})$ in \eqref{eq:finalstate} or \eqref{eq:finalstate1} in order to get an economic/simplified BGK schemes.
For the inviscid fluid flows, the non-equilibrium terms in the distribution function $\hat{f}$  only provides the  numerical viscosity near the discontinuity, so that they may be removed in order to simplify the integrals. On the other hand, since the terms with the coefficient $\exp\left(-\frac{{U}_{\alpha}^0p^{\alpha}}{p^0\tau}\tilde{t}\right)$ does not affect the accuracy for the smooth problems and only works near the discontinuity, so we try to simplify it appropriately by ignoring the terms including $\tilde{t}$.
At this point one can replace the distribution function $\hat{f}(x_{i+\frac{1}{2}},y_j,t,\vec{p})$ in \eqref{eq:finalstate} or \eqref{eq:finalstate1} with a simplified version
\begin{align}\label{eq:modified_f}
      \hat{f}(x_{i+\frac{1}{2}},y_j,t,\vec{p})=& g^0 + g_t\tilde{t}
 - \exp\left(-\frac{{U}_{\alpha}^0p^{\alpha}}{p^0\tau}\tilde{t}\right) {g}^0
 + \exp\left(-\frac{U^0_{\alpha}p^{\alpha}}{p^0\tau}\tilde{t}\right) \hat{g}\nonumber\\
 =& \left(1-\exp\left(-\frac{{U}_{\alpha}^0p^{\alpha}}{p^0\tau}\tilde{t}\right)\right)g^0
 + \exp\left(-\frac{{U}^0_{\alpha}p^{\alpha}}{p^0\tau}\tilde{t}\right) \hat{g} + g_t^0\tilde{t},
\end{align}
which it is much simpler than $\hat{f}$ in \eqref{eq:finalstate}.
The sum $\left(1-\exp\left(-\frac{{U}_{\alpha}^0p^{\alpha}}{p^0\tau}\right)\right)g^0 + \exp\left(-\frac{{U}^0_{\alpha}p^{\alpha}}{p^0\tau}\right) \hat{g}$ approximates the equilibrium distribution at the point $(x_{i+1/2},y_j)$ using a non-linear weighted average, which simultaneously contains the particle collisions and the free transport. The third term ensures that the approximate distribution function
 $\hat{f}(x_{i+\frac{1}{2}},y_j,t,\vec{p})$ can achieve second order in time. In Section \ref{sec:test},  the   BGK scheme with \eqref{eq:finalstate} or \eqref{eq:finalstate1} and the simplified BGK scheme with \eqref{eq:modified_f} are compared by using different examples in terms of accuracy, efficiency and resolution.

\begin{remark}
  The terms with the coefficient of $\exp\left(-\frac{{U}_{\alpha}^0p^{\alpha}}{p^0\tau}\tilde{t}\right)$
  are the nonliear weight adjusting the collision and transport effects in the distribution function at the cell interface. They do not affect the accuracy of the scheme and only provide numerical viscosity near the discontinuity. Therefore, in order to make the integral simple, $\exp\left(-\frac{{U}_{\alpha}^0p^{\alpha}}{p^0\tau}\tilde{t}\right)$is reduced to $\exp\left(-\frac{\tilde{t}}{\tau}\right)$ in the practical calculations. 
\end{remark}


\section{Moment integrals}\label{sec:moments}
 Calculating the expansion coefficients $a, b$ and $A$ in the initial distribution function $f_{h,0}$ and the equilibrium distribution $g_h$
 requires
  the moments of Maxwell-J\"uttner distribution function $g$ in \eqref{juttner}.
  This section will give the moment expressions by using the Lorentz transform, and then the elements of the matrix $M^0 $, $M^1 $ and $M^2 $.

\subsection{Matrices $M^{0}$, $M^1$ and $M^2$}
The matrices  $M^{0}$, $M^1$ and $M^2$ will be first  given in the local rest frame $(U^{\alpha})=(1,0,0,0)$ and then the Lorentz transformation is used to give the expressions in a general frame.
In the local rest frame, one has $E=U_{\alpha}p^{\prime\alpha}=p^{\prime0}$  and
\[\int_{\mathbb{R}^3}(p^{\prime0})^l gd\varXi = \frac{\rho \zeta}{4\pi K_2(\zeta)}\int_{\mathbb{R}^3} (p^{\prime0})^{l-1}e^{-\zeta p^{\prime0}} d^3\vec{p}\prime.\]
The polar coordinate transformation $$p^{\prime1}=r\sin\theta\cos\varphi,p^{\prime2}=r\sin\theta\sin\varphi,
p^{\prime3}=r\cos\theta,\theta\in[0,\pi],\varphi\in[0,2\pi),r\in[0,\infty),$$ gives $p^{\prime0}=\sqrt{1+p^{\prime1}+p^{\prime2}+p^{\prime3}}=\sqrt{1+r^2}$, so that
\begin{align*}
   \int_{\mathbb{R}^3}(p^{\prime0})^l gd\varXi =& \frac{\rho\zeta}{K_2}\int_0^{+\infty}(\sqrt{1+r^2})^{l-1}e^{-\zeta\sqrt{1+r^2}}r^2dr\\
   =&\frac{\rho\zeta}{K_2}\int^{+\infty}_0 (\cosh^{l+2}x-\cosh^l x)e^{-\zeta\cosh x}dx,
\end{align*}
where $r=\sinh x$ has been used to obtain the last equation. For the hyperbolic cosine function, we have the following power reduction
\[
    \cosh^l x=\begin{cases}
    \frac{1}{2^{l-1}}\left(\sum\limits_{i=0}^{k-1}\binom{l}{i}\cosh((l-2i)x)+\frac{1}{2}\binom{l}{k}\right), & l=2k,\\
    \frac{1}{2^{l-1}}\left(\sum\limits_{i=0}^{k}\binom{l}{i}\cosh((l-2i)x)\right), &l=2k+1,
    \end{cases}
\]
where $k=0,1,\cdots.$ With the help of above formula and the definition of the modified Bessel function of the second kind $K_n(\zeta)=\int_0^\infty e^{-\zeta\cosh x}\cosh(nx)dx$, we have
\begin{align*}
   &\int_{\mathbb{R}^3}(p^{\prime0})^l gd\varXi/ \frac{\rho\zeta}{2^{l+1}K_2} \\
   =& \begin{cases}
   K_{l+2}(\zeta) + \sum\limits_{i=1}^{k}\left( \binom{l+2}{i} -4\binom{l}{i-1}\right) K_{l+2(1-i)}(\zeta) + \frac{1}{2}\left( \binom{l+2}{k+1} -4\binom{l}{k}\right)K_0(\zeta)  ,&l=2k,\\
   K_{l+2}(\zeta) + \sum\limits_{i=1}^{k+1}\left( \binom{l+2}{i} -4\binom{l}{i-1}\right) K_{l+2(1-i)}(\zeta)  ,&l=2k+1.
   \end{cases}
\end{align*}
For example,
\begin{align*}
&\int_{\mathbb{R}^3}gd\varXi = \rho\left(G-\frac{4}{\zeta}\right),\quad\int_{\mathbb{R}^3}p^{\prime0}gd\varXi = \rho,\\
&\int_{\mathbb{R}^3}(p^{\prime0})^2 gd\varXi = \rho\left(G-\frac{1}{\zeta}\right),\quad \int_{\mathbb{R}^3}(p^{\prime0})^3gd\varXi = \rho\left(\frac{3G}{\zeta}+1\right).
\end{align*}
The elements of $M^i$ in the local rest frame need the integrals in the interval $\mathbb{R}^3$, where the terms with odd power of $p^{\prime\alpha},\alpha=1,2,3$, will vanish,
while the  terms with even power of $p^{\prime\alpha}$ can be calculated with the help of $p^{\prime0}$, such as $$\int_{\mathbb{R}^3}(p^{\prime1})^2 gd\varXi = \int_{\mathbb{R}^3}(p^{\prime2})^2 gd\varXi = \int_{\mathbb{R}^3}(p^{\prime3})^2 gd\varXi = \frac13\int_{\mathbb{R}^3}((p^{\prime0})^2-1) gd\varXi.$$
The four-vector $p^{\alpha}$ in a general frame can be described by $(p^{\prime\alpha})$ in the local rest frame via the transformation
\begin{equation*}
  p^{\alpha} = \Lambda^{\alpha}_{\beta}p^{\prime\beta},\ \
 \Lambda^{\alpha}_{\beta}=
   \begin{pmatrix}
    U^0   & U^1                                   & U^2                                  & 0\\
    U^1  & 1+\frac{(U^0-1)u_1u_1}{|\vec{u}|^2}    & \frac{(U^0-1)u_1u_2}{|\vec{u}|^2}     & 0 \\
    U^2  & \frac{(U^0-1)u_2u_1}{|\vec{u}|^2}      & 1+\frac{(U^0-1)u_2u_2}{|\vec{u}|^2}   & 0\\
    0     & 0    & 0  & 1
  \end{pmatrix}.
  \end{equation*}
Hence the elements of $M^i$ in the general frame will be calculated by
\begin{align*}
    &\int_{\mathbb{R}^3} p^{\alpha} gd\varXi = \Lambda^{\alpha}_{\beta}\int_{\mathbb{R}^3} p^{\prime\beta} gd\varXi,\\
    &\int_{\mathbb{R}^3} p^{\alpha}p^{\beta} gd\varXi = \Lambda^{\alpha}_{\gamma}\Lambda^{\beta}_{\epsilon}\int_{\mathbb{R}^3} p^{\prime\gamma} p^{\prime\epsilon}gd\varXi,\\
    &\int_{\mathbb{R}^3} p^{\alpha}p^{\beta}p^{\gamma} gd\varXi = \Lambda^{\alpha}_{\mu}\Lambda^{\beta}_{\nu}\Lambda^{\gamma}_{\kappa}\int_{\mathbb{R}^3} p^{\prime\mu} p^{\prime\nu}p^{\prime\kappa}gd\varXi.
\end{align*}
The explicit expressions of $M^i,i=0,\cdots,d$ will be given in Appendices \ref{M0M11D} and \ref{M2D} for the 1D  and 2D cases, respectively.



\subsection{Moments in half plane}
For calculating the numerical flux by inserting the $\hat{f}$ in \eqref{eq:modified_f} or \eqref{eq:finalstate}  into \eqref{eq:2DF}, the integrals of the Maxwell-J\"uttner distribution in the half plane are needed. The specific integral formulas for the 1D and 2D cases are given
below.

Similar to \cite{Anderson1974}, the momentum $p^{\alpha}$ at a point is decomposed as
\begin{equation*}
    p^{\alpha} = U^{\alpha}E+\sqrt{E^2-1}l^{\alpha},
\end{equation*}
where $l^{\alpha}$ is an unit space-like vector orthogonal to $U^{\alpha}$, i.e.
\[l^{\alpha}l_{\alpha}=-1,l^{\alpha}U_{\alpha}=0.\]
Introduce an orthogonal tetrad $n^{\alpha}_i(i=1,2,3)$
orthogonal to $U^{\alpha}$ so that
 \[U_{\alpha}n^{\alpha}_i=0,\quad g_{\alpha\beta}n^{\alpha}_in^{\beta}_j=-\delta_{i,j}.\]
The scalar product of two four-vectors $U_{\alpha}n^{\alpha}_i$ and the tensor product $g_{\alpha\beta}n^{\alpha}_in^{\beta}_j$ are invariant, so  the Lorentz transformation can be used to the local rest frame with $(U^{\alpha})=(1,0,0,0)$ and  $n^{\alpha}_i=\delta_{i,\alpha}(i=1,2,3)$, which satisfy $U_{\alpha}n^{\alpha}_i=0$ and  $g_{\alpha\beta}n^{\alpha}_in^{\beta}_j=-\delta_{i,j}$. Then $n^{\alpha}_i$ in the general frame can be given by
\begin{equation*}
   n^0_i=U^i,\quad n^j_i=((U^0)^2-1)^{-1}U^iU^j(U^0-1)+\delta_{i,j},\quad i,j=1,2,3,
\end{equation*}
and $l^{\alpha}$  is taken as
\begin{equation*}
  l^{\alpha} = a_1n_1^{\alpha} + a_2n_2^{\alpha} + a_3n_3^{\alpha},
\end{equation*}
where $(a_1,a_2,a_3)$ is any unit vector.
It is easy to check that $l^{\alpha}U_{\alpha}=0$ and $l^\alpha l_{\alpha}=0$.

\subsubsection{1D case}
In the 1D case, let $(U^{\alpha})=(U^0,U^1,0,0)$, and $a_1=\cos\theta,a_2=\sin\theta\cos\varphi,a_3=\sin\theta\sin\varphi$, then
\begin{align*}
    &p^0 = U^0E + U^1\sqrt{E^2-1} \cos\theta,\quad
    p^1 = U^1E + U^0\sqrt{E^2-1} \cos\theta ,\\
    &p^2 = \sqrt{E^2-1}\sin \theta \cos\varphi,\quad
    p^3 = \sqrt{E^2-1}\sin \theta \sin\varphi,
\end{align*}
where $\theta\in[0,\pi],\varphi\in[0,2\pi),E\in[1,\infty)$.
Using that coordinate transformation and the volume element $d\varXi=\sqrt{E^2-1}\sin\theta d\varphi d\theta dE$ can simplify
the calculation of the triple integrals in the moment computations, which   {reduce} to a double integral with respect to $E$ and $\theta$ in one dimension since the integrands  do  not depend on the variable $\varphi$.
Specially, the moments for any arbitrary function $\phi$ of $p^{\alpha}$ in the half plane can be obtained by
\begin{align*}
\int_{\mathbb{R}^{3}\cap{p^{1}>0}}\phi gd\varXi=&\frac{\rho\zeta}{2K_{2}(\zeta)}\left(H(u_1)\int_{1}^{\frac{1}{\sqrt{1-u_1^2}}}\int_{-1}^{1}\phi\exp(-\zeta E)\sqrt{E^2-1}d\eta dE \right.\notag\\
&\left.+\int_{\frac{1}{\sqrt{1-u_1^2}}}^{\infty}\int_{-\frac{u_1E}{\sqrt{E^2-1}}}^{1}\phi\exp(-\zeta E)\sqrt{E^2-1}d\eta dE\right),\\
\int_{\mathbb{R}^{3}\cap{p^{1}<0}}\phi gd\varXi=&\frac{\rho\zeta}{2K_{2}(\zeta)}\left((1-H(u_1))\int_{1}^{\frac{1}{\sqrt{1-u_1^2}}}\int_{-1}^{1}\phi\exp(-\zeta E)\sqrt{E^2-1}d\eta dE \right.\notag\\
&\left.+\int_{\frac{1}{\sqrt{1-u_1^2}}}^{\infty}\int_{-1}^{-\frac{u_1E}{\sqrt{E^2-1}}}\phi\exp(-\zeta E)\sqrt{E^2-1}d\eta dE\right),
\end{align*}
which can be integrated numerically, where $\eta=\cos\theta\in[-1,1]$ for $\theta\in[0,\pi]$.

\subsubsection{2D case}
The 2D case is much more complicate than the 1D case.
 Let   $(U^{\alpha})=(U^0,U^1,U^2,0)$, and $a_1=\sqrt{1-\eta^2}\cos\varphi$, $a_2=\sqrt{1-\eta^2}\sin\varphi$,
$a_3= \eta$, then
\begin{align*}
    &p^0 = U^0E + \sqrt{E^2-1}\sqrt{1-\eta^2}\left(U^1\cos\varphi + U^2\sin\varphi \right),\\
    &p^1 = U^1E + \sqrt{E^2-1}\sqrt{1-\eta^2}\left(\left(1+\frac{(U^1)^2}{1+U^0}\right)\cos\varphi + \frac{U^1U^2}{1+U^0}\sin\varphi \right),\\
    &p^2 = U^2E+\sqrt{E^2-1}\sqrt{1-\eta^2}\left(\left(\frac{U^1U^2}{1+U^0}\right)\cos\varphi + \left(1+\frac{(U^2)^2}{1+U^0}\right)\sin\varphi
    \right),\\
    &p^3 = \sqrt{E^2-1}\eta,
\end{align*}
where $\eta\in[-1,1],\varphi\in[0,2\pi),E\in[1,\infty)$, and the volume element
in the new coordinate becomes $d\varXi=\sqrt{E^2-1}d\varphi d\eta dE$.
%
Write $p^1$ as
\[p^1=U^1 E + \sqrt{E^1-1}\sqrt{1-\eta^2}\sqrt{1+(U^1)^2}\sin(\varphi + \varphi_c),\]
with $\sin\varphi_c=\left(1+(U^1)^2/(1+U^0)\right)/\sqrt{1+(U^1)^2}$ and $\cos\varphi_c=U^1U^2/(1+U^0)/\sqrt{1+(U^1)^2}$
For the sake of simplicity, we will still use $\varphi$ instead of $\varphi+\varphi_c$ hereafter. Thus $p^{\alpha}$ will be rewritten as
\begin{equation}\label{eq:p1}
\begin{aligned}
    &p^1=U^1 E + \sqrt{E^2-1}\sqrt{1-\eta^2}\sqrt{1+(U^1)^2}\sin\varphi,\\
    &p^0 = U^0E +\frac{ \sqrt{E^2-1}\sqrt{1-\eta^2}}{\sqrt{1+(U^1)^2}}\left(U^0U^1\sin\varphi-U^2\cos\varphi\right),\\
    &p^2 = U^2E+\frac{\sqrt{E^2-1}\sqrt{1-\eta^2}}{\sqrt{1+(U^1)^2}}\left(U^1U^2\sin\varphi - U^0\cos\varphi\right),\\
    &p^3 = \sqrt{E^2-1}\eta,
\end{aligned}
\end{equation}
where $\eta\in[-1,1],\varphi\in[0,2\pi),E\in[1,\infty)$.

From   \eqref{eq:p1}, it is not difficult to know that $p^1>0$ is equivalent to
\[
\sin\varphi>-\frac{U^1  E}{\sqrt{1+(U^1)^2}\sqrt{E^2-1}\sqrt{1-\eta^2}}=:g(E,\eta),
\]
which holds only for $g(E,\eta)\leq1$ since $\sin\varphi\leq1$.
It includes two cases: $g(E,\eta)\leq-1$ when $\varphi\in(0,2\pi)$ and $|g(E,\eta)|<-1$ when $\varphi\in(\arcsin g,\pi-\arcsin g)$.
The values of $E$ and $\eta$ should be discussed case by case as follows:
\begin{enumerate}[(i)]
\item $g(E,\eta)\leq -1$ is equivalent to $\sqrt{1-\eta^2}\leq\frac{U^1E}{\sqrt{1+(U^1)^2}\sqrt{E^2-1}}=:h(E)$. It is only valid for $U^1\geq0$ since $\sqrt{1-\eta^2}\geq0$.
On the other hand, 
if $h(E)\geq1$, i.e. $E\leq\sqrt{1+(U^1)^2}$, then
$\eta\in(-1,1)$; otherwise,
$$\eta\in\left(-1,-\sqrt{1-h(E)^2}\right)\cup\left(\sqrt{1-h(E)^2},1\right).$$
In this case, the integral of an arbitrary function $\phi(p^{\alpha})$ in the half plane (i.e. $p^1>0$) is calculated by
\begin{align*}
&\int_{\mathbb{R}^{3}\cap{p^{1}>0}}\phi gd\varXi=H(U^1)\left(\int_{1}^{\frac{1}{\sqrt{1+(U^1)^2}}}\int_{-1}^{1}\int_0^{2\pi}\phi g\sqrt{E^2-1}d\varphi d\eta dE \right.\\
& + \int_{\frac{1}{\sqrt{1+(U^1)^2}}}^{+\infty}\int_{-1}^{-\sqrt{1-h(E)^2}}\int_0^{2\pi}\phi g\sqrt{E^2-1}d\varphi d\eta dE
 \\ & \left.+ \int_{\frac{1}{\sqrt{1+(U^1)^2}}}^{+\infty}\int^{1}_{\sqrt{1-h(E)^2}}\int_0^{2\pi}\phi g\sqrt{E^2-1}d\varphi d\eta dE\right).
\end{align*}

\item $|g(E,\eta)|\leq 1$ is equivalent to $|\eta|<\sqrt{1-h(E)^2}$. From $h(E)^2\leq1$, one can get $E\geq\sqrt{1+(U^1)^2}$. Thus, in this case, the integral of an arbitrary function $\phi(p^{\alpha})$ in the positive half plane is calculated by
\begin{align*}
\int_{\mathbb{R}^{3}\cap{p^{1}>0}}\phi gd\varXi=\int_{\sqrt{1+(U^1)^2}}^{+\infty}\int_{-\sqrt{1-h(E)^2}}^{\sqrt{1-h(E)^2}}\int_{\arcsin g}^{\pi-\arcsin g}\phi g\sqrt{E^2-1}d\varphi d\eta dE.
\end{align*}
\end{enumerate}

Combining those two cases in above we conclude that the integral of $\phi$ in the positive half plane ($p^1>0$) is obtained by
\begin{align}
\int_{\mathbb{R}^{3}\cap{p^{1}>0}}\phi gd\varXi&=H(U^1)\left(\int_{1}^{\frac{1}{\sqrt{1+(U^1)^2}}}\int_{-1}^{1}\int_0^{2\pi}\phi g\sqrt{E^2-1}d\varphi d\eta dE\right.\notag\\
 &\left.+2\int_{\frac{1}{\sqrt{1+(U^1)^2}}}^{+\infty}\int^{1}_{\sqrt{1-h(E)^2}}\int_0^{2\pi}\phi g\sqrt{E^2-1}d\varphi d\eta dE\right)\notag\\
&+\int_{\sqrt{1+(U^1)^2}}^{+\infty}\int_{-\sqrt{1-h(E)^2}}^{\sqrt{1-h(E)^2}}\int_{\arcsin g}^{\pi-\arcsin g}\phi g\sqrt{E^2-1}d\varphi d\eta dE,\label{eq:PositiveP1}
\end{align}
since $\phi$ is an even function with respect to $\eta$ in two dimension.

In the case of $p^1<0$, the integral can be done in a similar way by
\begin{align}
\int_{\mathbb{R}^{3}\cap{p^{1}<0}}\phi gd\varXi
&=                      \left(1-H(U^1)\right)\left(\int_{1}^{\frac{1}{\sqrt{1+(U^1)^2}}}\int_{-1}^{1}\int_0^{2\pi}\phi g\sqrt{E^2-1}d\varphi d\eta dE \right.\notag\\
 &\left.+ 2\int_{\frac{1}{\sqrt{1+(U^1)^2}}}^{+\infty}\int^{1}_{\sqrt{1-h(E)^2}}\int_0^{2\pi}\phi g\sqrt{E^2-1}d\varphi d\eta dE\right)\notag\\
&+\int_{\sqrt{1+(U^1)^2}}^{+\infty}\int_{-\sqrt{1-h(E)^2}}^{\sqrt{1-h(E)^2}}\int_{\pi-\arcsin g}^{2\pi+\arcsin g}\phi g\sqrt{E^2-1}d\varphi d\eta dE. \label{eq:negativeP1}
\end{align}

In \eqref{eq:PositiveP1} and \eqref{eq:negativeP1}, the triple integrals can not be simplified at all in the new coordinates.
However, fortunately, in our simplified BGK methods, the integrand function $\phi$ is one of the forms $p^{\alpha}$ and $p^{\alpha}p^{\beta}$, while  those terms $p^{\alpha}$ and $p^{\alpha}p^{\beta}$ are just some linear combinations of $Q_i$, $i=1,\cdots, 6$, defined by \begin{align*}
Q_1:&=1,\ Q_2:=\sqrt{1-\eta^2}\cos\varphi,\ Q_3:=\sqrt{1-\eta^2}\sin\varphi,
\\
Q_4:&=(1-\eta^2)\cos2\varphi,\ Q_5:=(1-\eta^2)\sin2\varphi,\ Q_6:=1-\eta^2,
\end{align*}
for which $\eta$ and $\varphi$ can be integrated exactly. As a result, the triple integrals in the half plane reduce to a single integral with respect to $E$, which can be effectively integrated.
Specially, the integrals
\begin{equation*}
\begin{aligned}
   &I_{0,k}=\int_{-1}^{1}\int_0^{2\pi}Q_kd\varphi d\eta,\quad I_{1,k}=\int_{\sqrt{1-h(E)^2}}^{1}\int_0^{2\pi}Q_kd\varphi d\eta,\\
   &I_{2,k}=\int_{-\sqrt{1-h(E)^2}}^{\sqrt{1-h(E)^2}}\int_{\arcsin g}^{\pi-\arcsin g}Q_kd\varphi d\eta,
   \quad I_{3,k}=\int_{-\sqrt{1-h(E)^2}}^{\sqrt{1-h(E)^2}}\int_{\pi-\arcsin g}^{2\pi+\arcsin g}Q_kd\varphi d\eta,
\end{aligned}
\end{equation*}
where $k=1,\cdots,6$, are respectively calculated as follows
$$
\begin{aligned}
  &I_{0,1}=4\pi,\ I_{0,6}=\frac{8}{3}\pi,\
  I_{1,1}=2\pi\left(1-\sqrt{1-h(E)^2}\right),
  I_{\ell_1,\ell_2}=0, \ell_1=1,2, \ell_2=2,\cdots,5,\\
  &I_{1,6}=\frac{2}{3}\pi\left(1-\sqrt{1-h(E)^2}\right)^2\left(2+\sqrt{1-h(E)^2}\right),\\
  &I_{2,1}=\int_{-\sqrt{1-h(E)^2}}^{\sqrt{1-h(E)^2}}\left(\pi+2\arcsin\left(\frac{h(E)}{\sqrt{1-\eta^2}}\right)\right)d\eta \\
  &\quad~= 2\pi\sqrt{1-h(E)^2}(1+\sign(h(E)))-2\pi\sign(h(E))+2\pi h(E),\\
  &I_{2,2}=I_{2,5}=0,\quad
  I_{2,3}=\int_{-\sqrt{1-h(E)^2}}^{\sqrt{1-h(E)^2}}\left(2\sqrt{1-\eta^2-h(E)^2}\right)d\eta = \pi(1-h(E)^2),\\
   &I_{2,4}=\int_{-\sqrt{1-h(E)^2}}^{\sqrt{1-h(E)^2}}2h\sqrt{1-\eta^2-h(E)^2}d\eta = h(E)I_{2,3}=h(E)\pi(1-h^2),
  \end{aligned}
 $$
 $$
\begin{aligned}
  &I_{2,6}=\int_{-\sqrt{1-h(E)^2}}^{\sqrt{1-h(E)^2}}(1-\eta^2)\left(\pi+2\arcsin \left(\frac{h(E)}{\sqrt{1-\eta^2}}\right)\right)d\eta \\
  &\quad~= \frac{2}{3}\pi\sqrt{1-h(E)^2}(h(E)^2+2)(1+\sign h(E))+\frac{1}{3}\pi(h(E)^3-4\sign h(E)) + \pi h(E),\\
  &I_{3,1}=\int_{-\sqrt{1-h(E)^2}}^{\sqrt{1-h(E)^2}}\left(\pi-2\arcsin \left(\frac{h(E)}{\sqrt{1-\eta^2}}\right)\right)d\eta \\
  &\quad~= 2\pi\sqrt{1-h(E)^2}(1-\sign h(E))+2\pi\sign h(E)-2\pi h(E),\\
  &I_{3,2}=I_{3,5}=0,I_{3,3}=-I_{2,3}=-\pi(1-h(E)^2),\quad I_{3,4}=-I_{2,4}=-h(E)\pi(1-h(E)^2),\\
  &I_{3,6}=\int_{-\sqrt{1-h(E)^2}}^{\sqrt{1-h(E)^2}}(1-\eta^2)\left(\pi-2\arcsin \left(\frac{h(E)}{\sqrt{1-\eta^2}}\right)\right)d\eta \\
  &\quad~= \frac{2}{3}\pi\sqrt{1-h(E)^2}(h(E)^2+2)(1-\sign h(E))-\frac{1}{3}\pi(h(E)^3-4\sign h(E)) - \pi h(E).
\end{aligned}
$$

\section{Numerical experiments}
\label{sec:test}
This section will solve several 1D and 2D problems of the special-relativistic Euler equations for the perfect relativistic gas to demonstrate the accuracy and efficiency of our simplified BGK (sBGK) schemes, which will be compared to the second-order accurate BGK-type  and KFVS schemes.
Moreover, our sBGK schemes are also  compared to
the  BGK  scheme (before simplification) by using the 1D numerical results 
in order to illustrate that the former is not inferior to
the latter in terms of the shock wave capture and the accuracy.
In our computations, the characteristic variables are reconstructed with the van Leer limiter,
and the collision time $\tau$  is taken as
\[
\tau=\tau_m+C_{2}\Delta t^{\alpha_2}_{n}\frac{|p_{L}-p_{R}|}{p_{L}+p_{R}},
\]
where $\tau_m=C_{1}\Delta t^{\alpha_1}_n$, $C_1$, $C_2$
$\alpha_1$ and $\alpha_2$ are four constants, $p_L$ and $p_R$ are the left and right-hand limits of the pressure at the cell interface, respectively.
Unless specifically stated,  $C_1=C_2=1$,  $\alpha_1=2$, $\alpha_2=1$, and the time step-size $\Delta t_n$ is determined by the CFL condition with the CFL number of 0.4.

\subsection{1D case}
\begin{example}[Accuracy test]\label{ex:accurary1D}\rm To check the accuracy of BGK, sBGK, KFVS and BGK-type schemes, we first solve a smooth problem, which describes a sine wave propagating  periodically in the domain $\Omega=[0,1]$. The 
exact solutions are given by
 \[ \rho (x,t)= 1+0.5\sin(2\pi (x-0.2t)), \quad u_1(x,t)=0.2, \quad p(x,t)=1.\]
The domain $\Omega$ is divided into $N$ uniform cells and the periodic boundary conditions are specified at $x=0,1$.
 \scriptsize
 \begin{table}[H]
   \setlength{\abovecaptionskip}{0.cm}
   \setlength{\belowcaptionskip}{-0.cm}
   \caption{Example \ref{ex:accurary1D}: Numerical $l^1$-errors of $\rho$ and convergence rates at $t = 0.2$ by BGK, sBGK, KFVS and BGK-type schemes.}\label{accuracy}
     \begin{tabular}{*{9}{c}}
       \toprule
       \multirow{2}*{$N$} &\multicolumn{2}{c}{BGK} &\multicolumn{2}{c}{sBGK} &\multicolumn{2}{c}{KFVS} &\multicolumn{2}{c}{BGK-type}\\
       \cmidrule(lr){2-3}\cmidrule(lr){4-5}\cmidrule(lr){6-7}\cmidrule(lr){8-9}
       & error  & order  &   error  &   order &   error  & order  &   error  & order \\
       \midrule
       25   & 1.7274e-03 &  --       & 1.7541e-03   & --          &2.2608e-03    &--            & 1.7061e-03  & --      \\
       50   & 4.9909e-04 &  1.7913   & 5.1213e-04  & 1.7761       &6.0147e-04    &1.9103        & 4.2336e-04  & 2.0107  \\
       100  & 1.2569e-04 &  1.9894   & 1.2875e-04  & 1.9919       &1.5074e-04    &1.9964        & 1.1309e-04  & 1.9044  \\
       200  & 3.1865e-05 &  1.9798   & 3.2534e-05  & 1.9846       &3.7148e-05    &2.0207        & 2.6244e-05  & 2.1075  \\
       400  & 7.1502e-06 &  2.1559   & 7.1503e-06  & 2.1859       &8.8233e-06    &2.0739        & 6.4297e-06  & 2.0291  \\
     \bottomrule
     \end{tabular}
 \end{table}
\end{example}
Table \ref{accuracy} gives the $l^1$-errors of $\rho$  at $t=0.2$  and  corresponding convergence rates for the BGK, sBGK, KFVS and BGK-type schemes.
The results show that all those schemes can achieve second-order accuracy, which are in accordance with the theoretic results.
 However, the simplified BGK scheme is simpler and more efficient than the BGK.

The following simulates three Riemann problems in  the domain $[0,1]$, whose the analytic solutions are built on Appendix
\ref{sub:wavestructure}.
\begin{example}[Riemann problem I]\label{ex:1DRP3}\rm
 The initial data are taken as
 \begin{equation*}
   \label{eqex:1DRP3}
    (\rho,u_1,p)(x,0)=
   \begin{cases}
     (1,-0.5,2),& x<0.5,\\
     (1,0.5,2),& x>0.5.
   \end{cases}
 \end{equation*}

Fig. \ref{fig:1DRP3} plots the numerical results at $t=0.5$ obtained by the BGK scheme (``{$\circ$}"), the KFVS scheme (``{$*$}") and the BGK-type scheme (``{$+$}") with 200 uniform cells.
 The solutions consists of a left-moving rarefaction wave, a stationary contact discontinuity, and a right-moving rarefaction wave,
Fig. \ref{fig22:1DRP3} give a comparison of the sBGK scheme with the BGK scheme.
It is seen that
 the numerical solutions are in good agreement with the exact solutions, but there exists serious undershoot in the density at $x=0.5$. The phenomena is also observed in corresponding shock tube problem of the non-relativistic case. The sBGK scheme performs as well as the  BGK scheme, but much simpler than the original one.
\end{example}

\begin{figure}[h]
 \centering
 \subfigure[$ \rho $]{
   \includegraphics[width=0.45\textwidth]{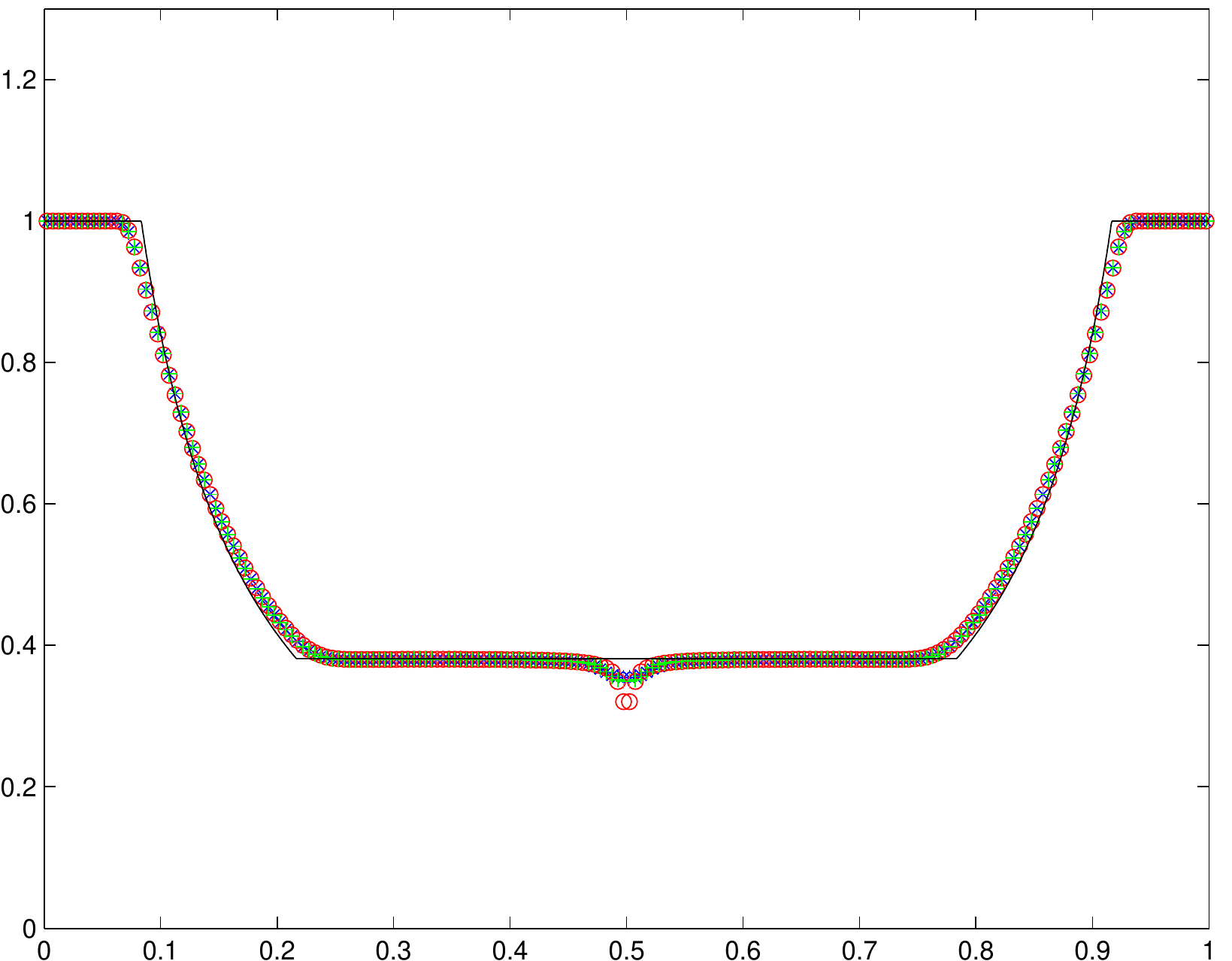}
 }
\subfigure[close-up of $\rho$]{
   \includegraphics[width=0.45\textwidth]{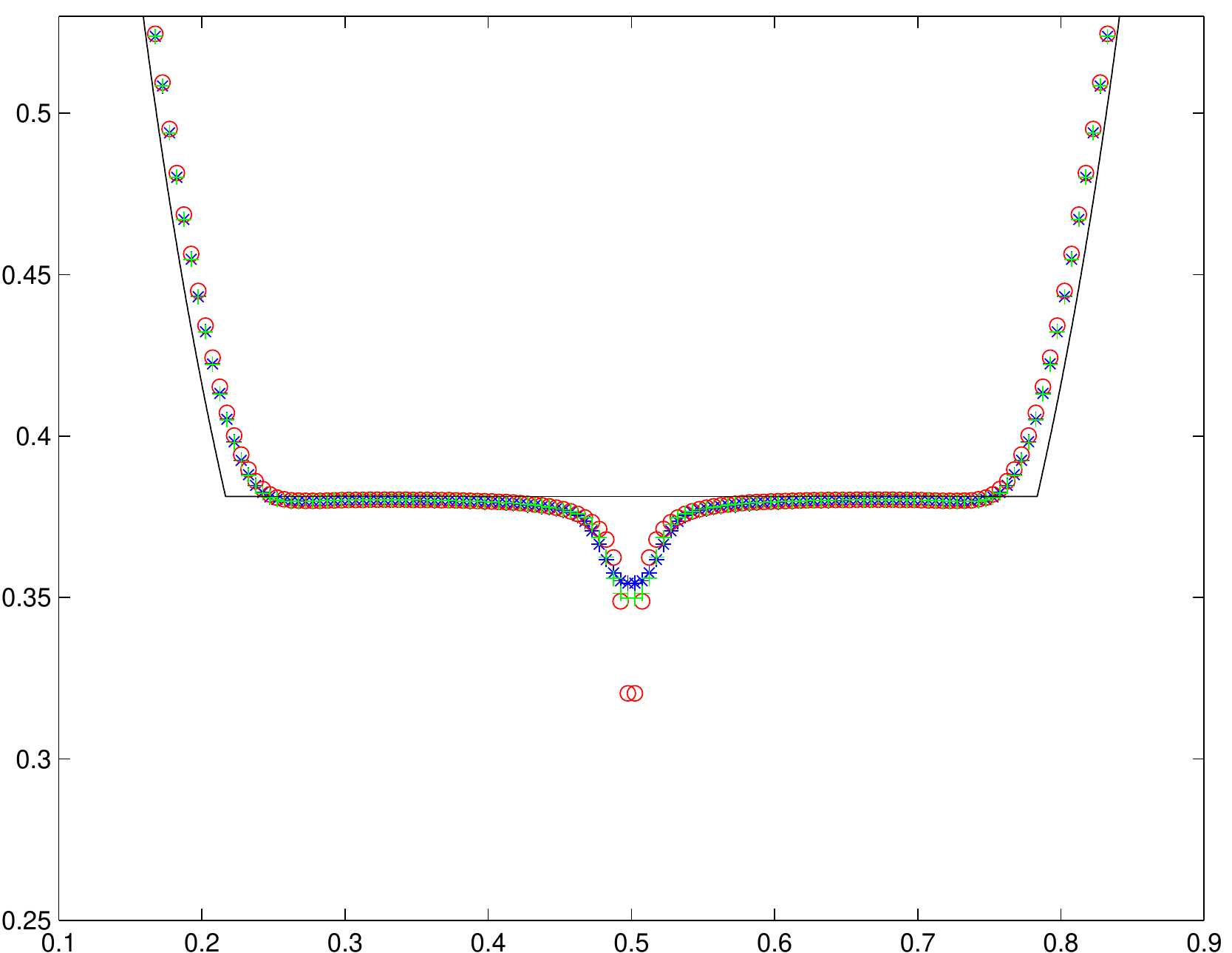}
 }
 \subfigure[$u_1$]{
   \includegraphics[width=0.45\textwidth]{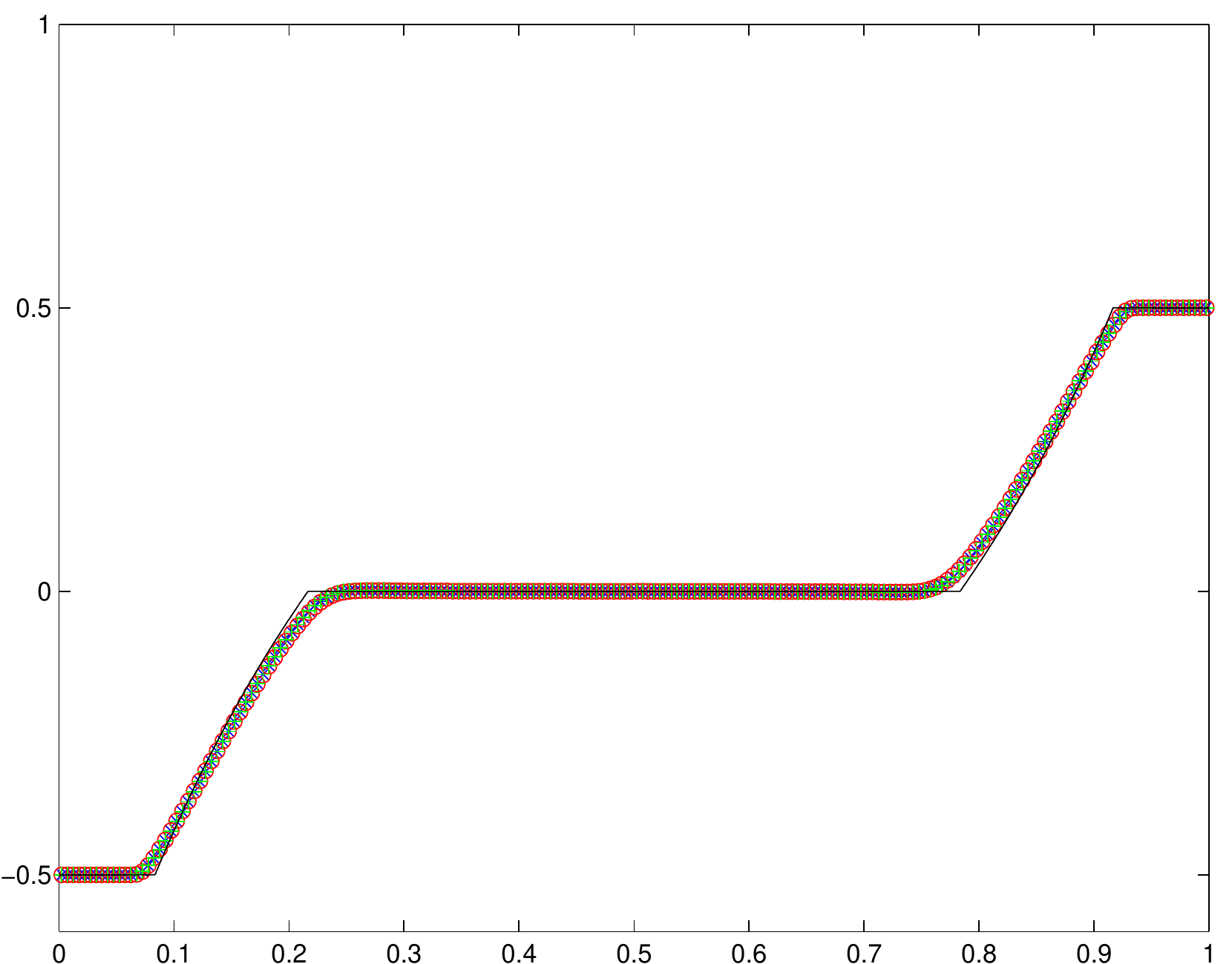}
 }
 \subfigure[$p$]{
   \includegraphics[width=0.45\textwidth]{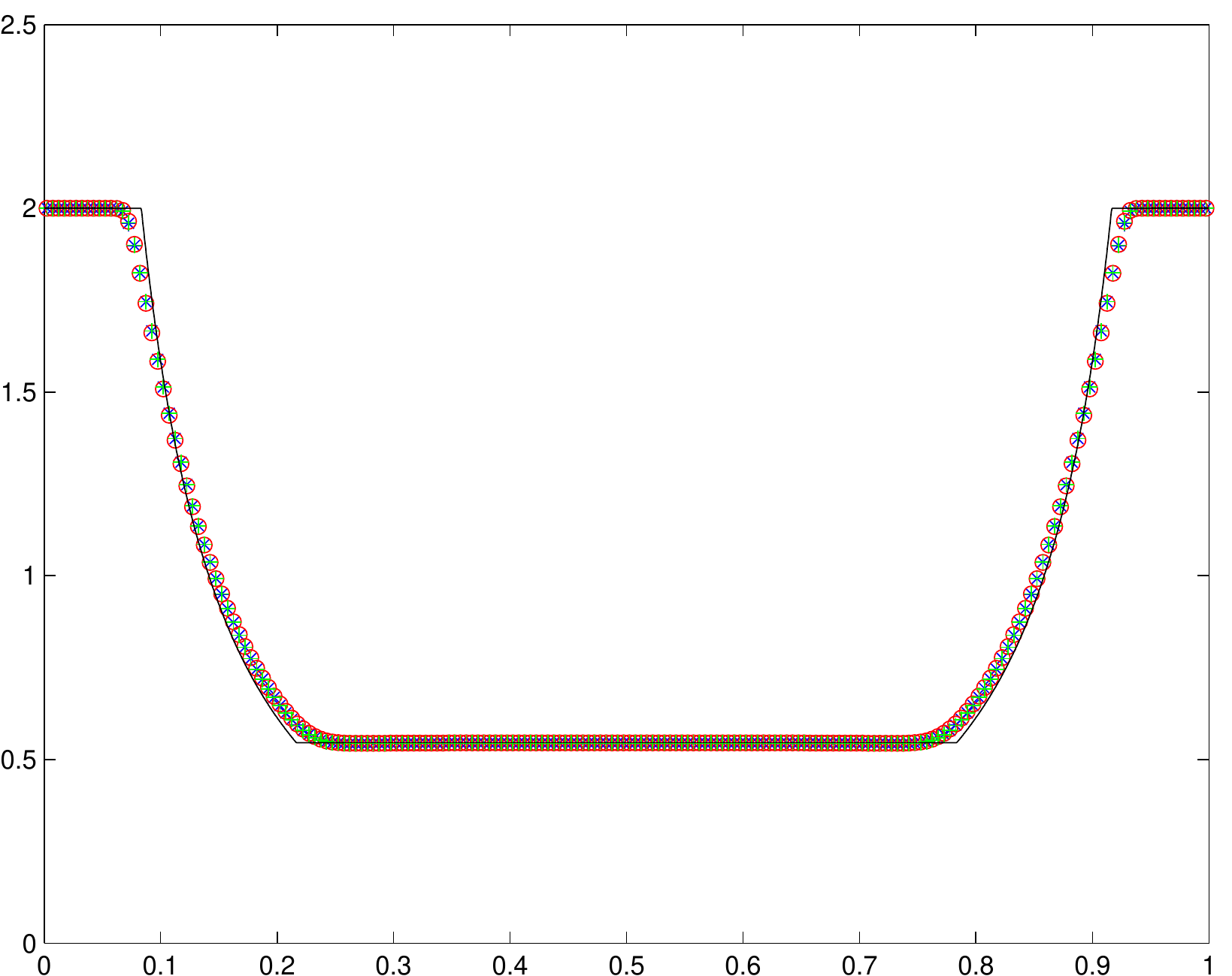}
 }
 \caption{Example \ref{ex:1DRP3}: The solutions 
 at $t=0.5$ obtained by the BGK scheme (``{$\circ$}"), the KFVS scheme (``{$*$}") and the BGK-type scheme (``{$+$}") with 200 uniform cells.}
 \label{fig:1DRP3}
\end{figure}

\begin{figure}[h]
 \centering
 \subfigure[$ \rho $]{
   \includegraphics[width=0.45\textwidth]{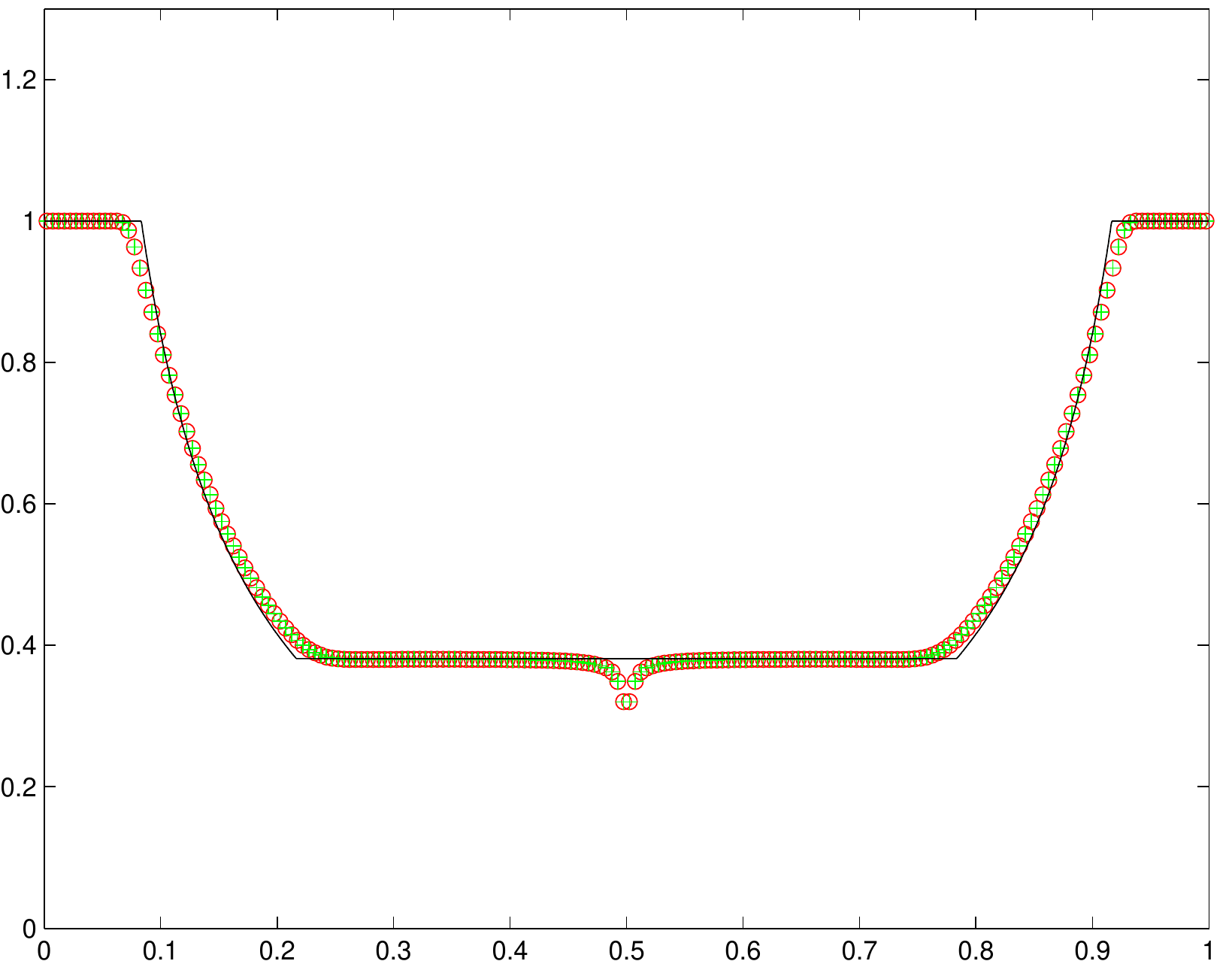}
 }
\subfigure[close-up of $\rho$]{
   \includegraphics[width=0.45\textwidth]{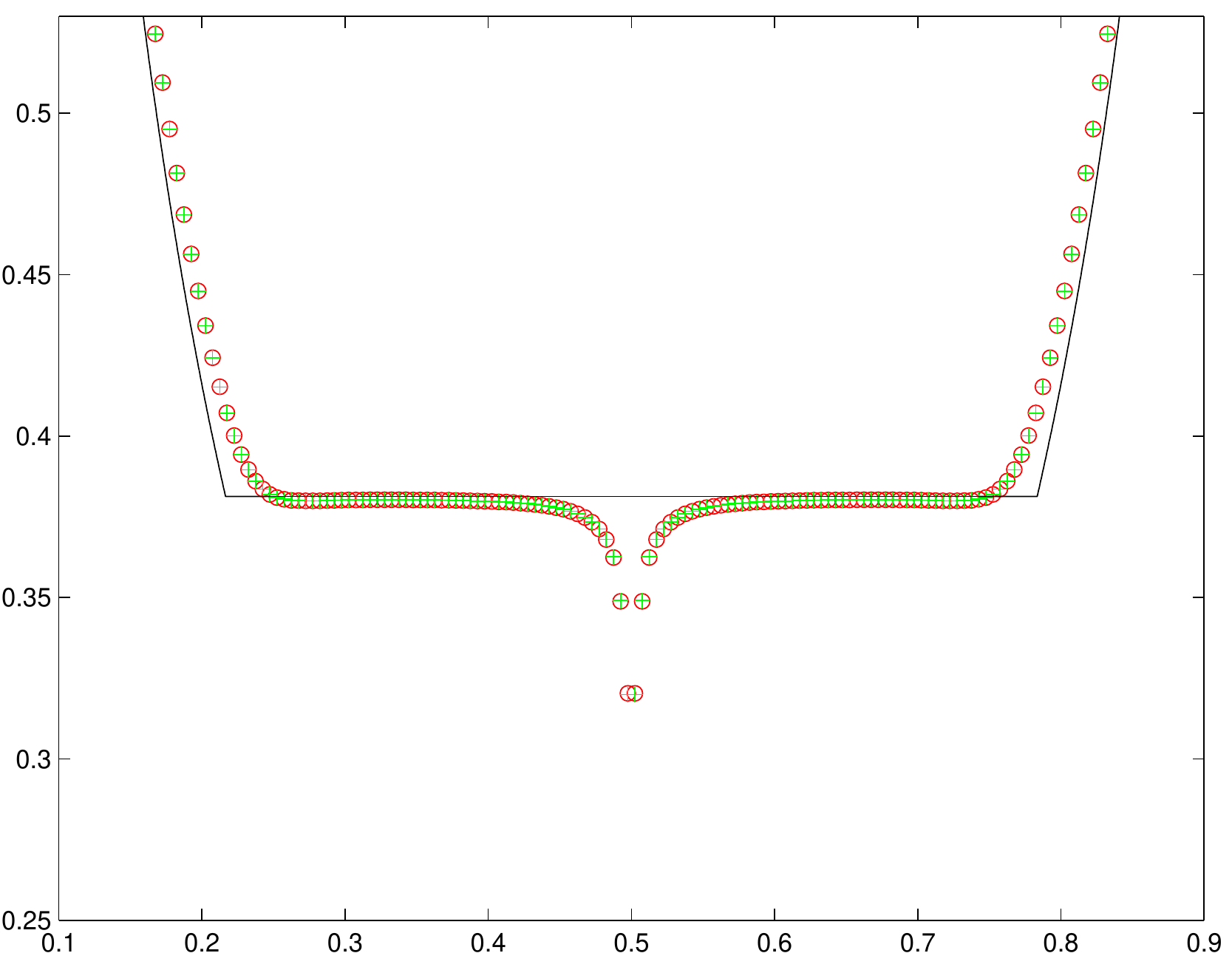}
 }
 \subfigure[$u_1$]{
   \includegraphics[width=0.45\textwidth]{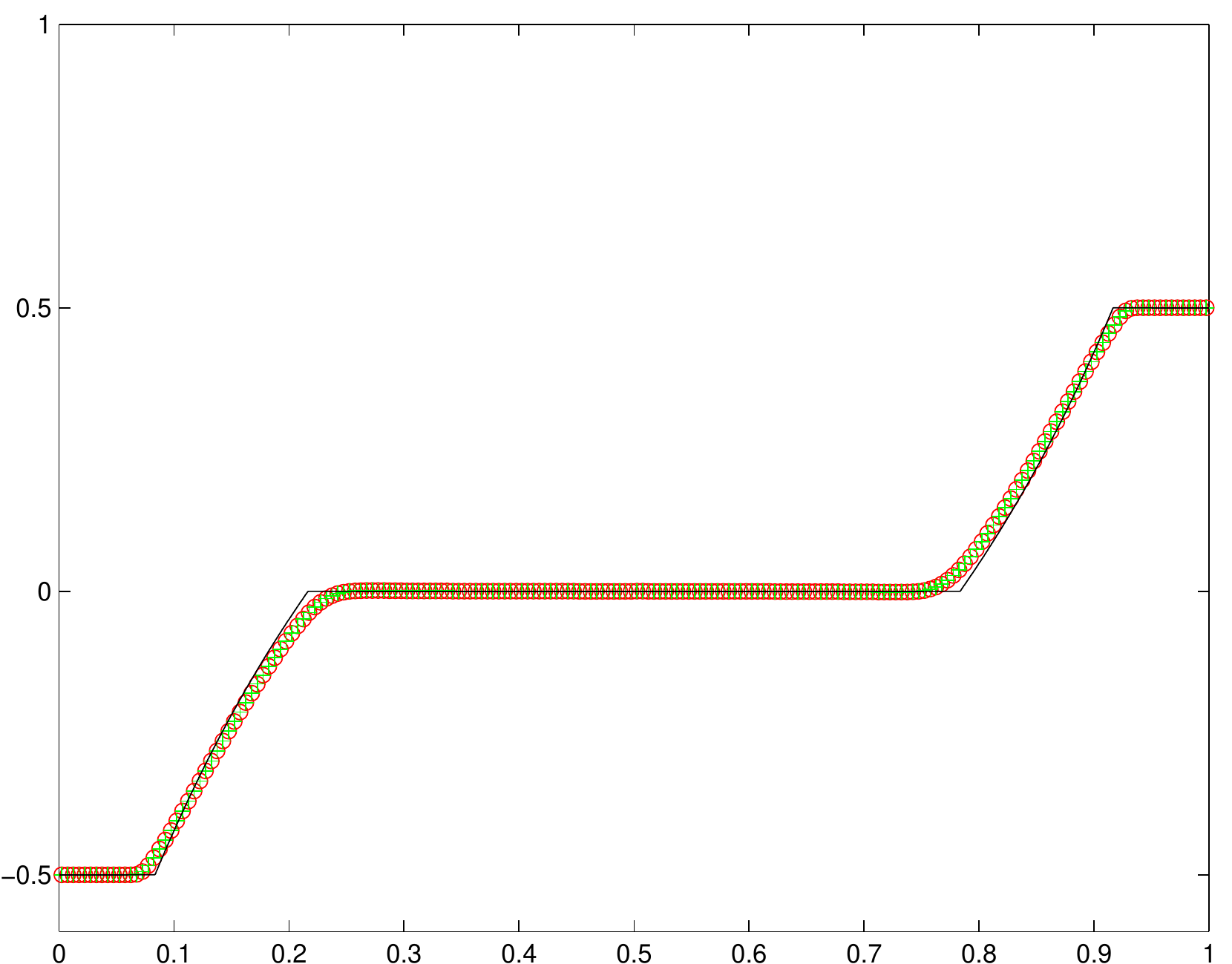}
 }
 \subfigure[$p$]{
   \includegraphics[width=0.45\textwidth]{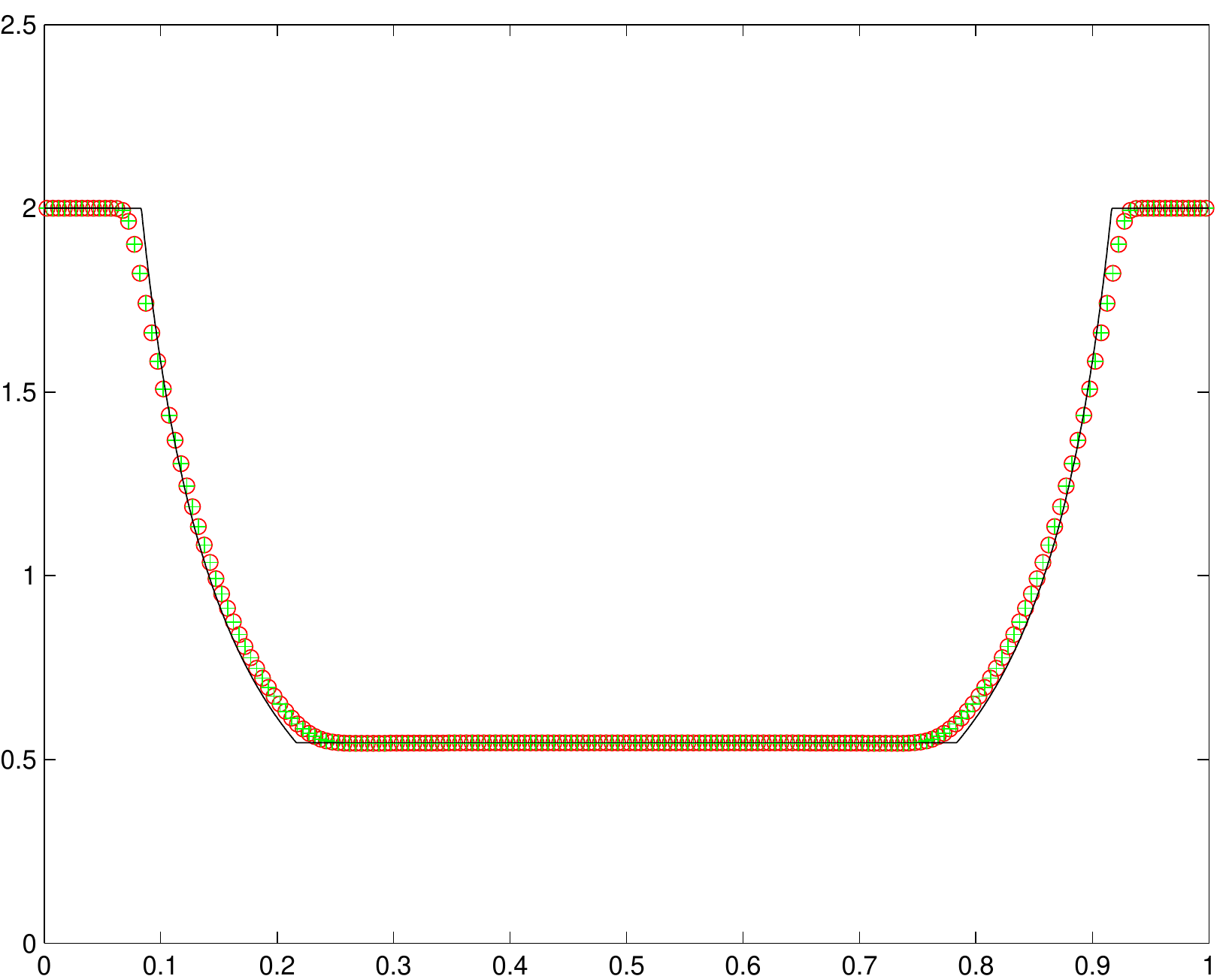}
 }
 \caption{Example \ref{ex:1DRP3}: The solutions 
 at $t=0.5$ obtained by the BGK scheme (``{$\circ$}") and the sBGK scheme (``{$+$}") with 200 uniform cells.}
 \label{fig22:1DRP3}
 \end{figure}

\begin{example}[Riemann problem II]\rm \label{ex:1DRP1} 
 The initial data are given by
 \begin{equation*}
   \label{exeq:1DRP1}
  (\rho,u_1,p)(x,0)=
   \begin{cases}
     (1.0,0.6,3.0),& x<0.5,\\
     (1.0,-0.5,2.0),& x>0.5.
   \end{cases}
 \end{equation*}
\end{example}
As the time increases, the initial discontinuity will be decomposed into a left-moving shock wave,  a right-moving contact discontinuity,
and a right-moving shock wave. Fig. \ref{fig:1DRP1}
displays the numerical results at $t=0.5$ by using
our BGK scheme (``{$\circ$}"), the KFVS scheme (``{$*$}") and the BGK-type scheme (``{$+$}")
with 400 uniform cells, where the solid line
denotes the exact solution.
It can be seen that the BGK scheme resolves the contact discontinuity better than the second-order accurate BGK-type and
KFVS schemes,
and they can well capture other waves. The comparison between
 the BGK and sBGK schemes in Fig. \ref{fig2:1DRP1} shows that the sBGK scheme exhibits almost the same resolution of the wave configuration as the BGK scheme.

\begin{figure}[h]
 \centering
 \subfigure[$ \rho $]{
  \includegraphics[width=0.45\textwidth]{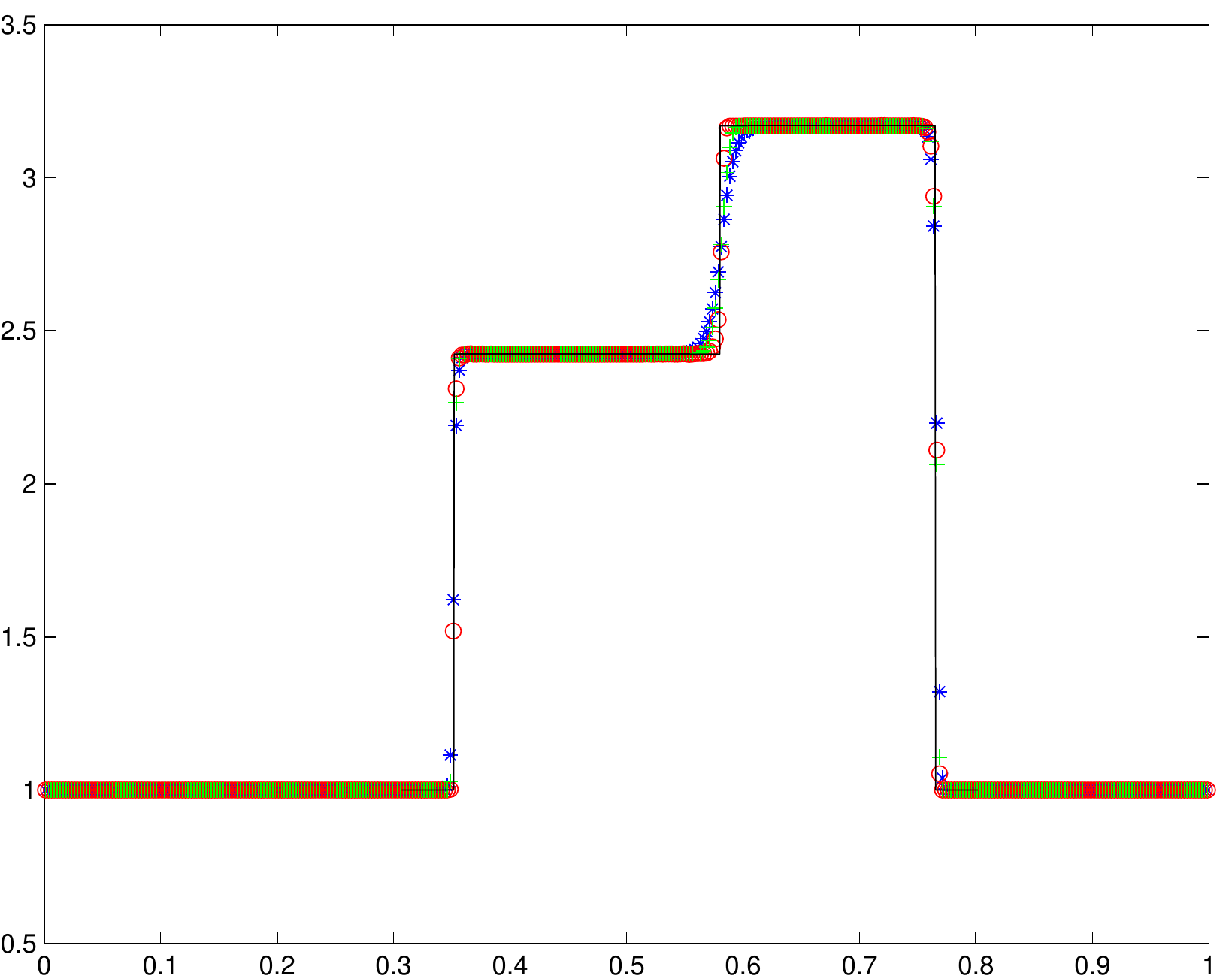}
 }
 \subfigure[close-up of $\rho$]{
  \includegraphics[width=0.45\textwidth]{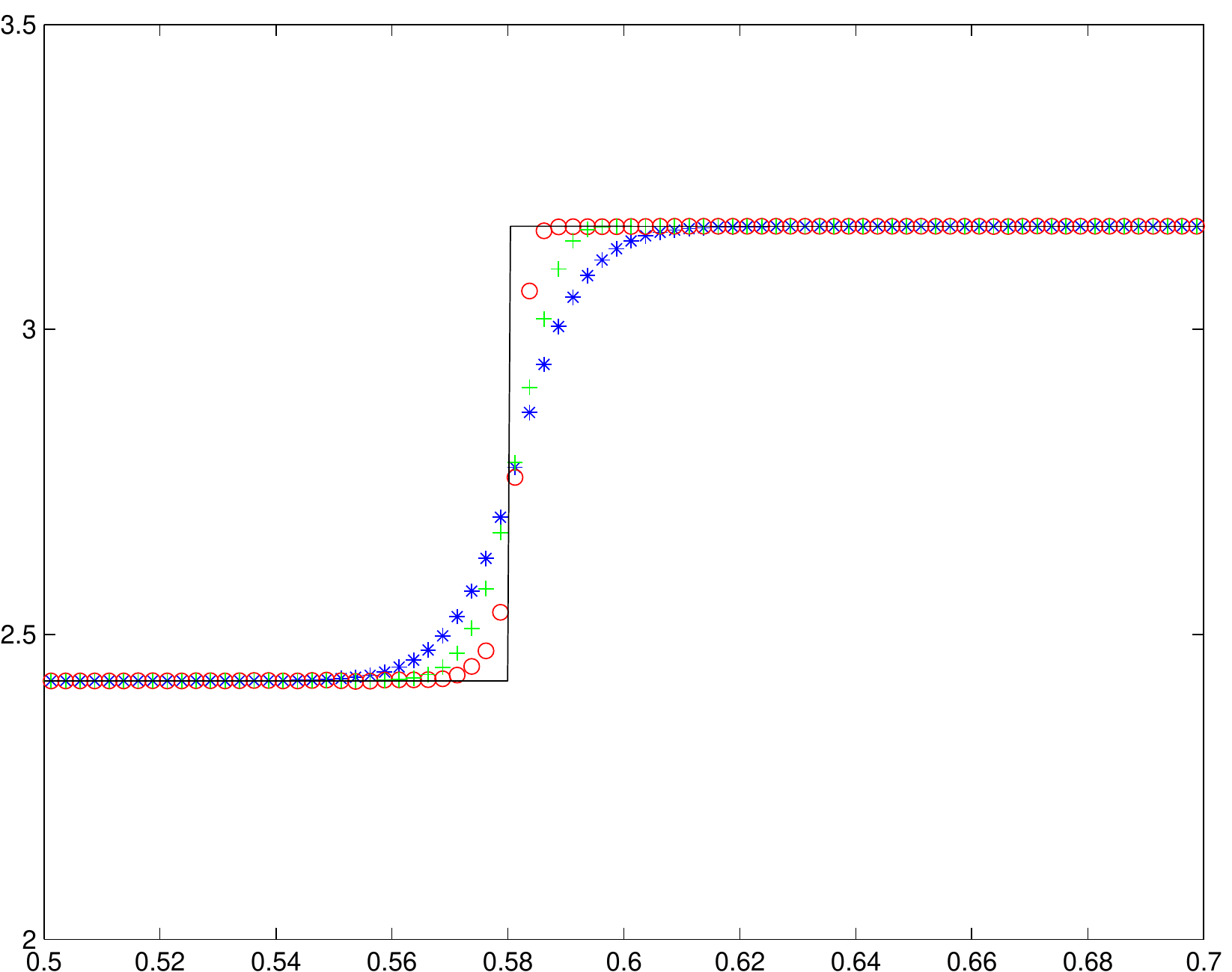}
 }
 \subfigure[$u_1$]{
  \includegraphics[width=0.45\textwidth]{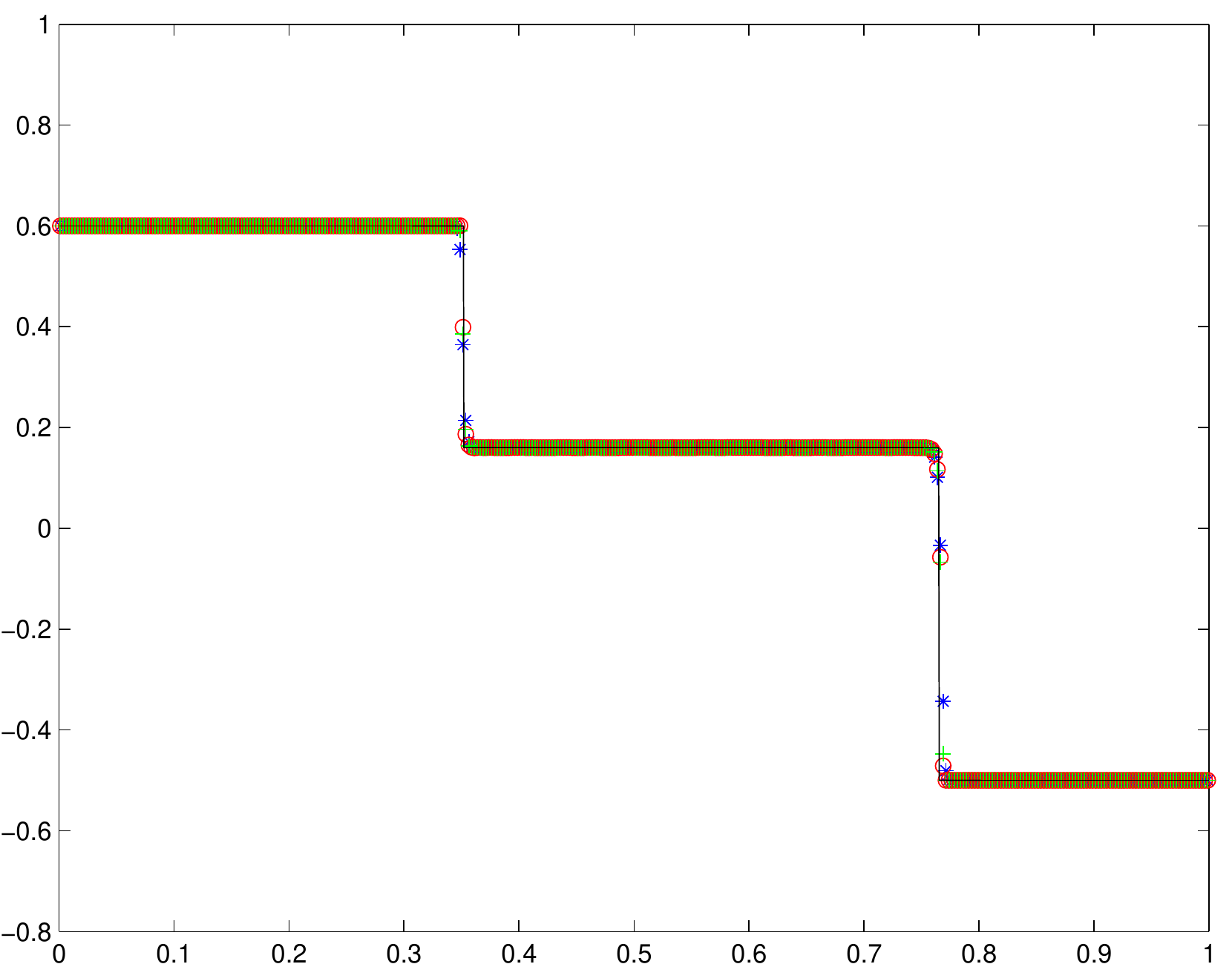}
 }
 \subfigure[$p$]{
  \includegraphics[width=0.45\textwidth]{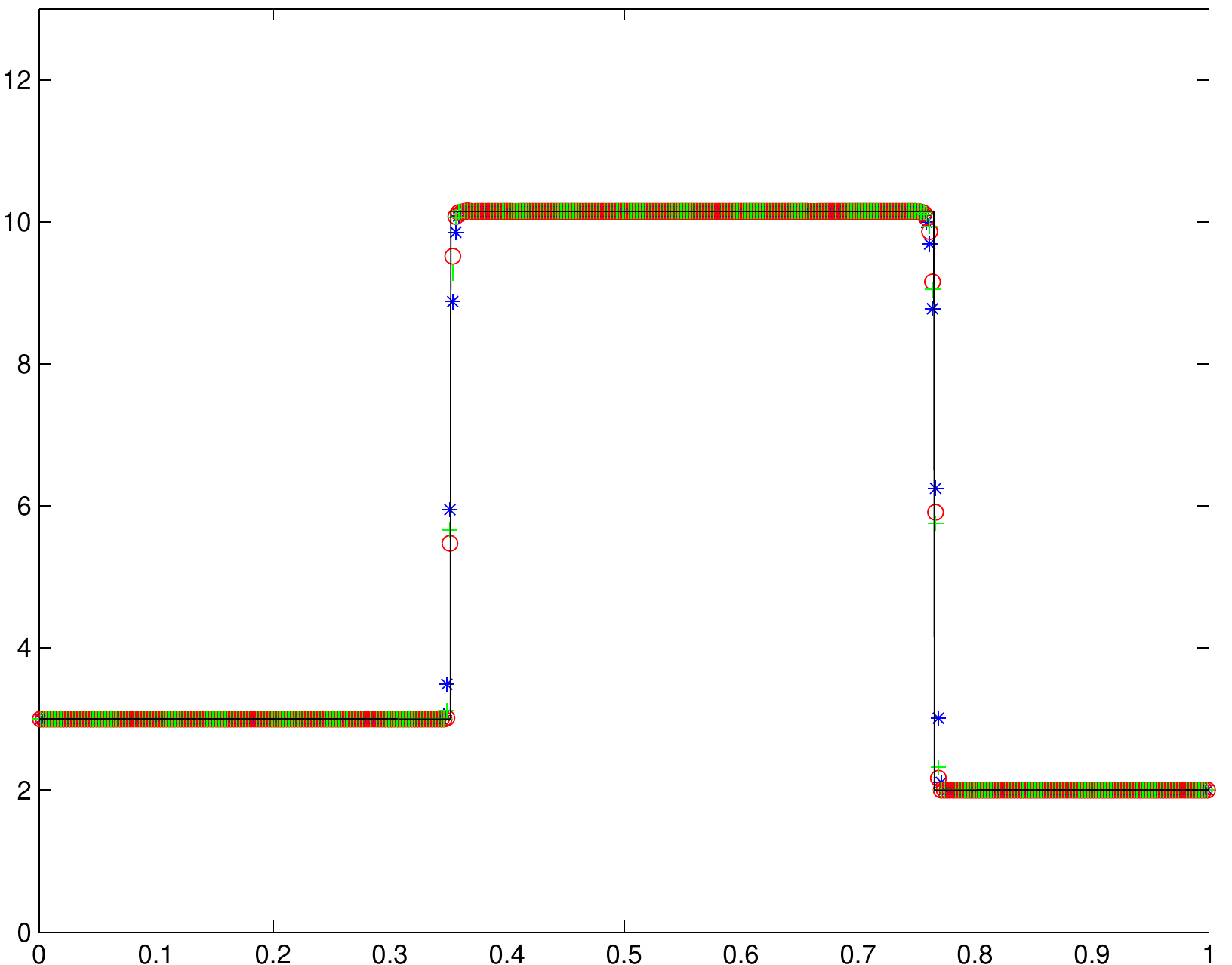}
 }
 \caption{Example \ref{ex:1DRP1}: The solutions {at} $t=0.5$  obtained by the BGK scheme (``{$\circ$}"),
 the KFVS scheme (``{$*$}") and the BGK-type scheme (``{$+$}") with 400 uniform cells.}
 \label{fig:1DRP1}
\end{figure}

\begin{figure}[h]
 \centering
 \subfigure[$ \rho $]{
  \includegraphics[width=0.45\textwidth]{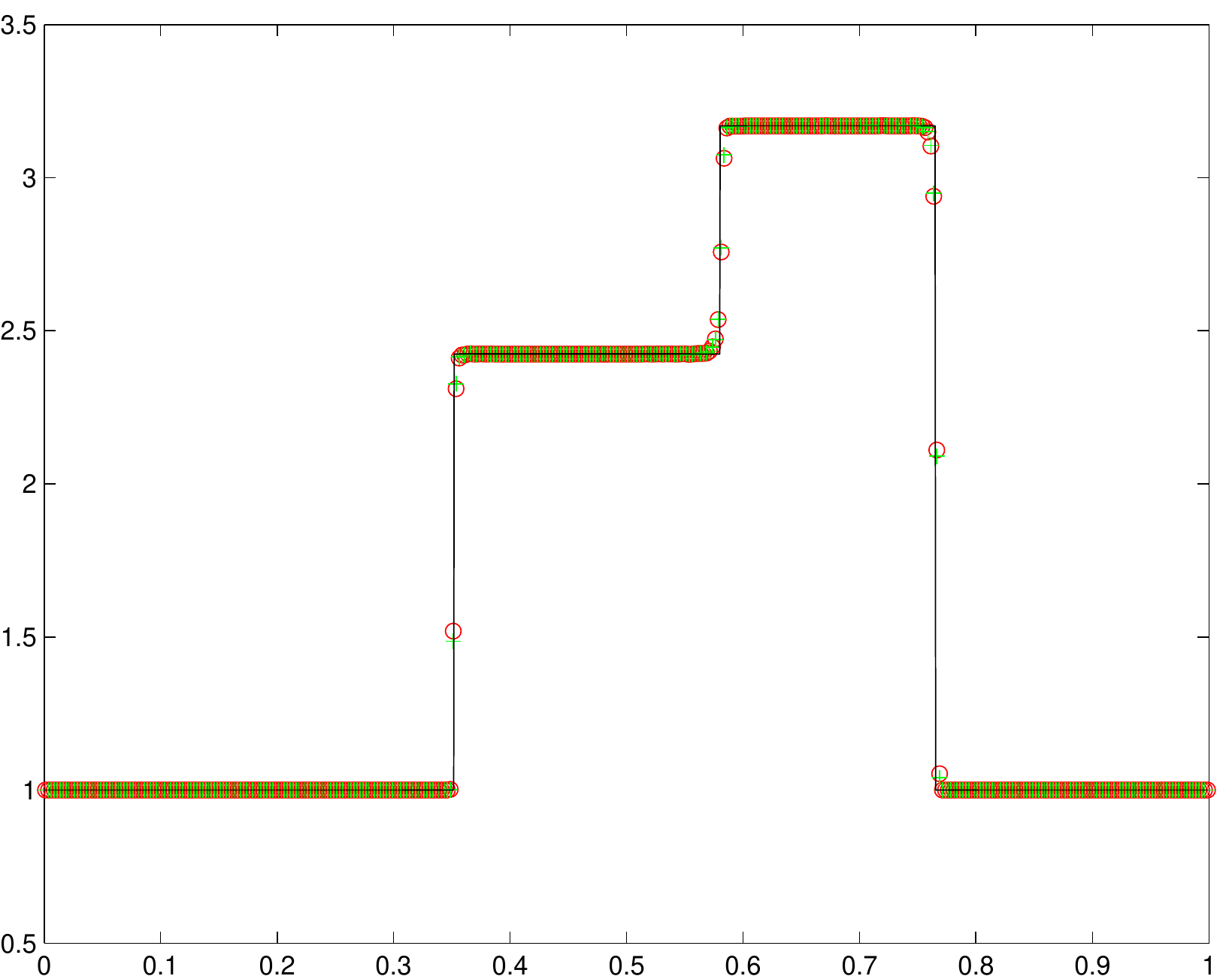}
 }
 \subfigure[close-up of $\rho$]{
  \includegraphics[width=0.45\textwidth]{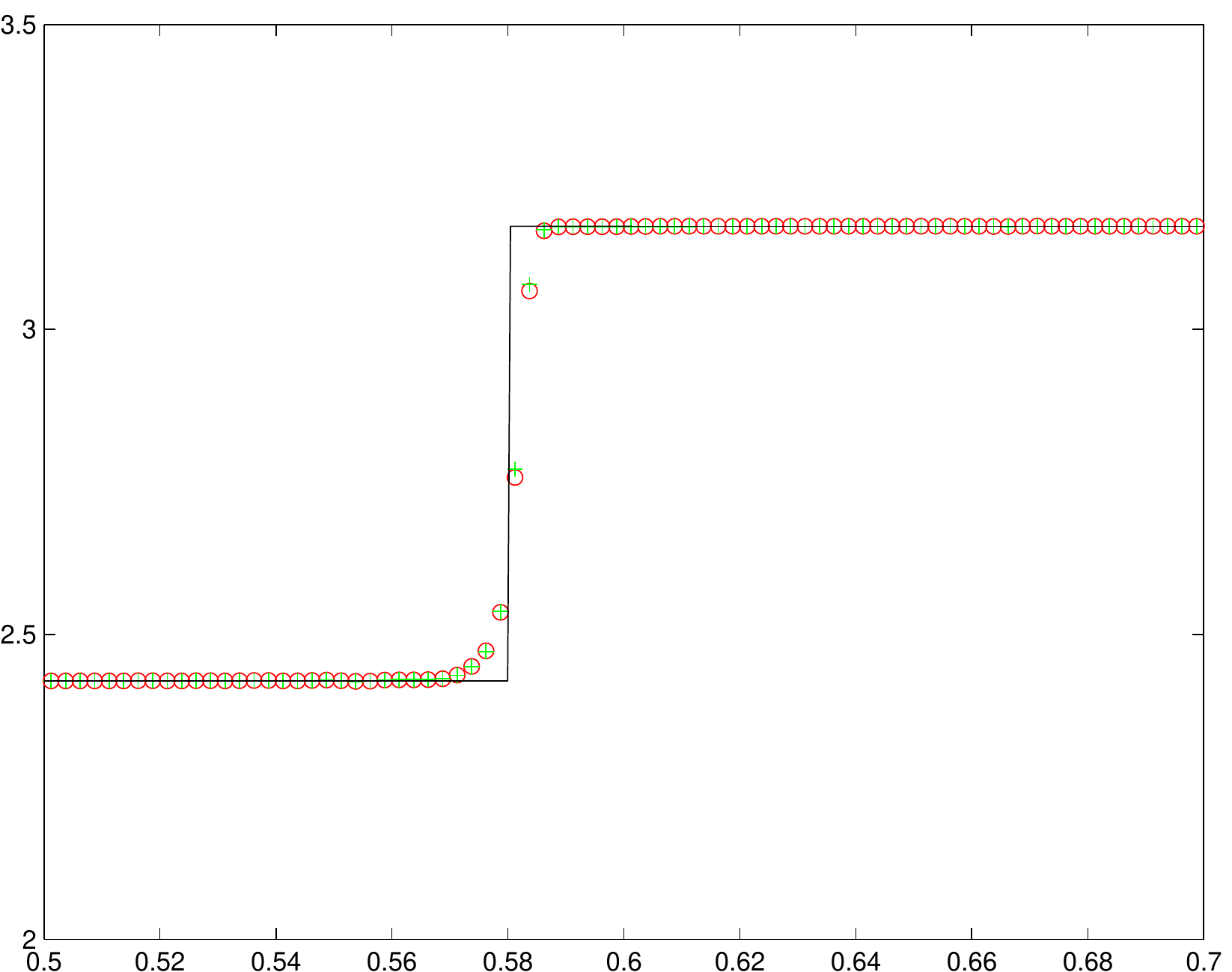}
 }
 \subfigure[$u_1$]{
  \includegraphics[width=0.45\textwidth]{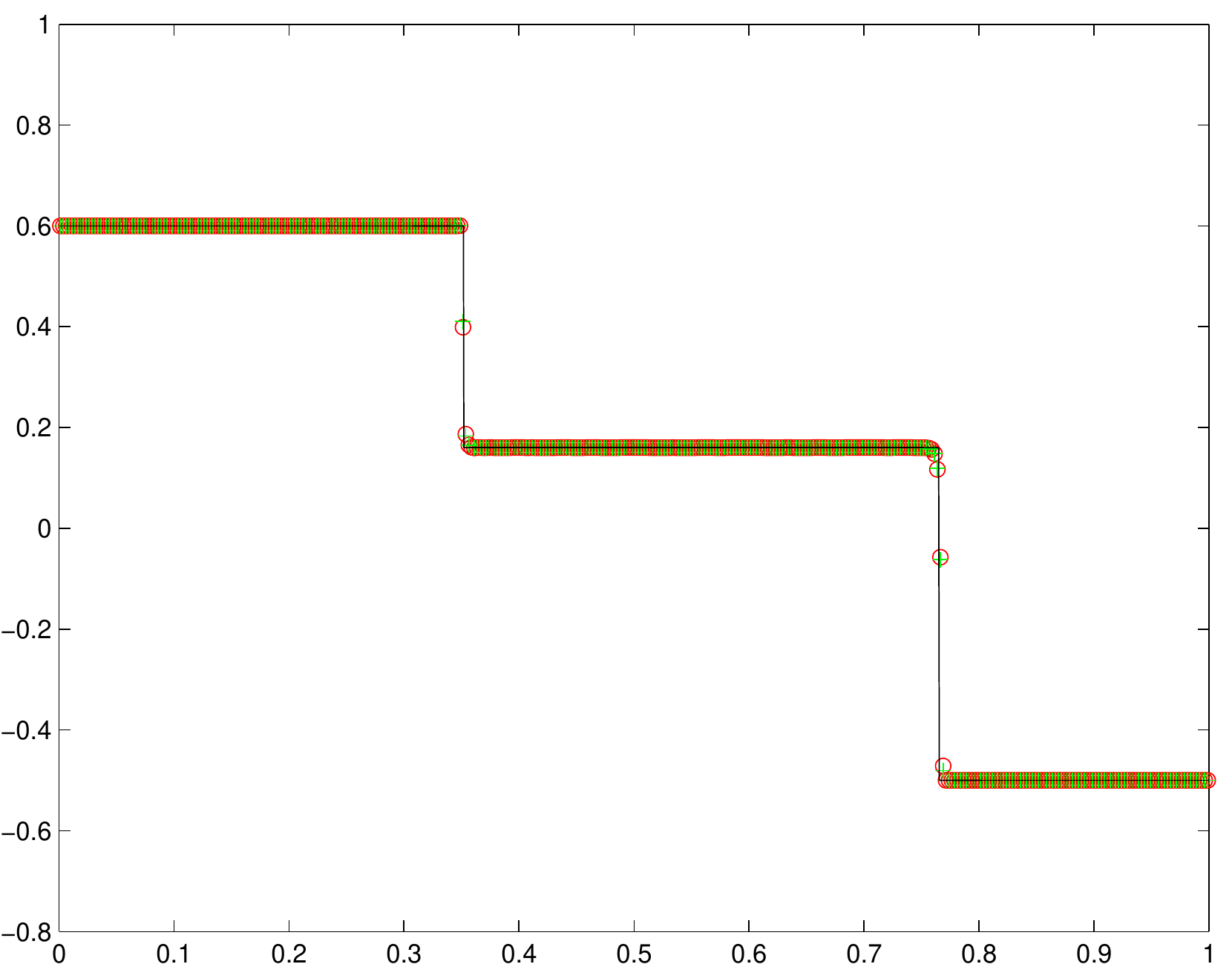}
 }
 \subfigure[$p$]{
  \includegraphics[width=0.45\textwidth]{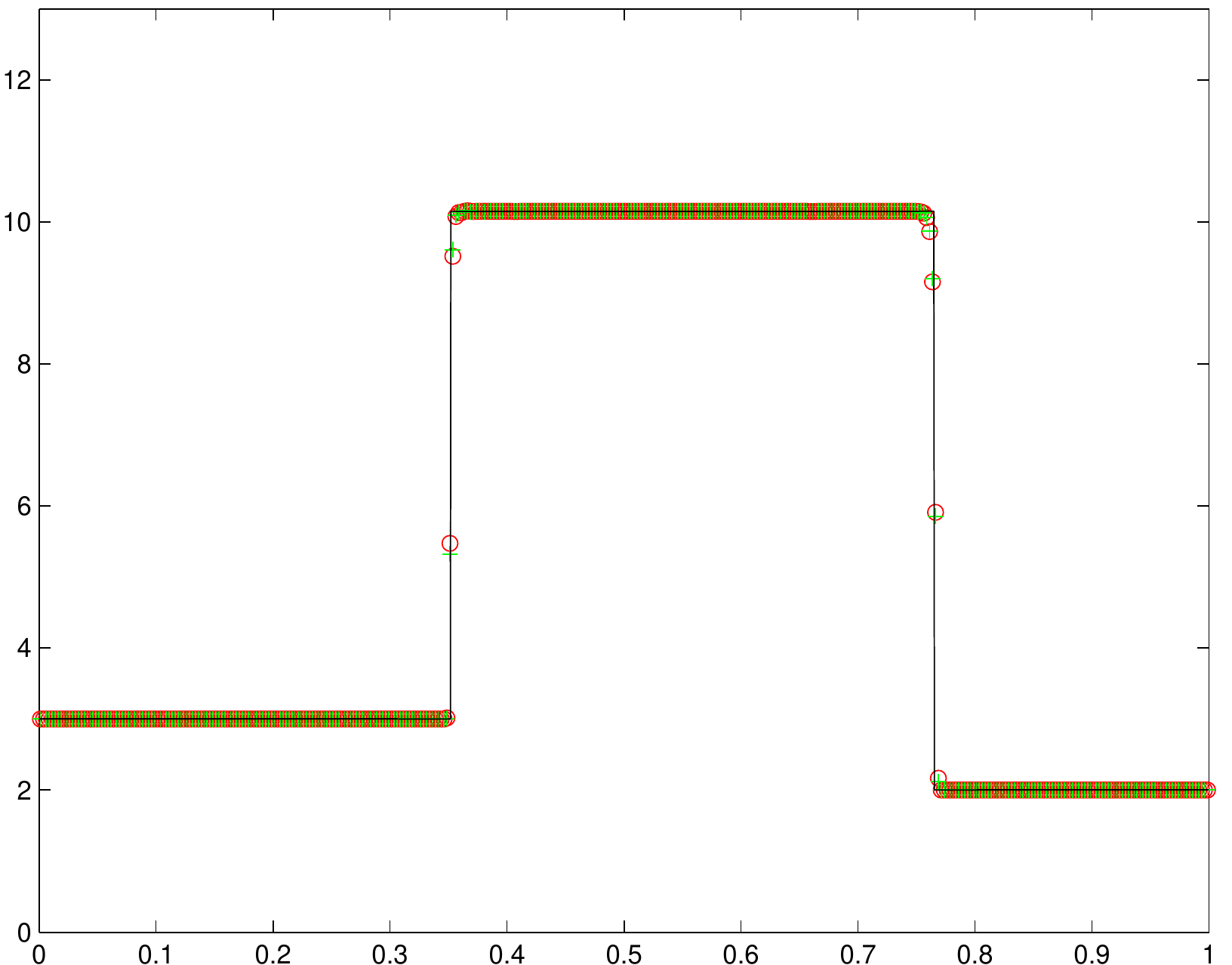}
 }
 \caption{Example \ref{ex:1DRP1}: The solutions {at} $t=0.5$  obtained by the BGK scheme (``{$\circ$}") and the sBGK scheme (``{$+$}") with 400 uniform cells.}
 \label{fig2:1DRP1}
\end{figure}

\clearpage
\begin{example}[Riemann problem III]\label{ex:1DRP2}\rm
 The initial conditions of this Riemann problem are
 \begin{equation*}
   \label{exeq:1DRP2}
  (\rho,u_1,p)(x,0)=
   \begin{cases}
     (5.0,0.0,10.0),& x<0.5,\\
     (1.0,0.0,0.5),& x>0.5.
   \end{cases}
 \end{equation*}
\end{example}
\begin{figure}[h]
\centering
\subfigure[$ \rho $]{
 \includegraphics[width=0.45\textwidth]{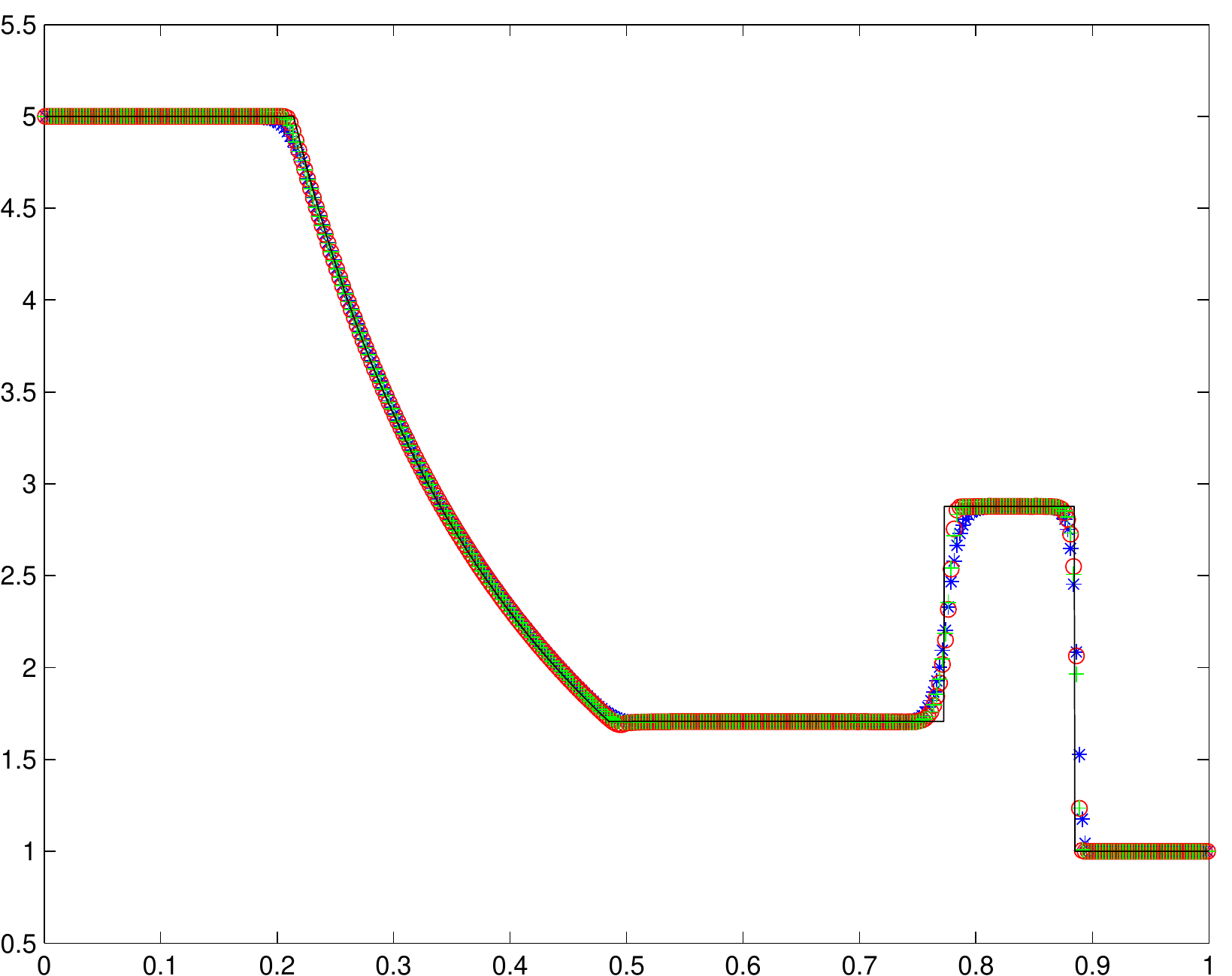}
}
\subfigure[close-up of $\rho$]{
 \includegraphics[width=0.45\textwidth]{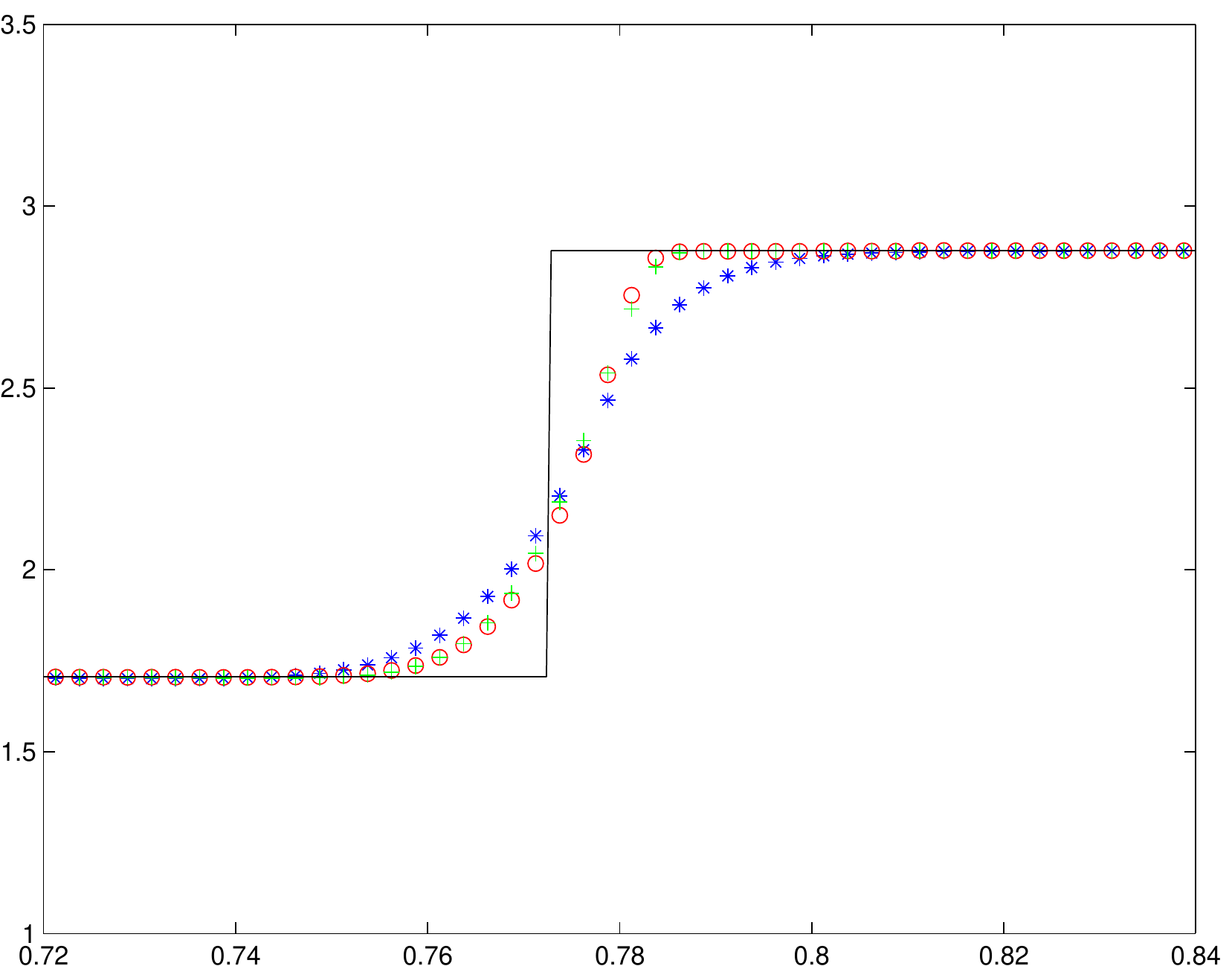}
}
\subfigure[$u_1$]{
 \includegraphics[width=0.45\textwidth]{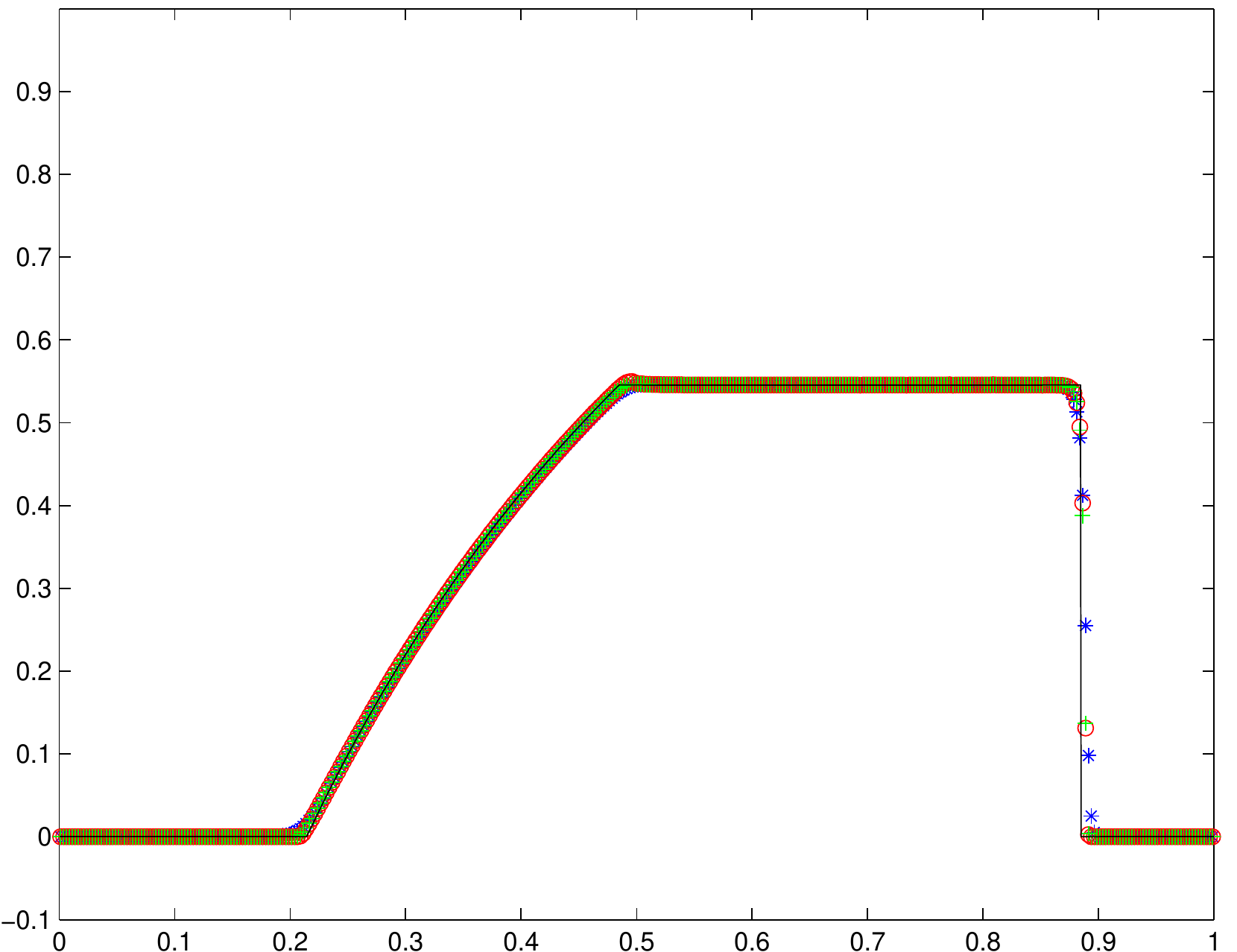}
}
\subfigure[$p$]{
 \includegraphics[width=0.45\textwidth]{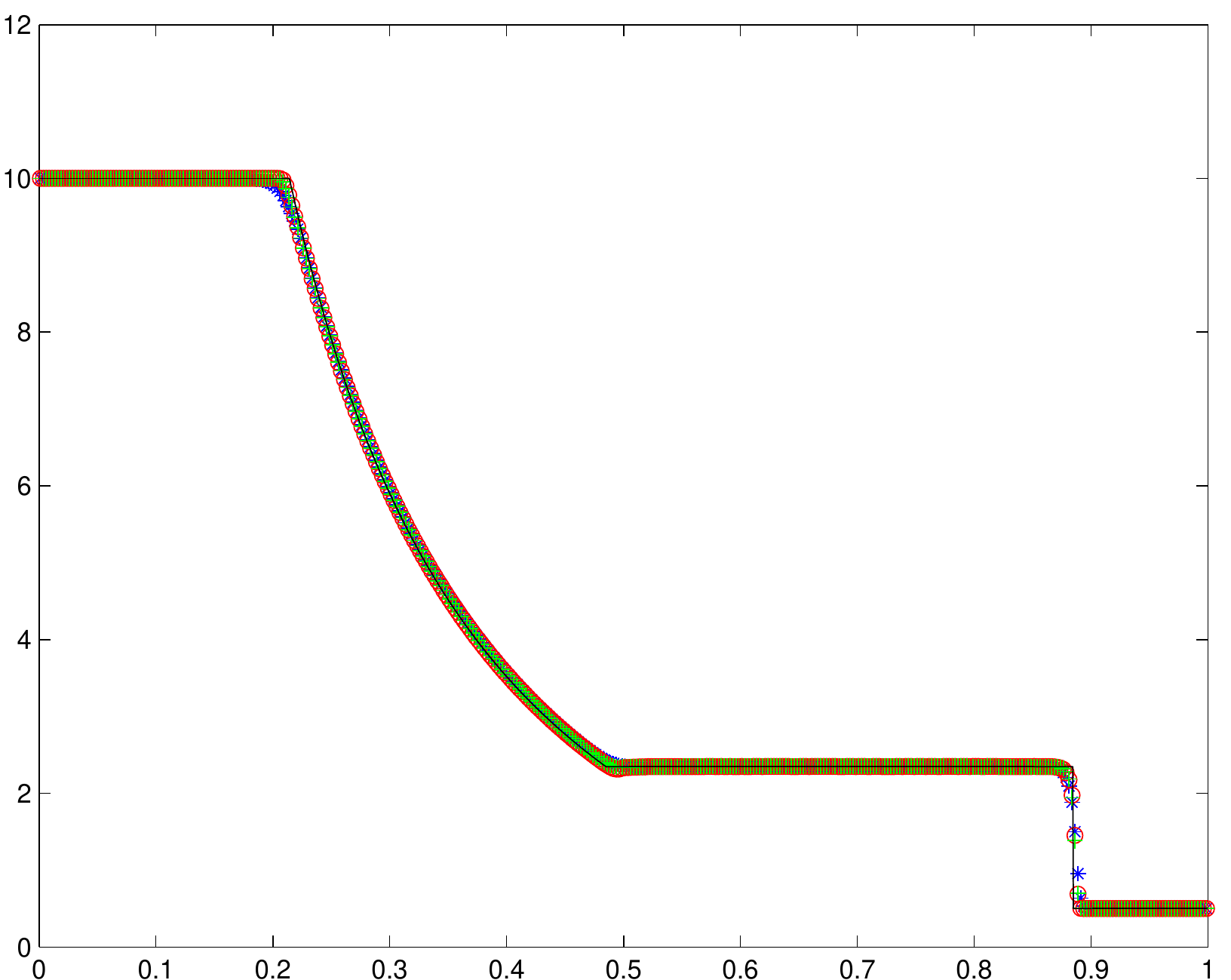}
}

\caption{Example \ref{ex:1DRP2}: The solutions at $t=0.5$ obtained by
using the BGK scheme (``{$\circ$}"), the KFVS scheme (``{$*$}") and the BGK-type scheme (``{$+$}") with 400 uniform cells.}
\label{fig:1DRP2}
\end{figure}

\begin{figure}[h]
\centering
\subfigure[$ \rho $]{
 \includegraphics[width=0.45\textwidth]{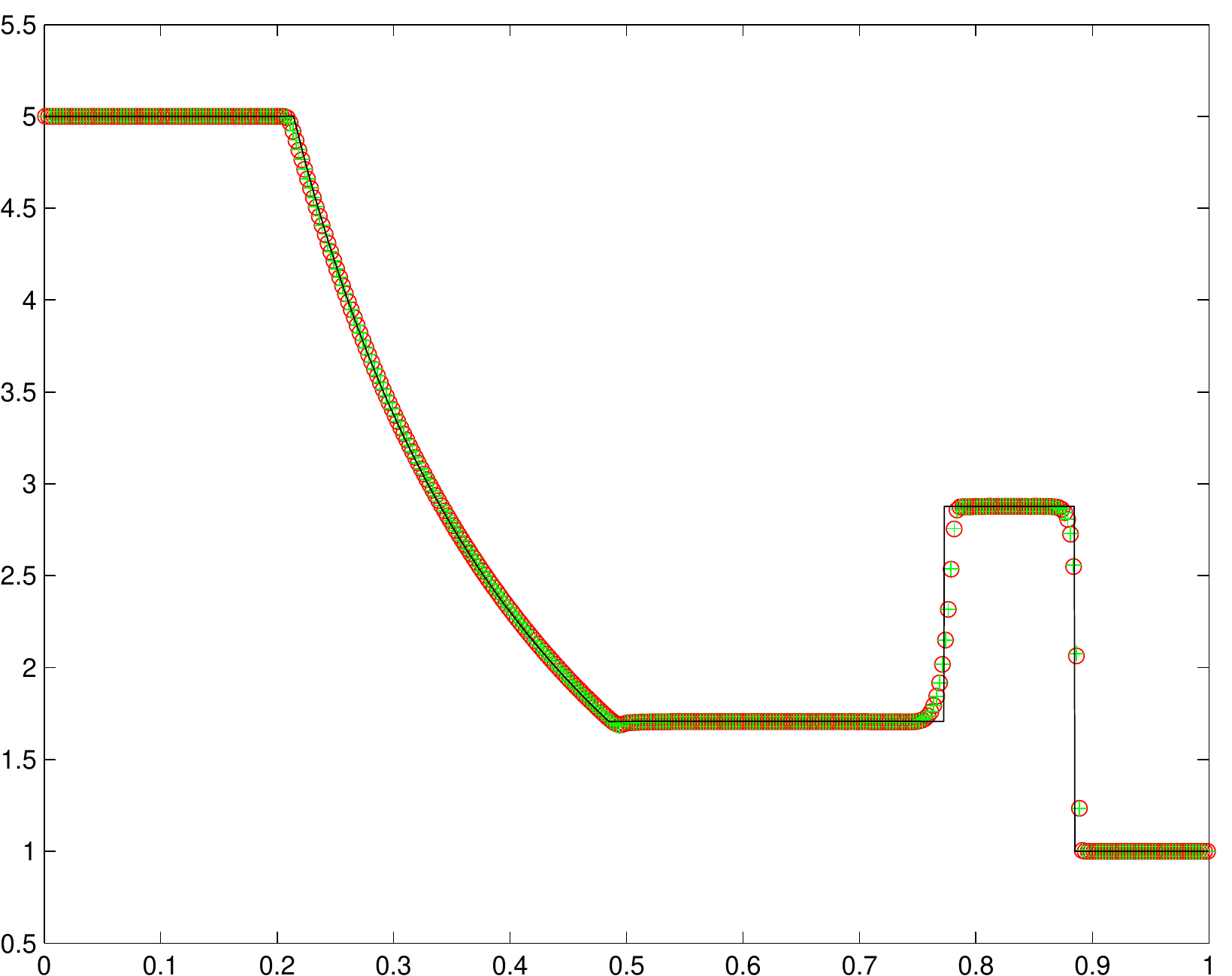}
}
\subfigure[close-up of $\rho$]{
 \includegraphics[width=0.45\textwidth]{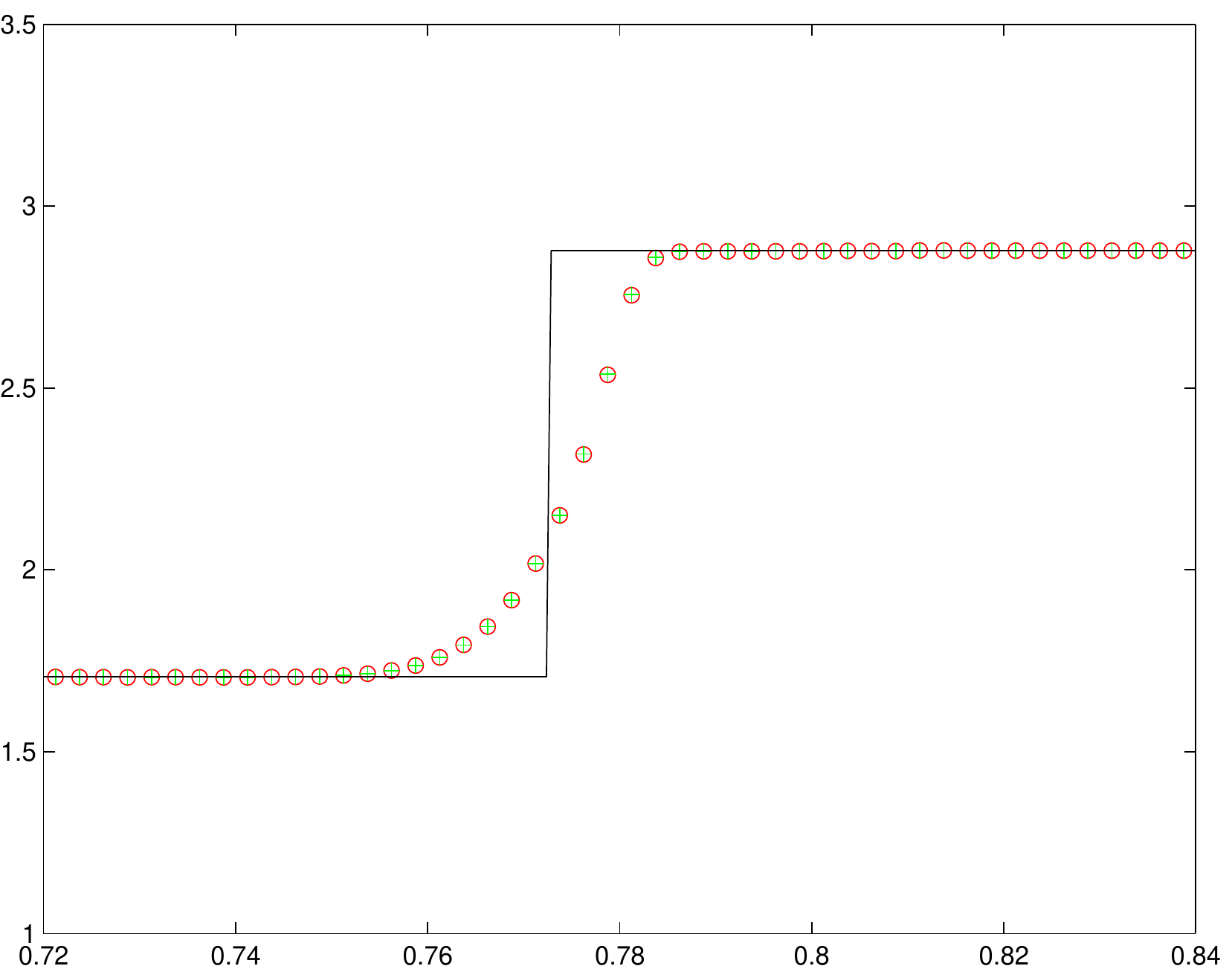}
}
\subfigure[$u_1$]{
 \includegraphics[width=0.45\textwidth]{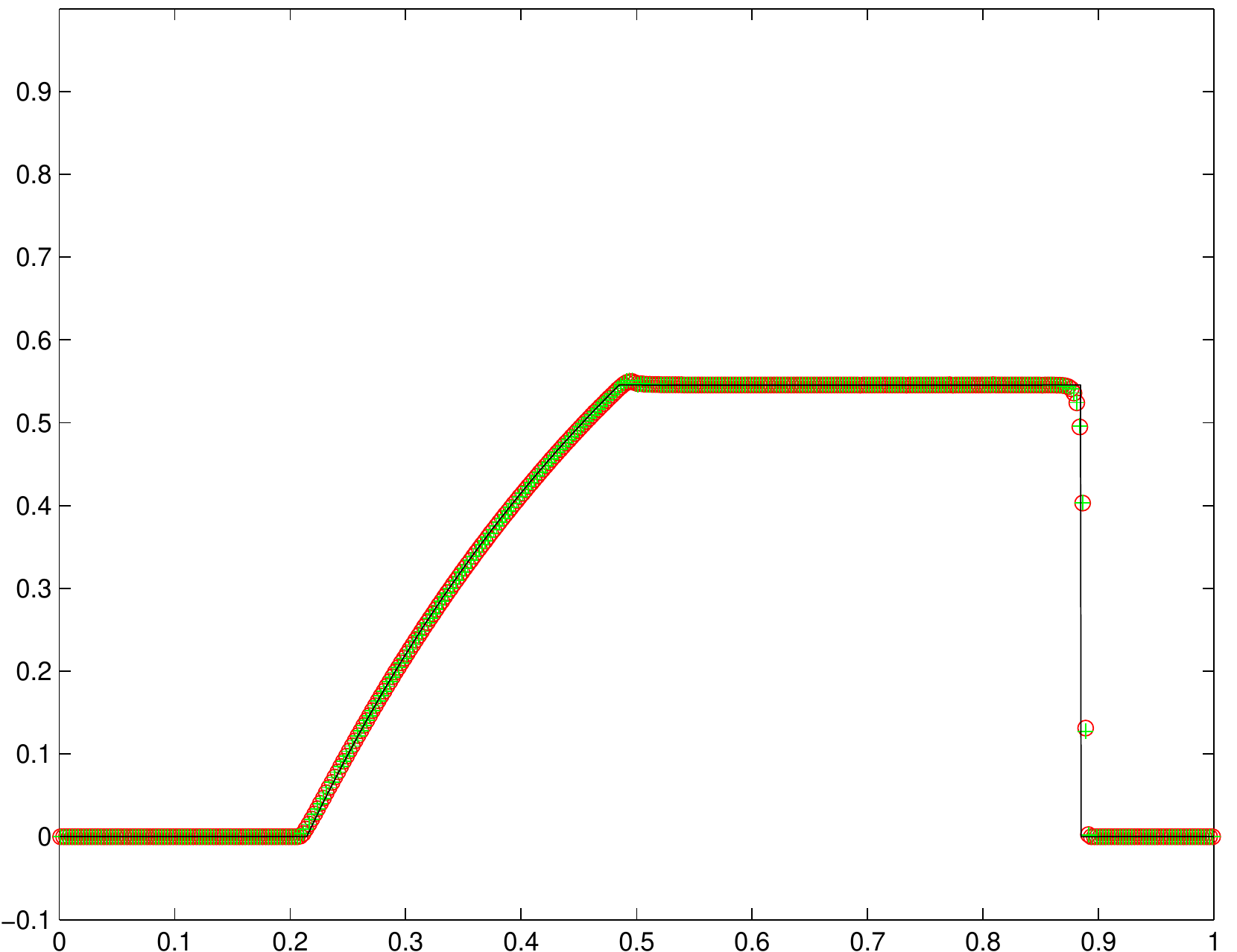}
}
\subfigure[$p$]{
 \includegraphics[width=0.45\textwidth]{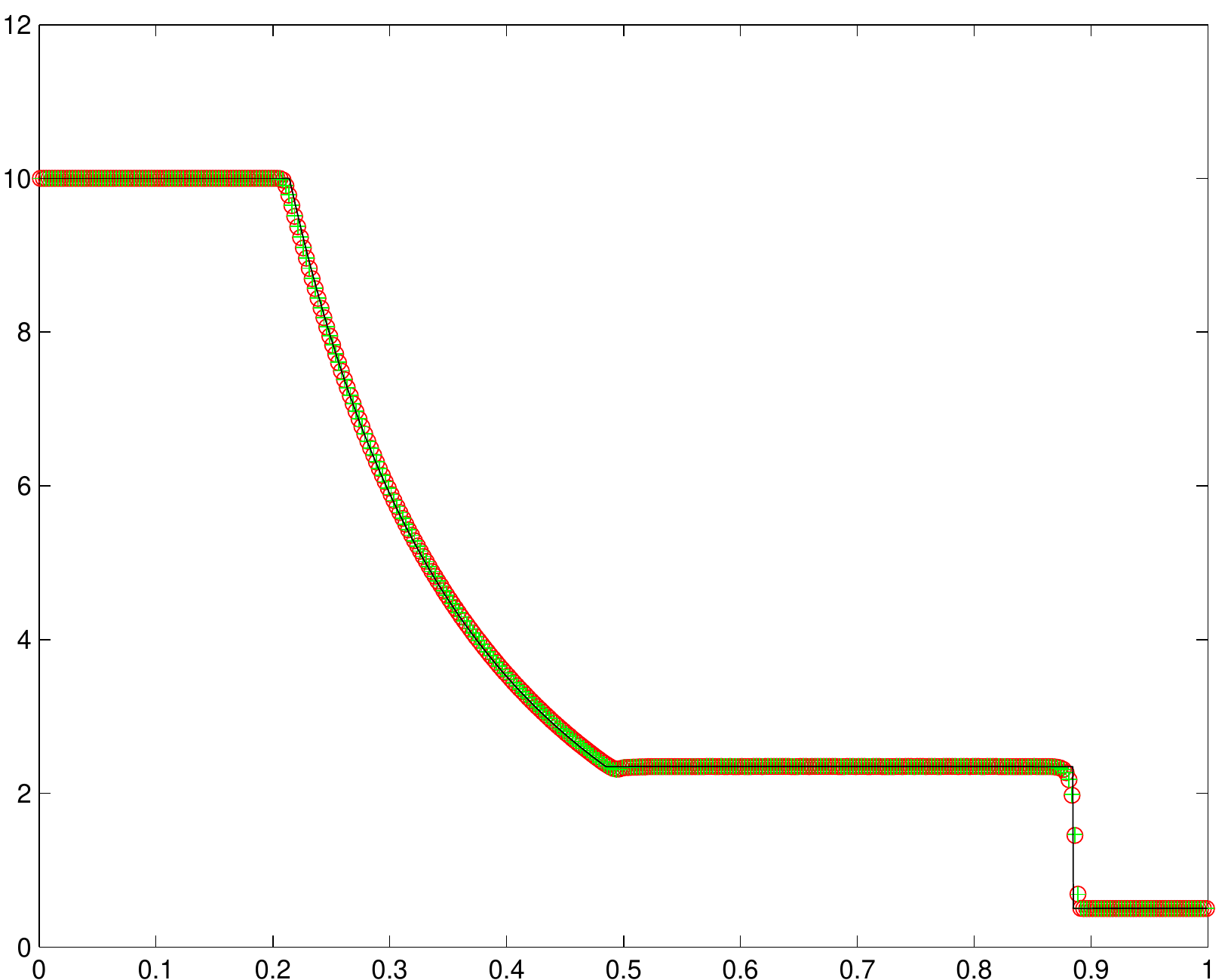}
}
\caption{Example \ref{ex:1DRP2}: The solutions at $t=0.5$ obtained by
 the BGK scheme (``{$\circ$}") and the sBGK scheme (``{$+$}") with 400 uniform cells.}
\label{fig22:1DRP2}
\end{figure}

Fig. \ref{fig:1DRP2}
shows the numerical solutions {at} $t=0.5$  obtained by the BGK scheme (``{$\circ$}"), the KFVS scheme (``{$*$}") and the BGK-type scheme ("{$+$}") with 400 uniform cells, where the solid line denotes the exact solution.
In this case, as the time increases, the initial discontinuity at $x=0.5$ is evolved into a left-moving rarefaction wave, a right-moving contact discontinuity and a right-moving shock wave. It is seen that the BGK scheme and BGK-type scheme apparently exhibit higher resolution for the contact discontinuity than the KFVS scheme, and the numerical solutions of the BGK scheme and BGK-type scheme  resolves the shock wave better than the KFVS scheme.
A comparison between the BGK and sBGK schemes given in Fig. \ref{fig22:1DRP2} shows that the sBGK performs as well as the  BGK scheme but is more efficient.

\clearpage
\begin{example}[Perturbed shock tube problem] \label{ex:sinewave}\rm
 The initial data are
 \begin{equation*}
 \label{exeq:sinewave}
(\rho,u_1,p)(x,0)=
 \begin{cases}
   (1.0,0.0,1.0),& x<0.5,\\
   ( \rho _r,0.0,0.1),& x>0.5,
 \end{cases}
\end{equation*}
where $ \rho_r=2. + 0.3\sin(50x)$.
It is a perturbed shock tube problem, which has widely been used to test the ability of the shock-capturing schemes in resolving non-relativistic small-scale flow features.
\end{example}

\begin{figure}[h]
 \centering
 \subfigure[$\rho $]{
 \includegraphics[width=0.31\textwidth]{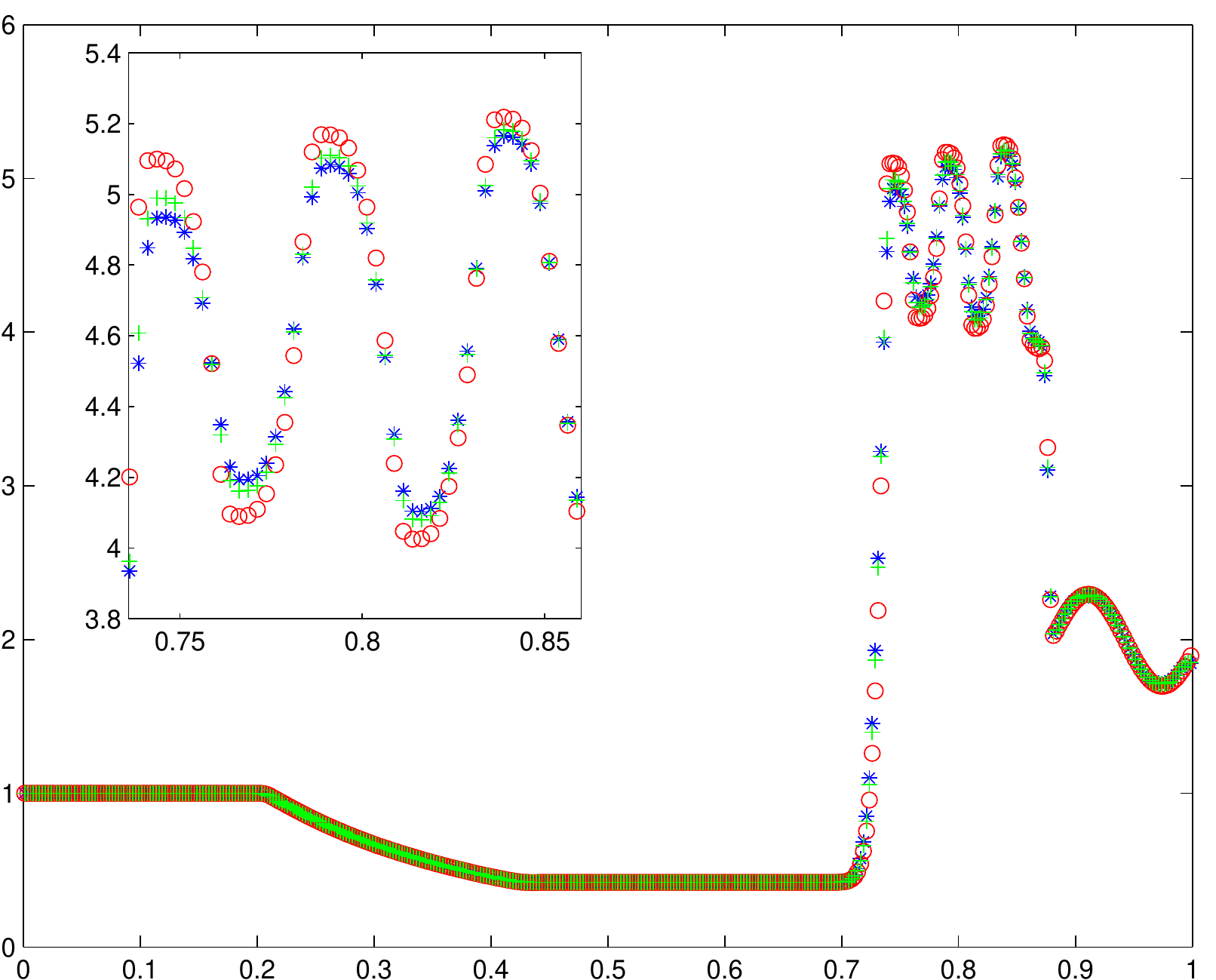}
 }
 \subfigure[$u_1$]{
 \includegraphics[width=0.31\textwidth]{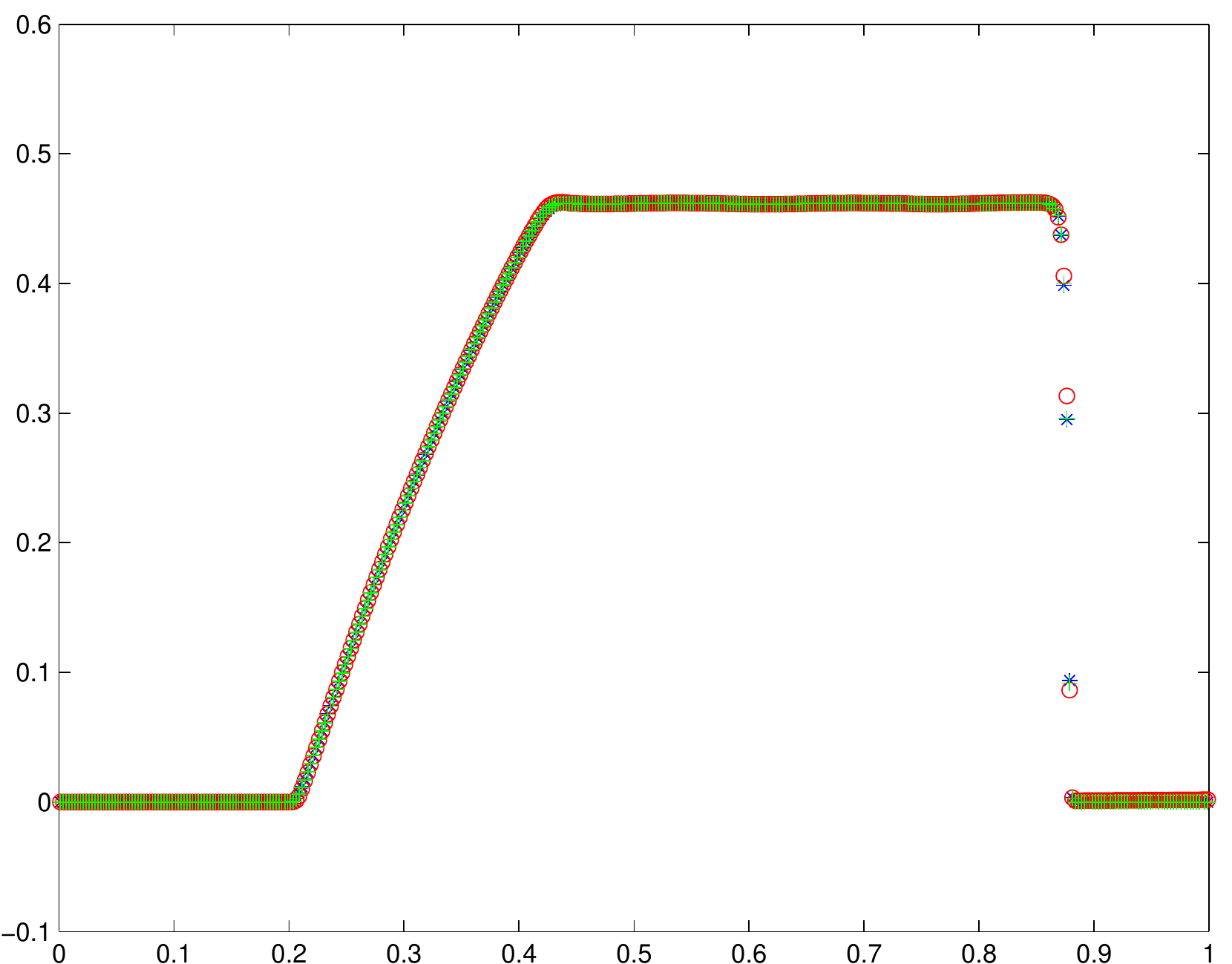}
 }
 \subfigure[$p$]{
 \includegraphics[width=0.31\textwidth]{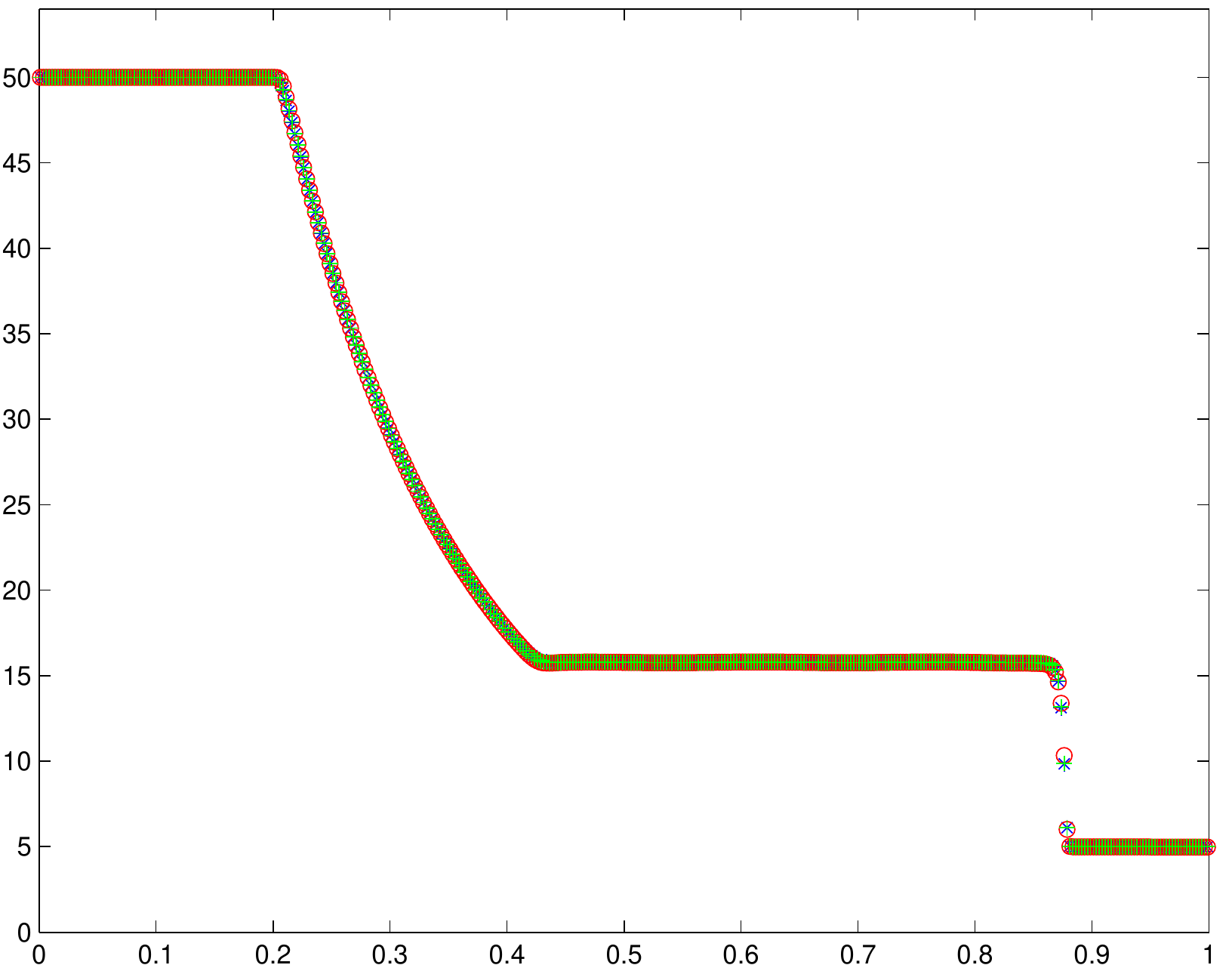}
 }
 \caption{Example \ref{ex:sinewave}: The solutions {at} $t=0.5$  obtained by the BGK scheme (``{$\circ$}"),
 the KFVS scheme (``{$*$}") and the BGK-type scheme ("{$+$}") with 400 uniform cells. }
 \label{fig:sinewave}
\end{figure}

\begin{figure}[h]
 \centering
 \subfigure[$\rho $]{
 \includegraphics[width=0.3\textwidth]{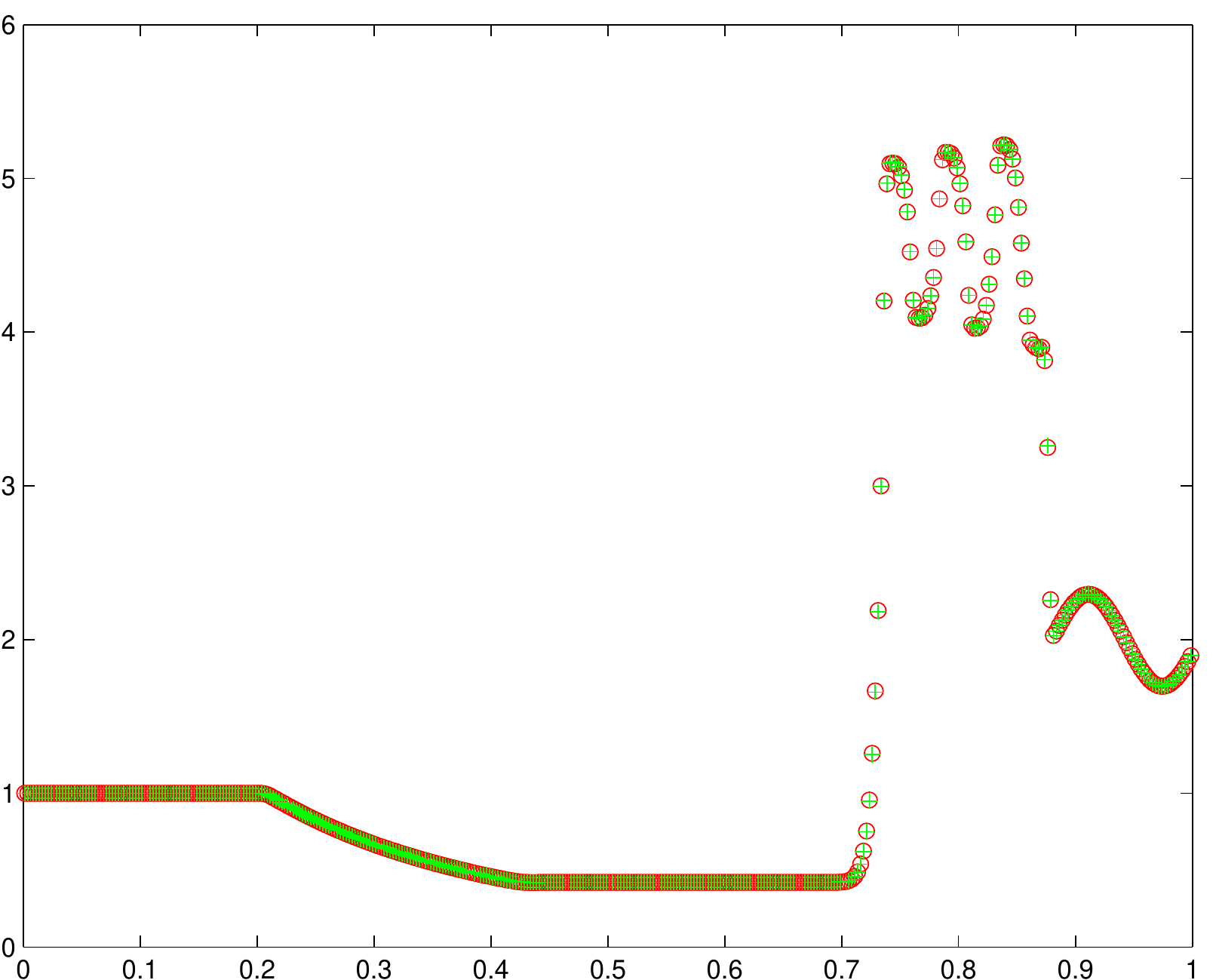}
 }
 \subfigure[$u_1$]{
 \includegraphics[width=0.3\textwidth]{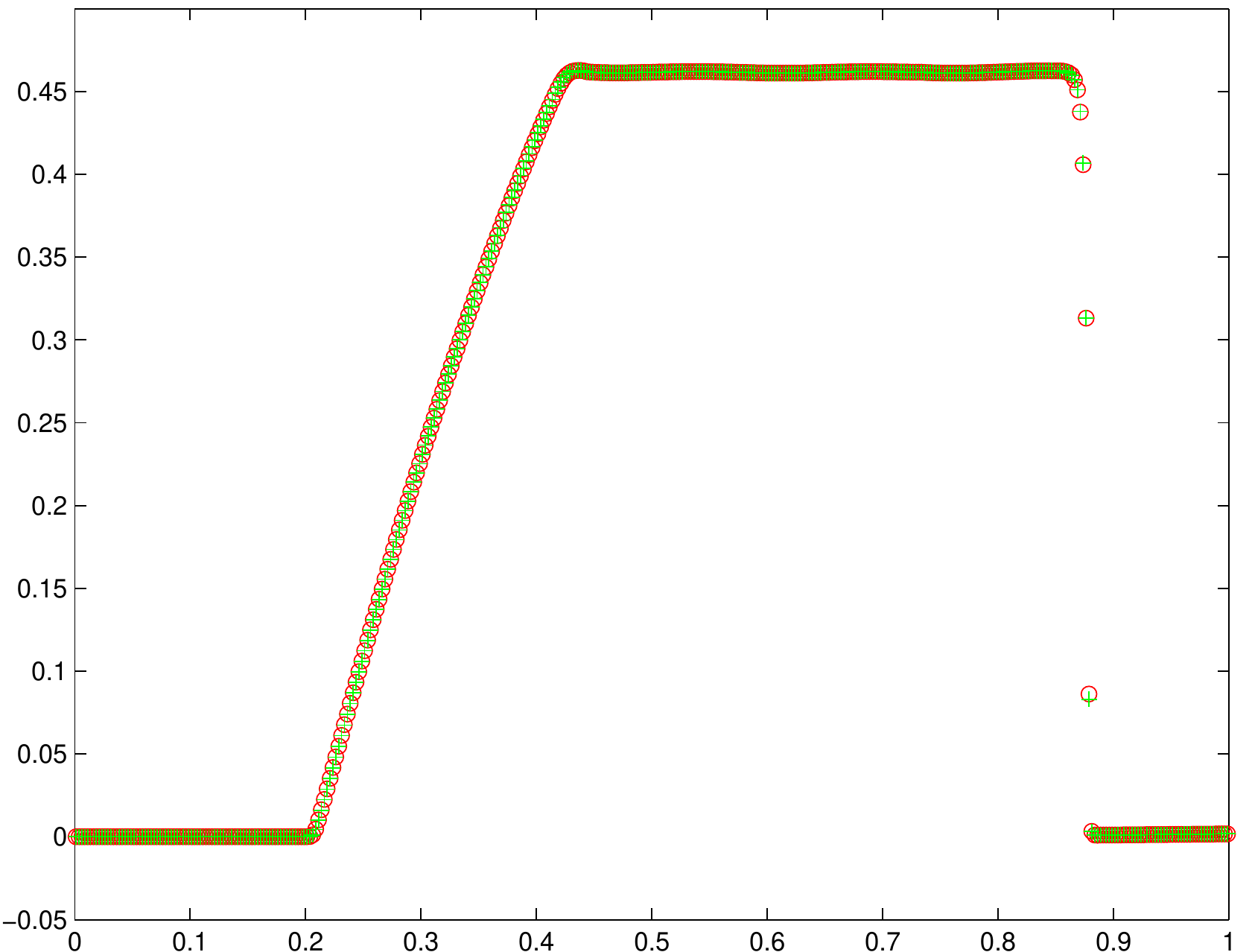}
 }
 \subfigure[$p$]{
 \includegraphics[width=0.3\textwidth]{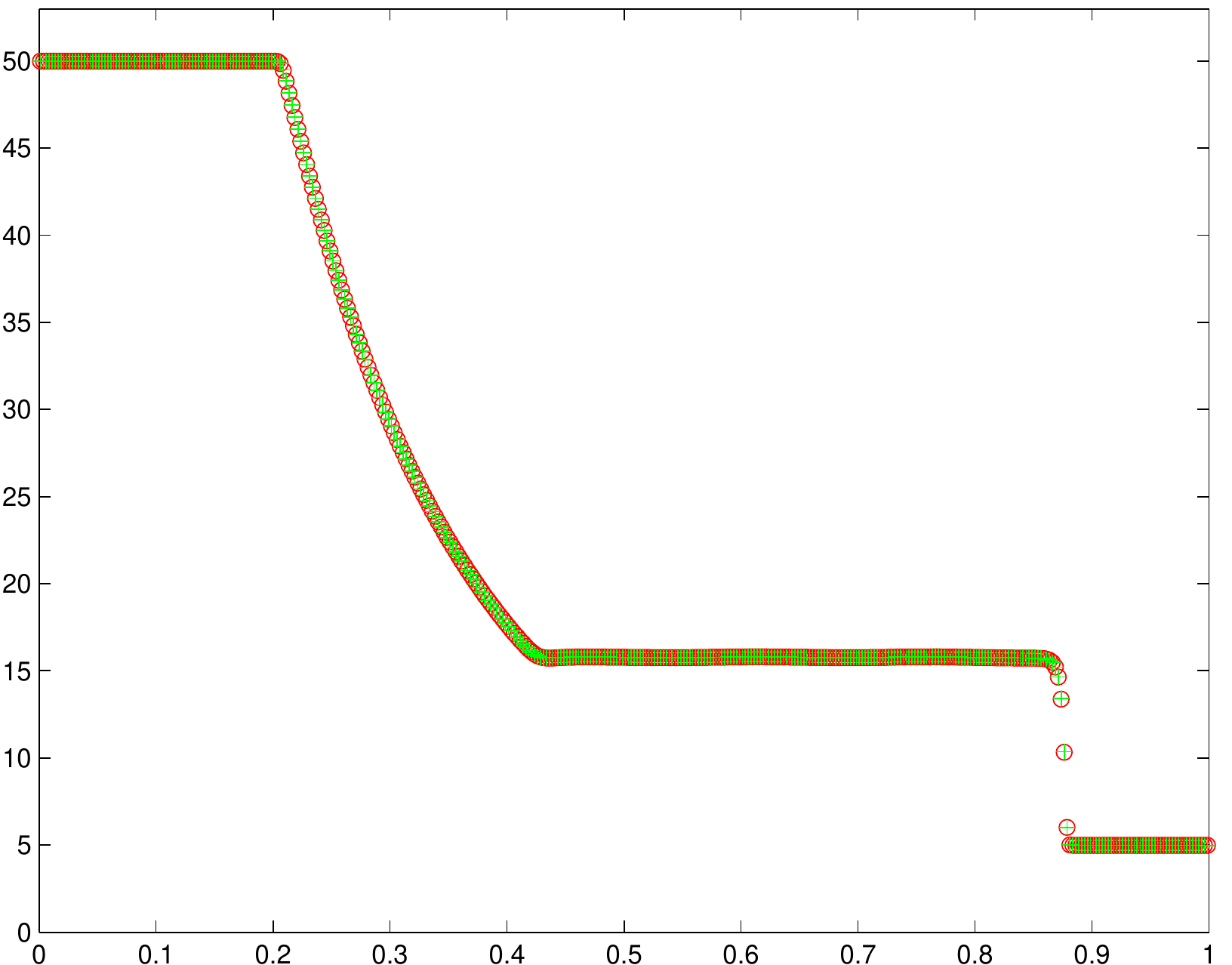}
 }
 \caption{Example \ref{ex:sinewave}: The solutions {at} $t=0.5$  obtained by the  BGK scheme (``{$\circ$}") and the sBGK scheme
 (``{$+$}") with 400 uniform cells. }
 \label{fig2:sinewave}
\end{figure}
As the time increases, the initial shock wave is moving into a sinusoidal density field, some complex but smooth structures are generated at the left hand side of the shock wave  when it interacts with the sine wave.
Fig. \ref{fig:sinewave} plots the numerical results at $t=0.5$ in the computational domain $\Omega=[0,1]$ obtained by using our BGK scheme (``{$\circ$}"), the KFVS scheme (``{$*$}") and the BGK-type scheme (``{$+$}") with 400 uniform cells. The numerical results show that the BGK scheme has better resolution for complex wave structures than the BGK-type scheme and the KFVS scheme. The results in Fig. \ref{fig2:sinewave} shows that the sBGK scheme can give the almost same resolution of the complex wave structure as  the BGK scheme.

\clearpage
\begin{example}[Collision of blast waves]\label{ex:blast}\rm
The last example is to simulate the collision of two strong relativistic blast waves. The initial data are taken as follows
 \begin{equation*}
   \label{eqex:blast}
  (\rho,u_1,p)(x,0)=
   \begin{cases}
     (1.0,0.0,100.0),& 0<x<0.1,\\
     (1.0,0.0,0.01),& 0.1<x<0.9,\\
     (1.0,0.0,10.0),& 0.9<x<1.0,
   \end{cases}
 \end{equation*}
\end{example}
and the reflecting boundary conditions are specified at the two ends of the computational domain $[0,1]$.
\begin{figure}[h]
 \centering
 \subfigure[$ \rho $]{
 \includegraphics[width=0.3\textwidth]{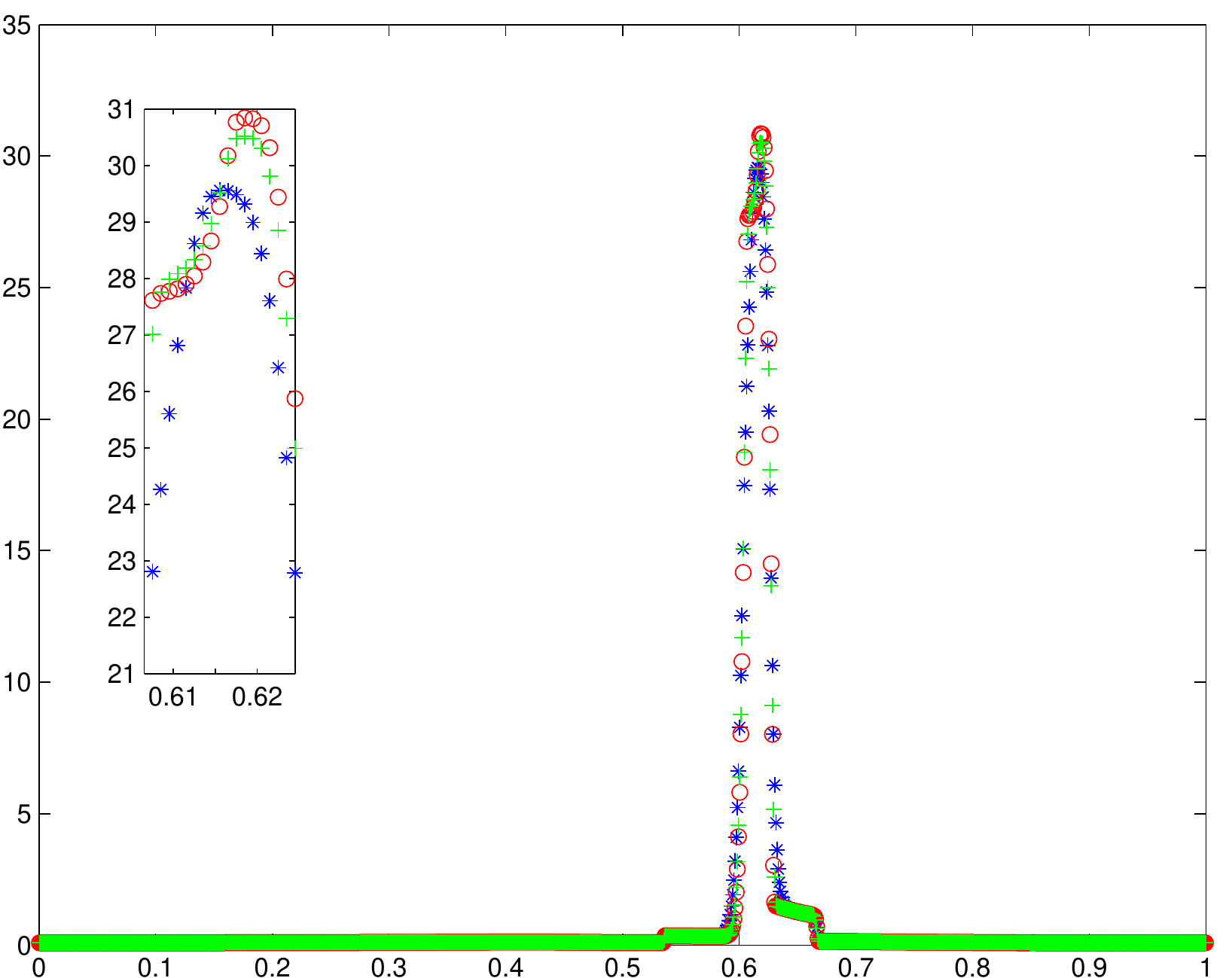}
 }
 \subfigure[$u_1$]{
 \includegraphics[width=0.3\textwidth]{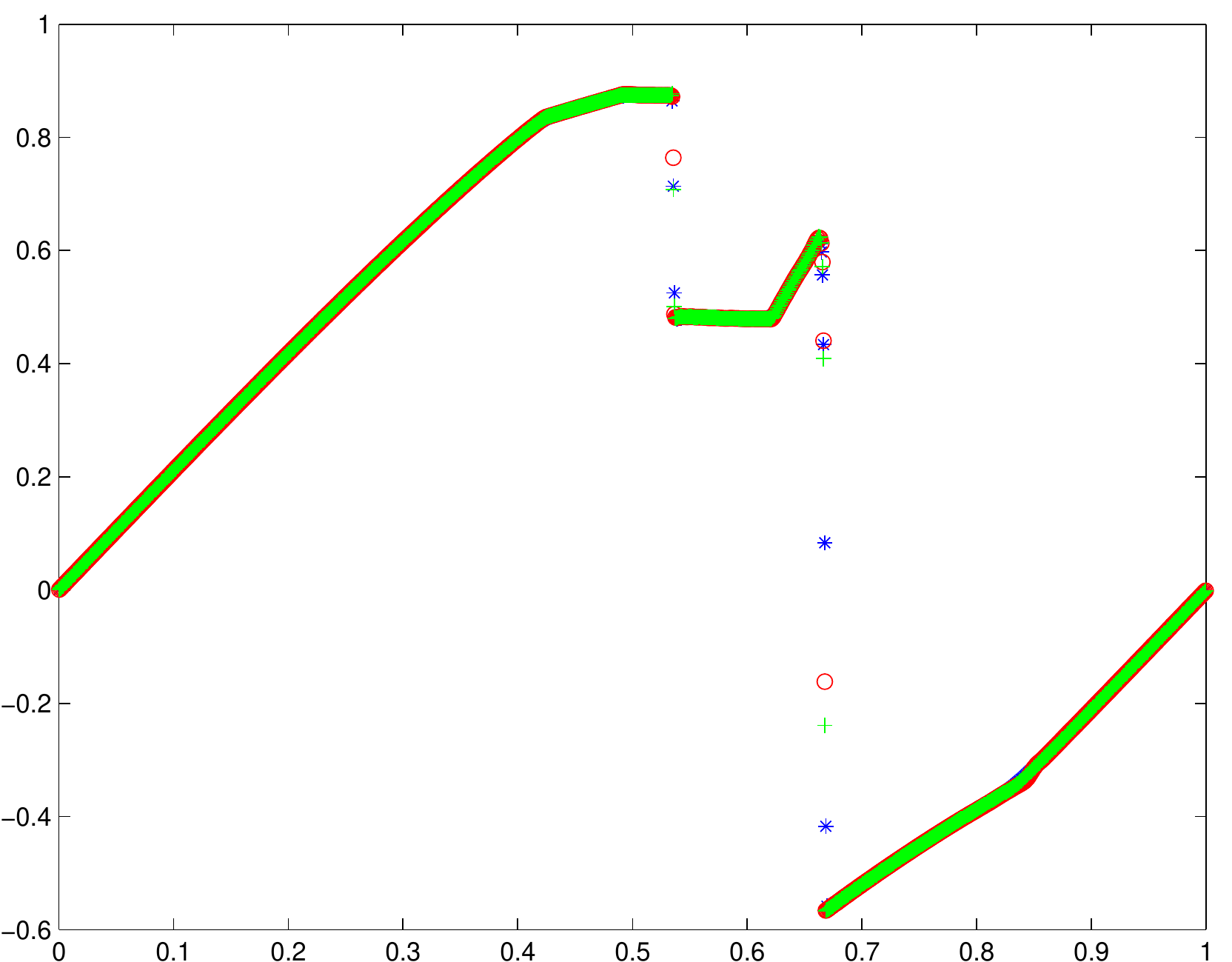}
 }
 \subfigure[$p$]{
 \includegraphics[width=0.3\textwidth]{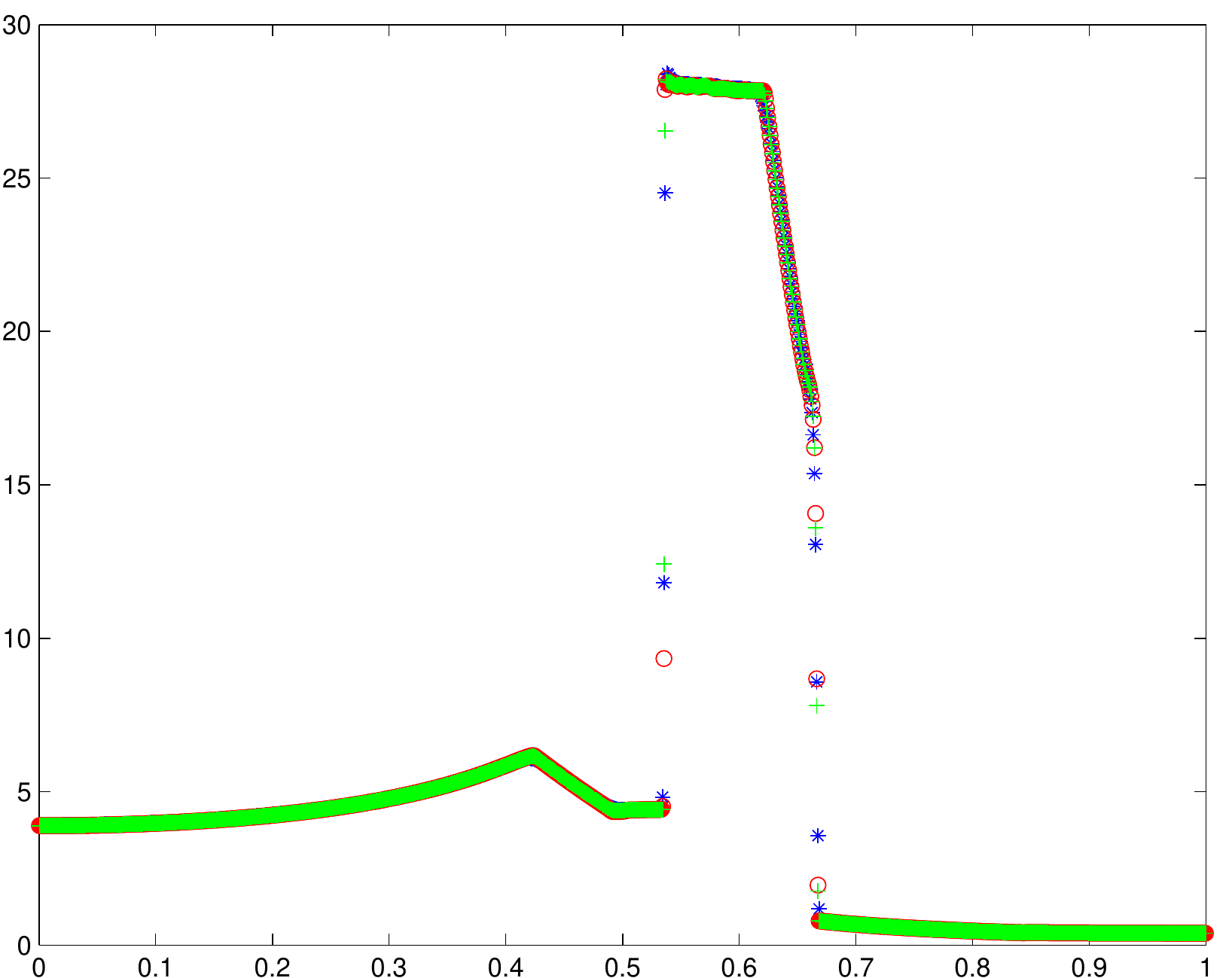}
 }
 \caption{Example \ref{ex:blast}: The results
 at $t=0.75$ obtained by the BGK scheme (``{$\circ$}"), the KFVS scheme (``{$*$}") and the BGK-type scheme (``{$+$}") with 1000  uniform cells.}
 \label{fig:blast}
\end{figure}

\begin{figure}[h]
 \centering
 \subfigure[$ \rho $]{
 \includegraphics[width=0.3\textwidth]{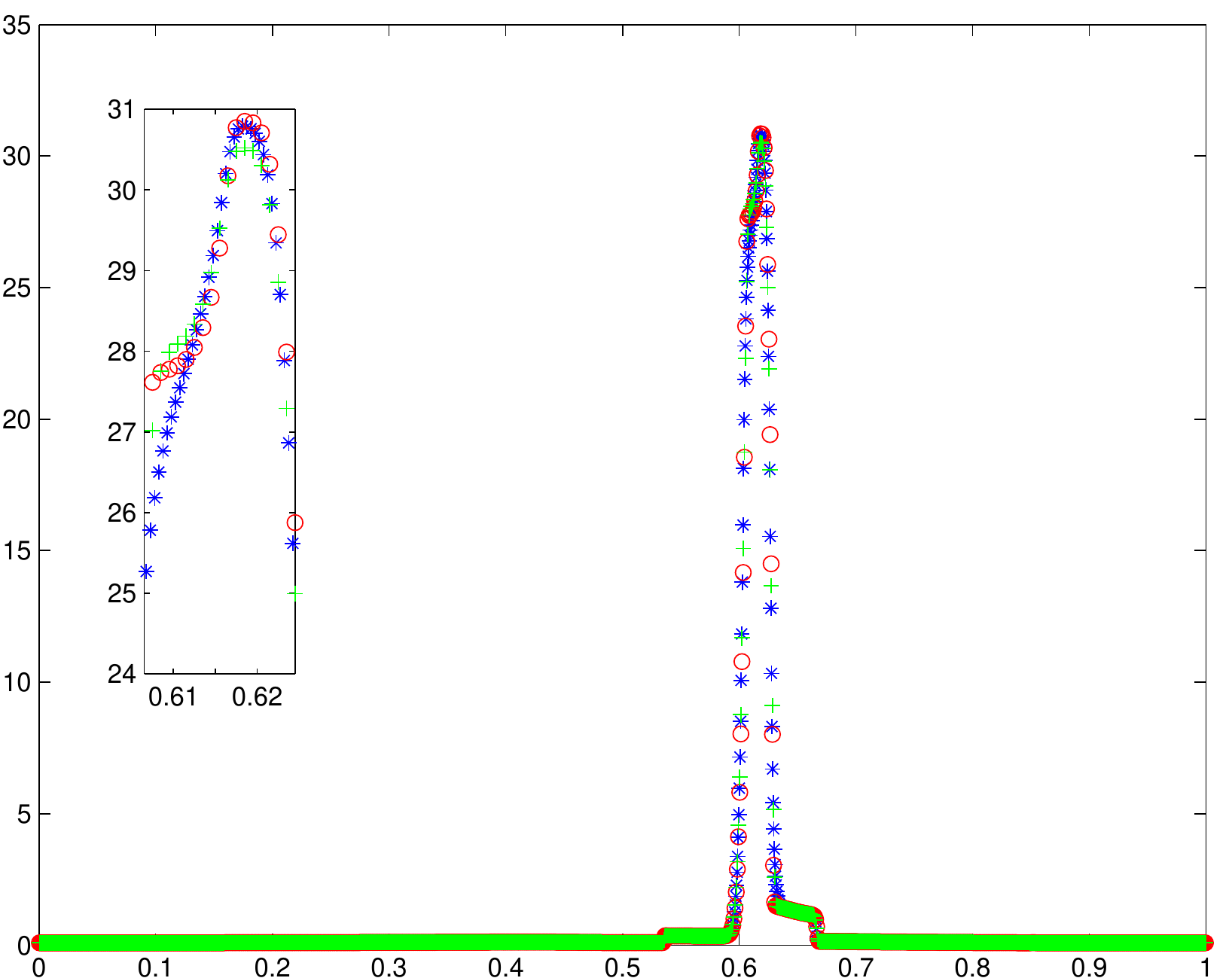}
 }
 \subfigure[$u_1$]{
 \includegraphics[width=0.3\textwidth]{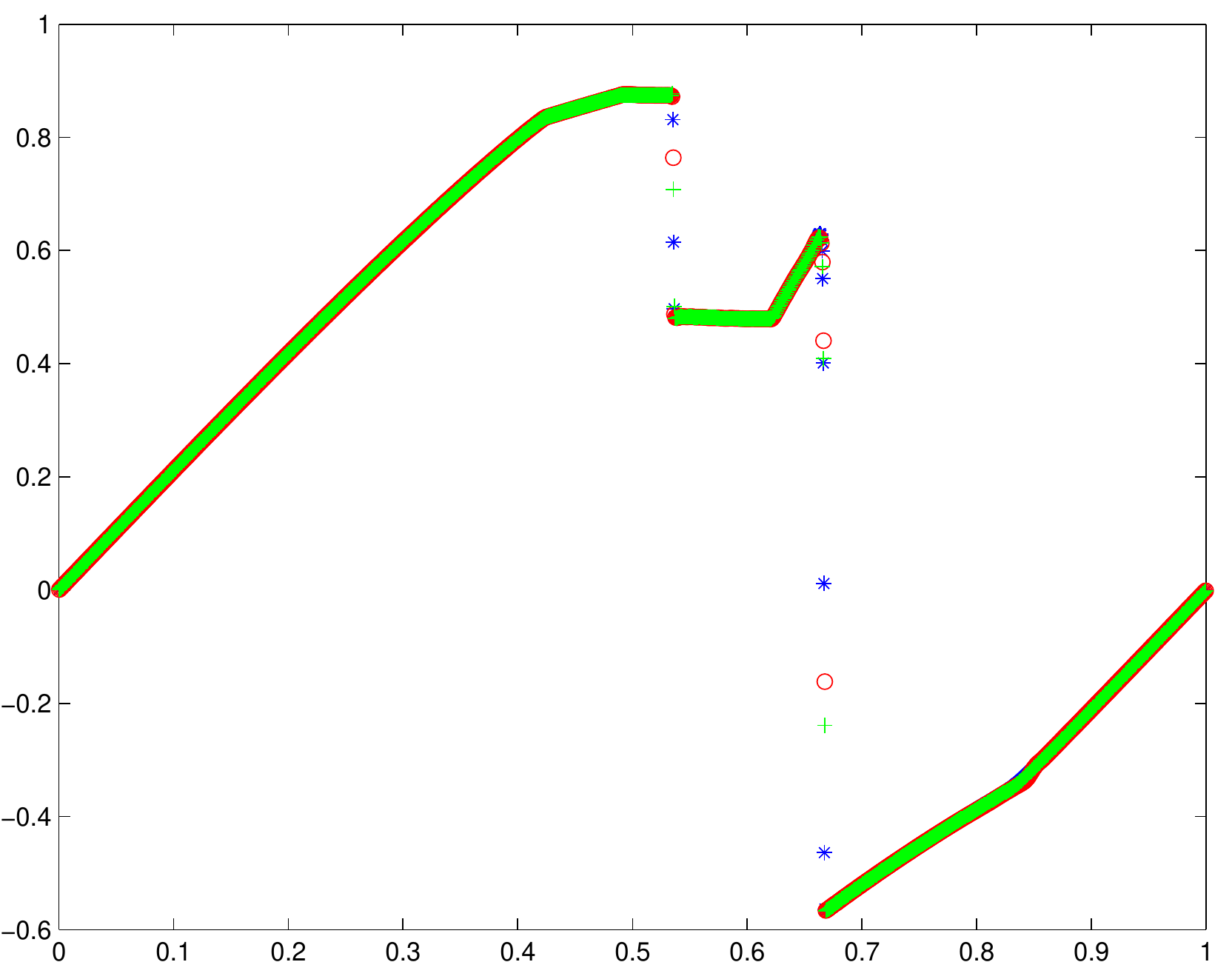}
 }
 \subfigure[$p$]{
 \includegraphics[width=0.3\textwidth]{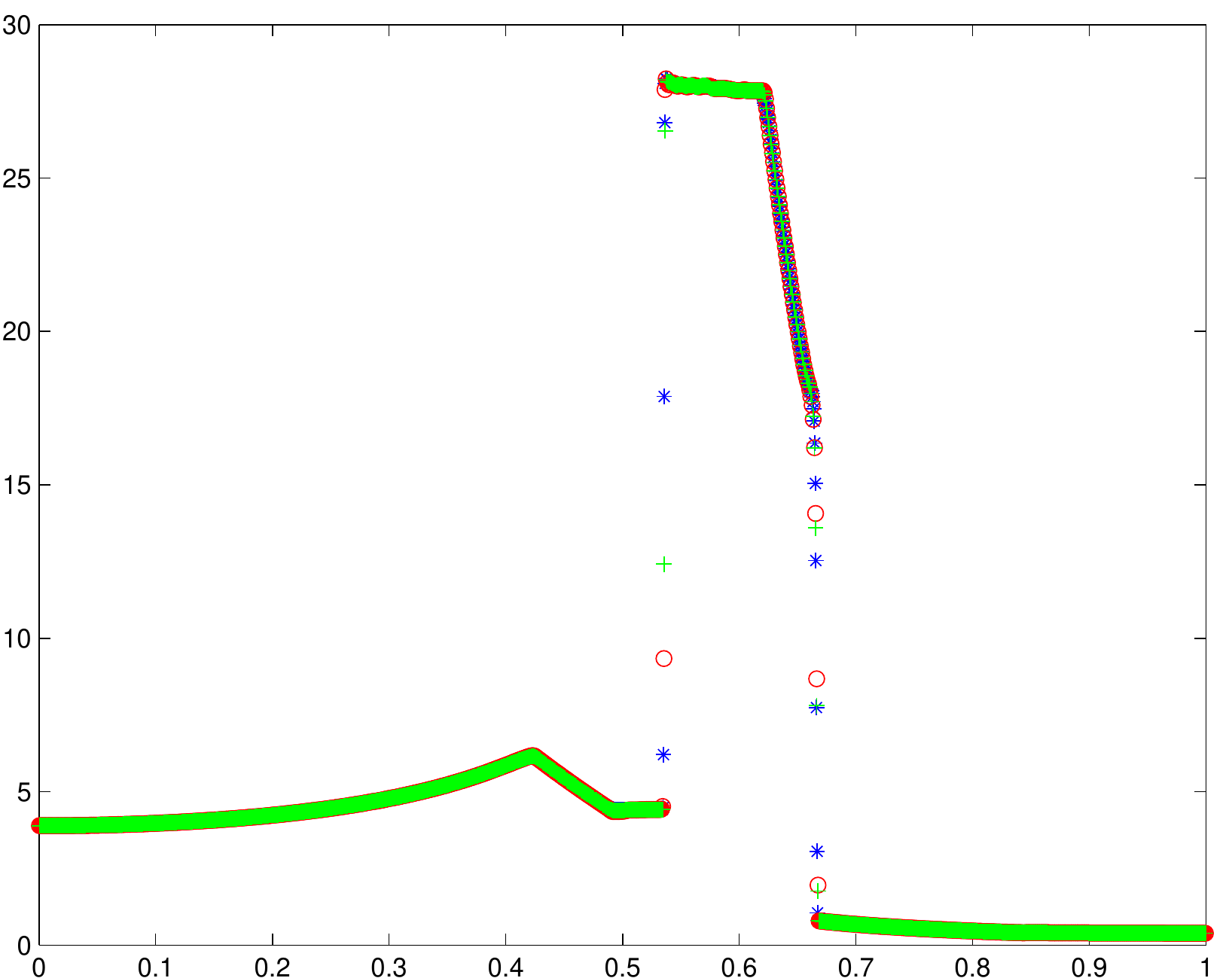}
 }
 \caption{Example \ref{ex:blast}: The results
 at $t=0.75$ obtained by the BGK scheme (``{$\circ$}", with 1000 uniform cells), the KFVS scheme (``{$*$}", with 2000  uniform cells)
 and the BGK-type scheme (``{$+$}", with 1000  uniform cells).}
 \label{fig:blast2}
\end{figure}

\begin{figure}[h]
 \centering
 \subfigure[$ \rho $]{
 \includegraphics[width=0.3\textwidth]{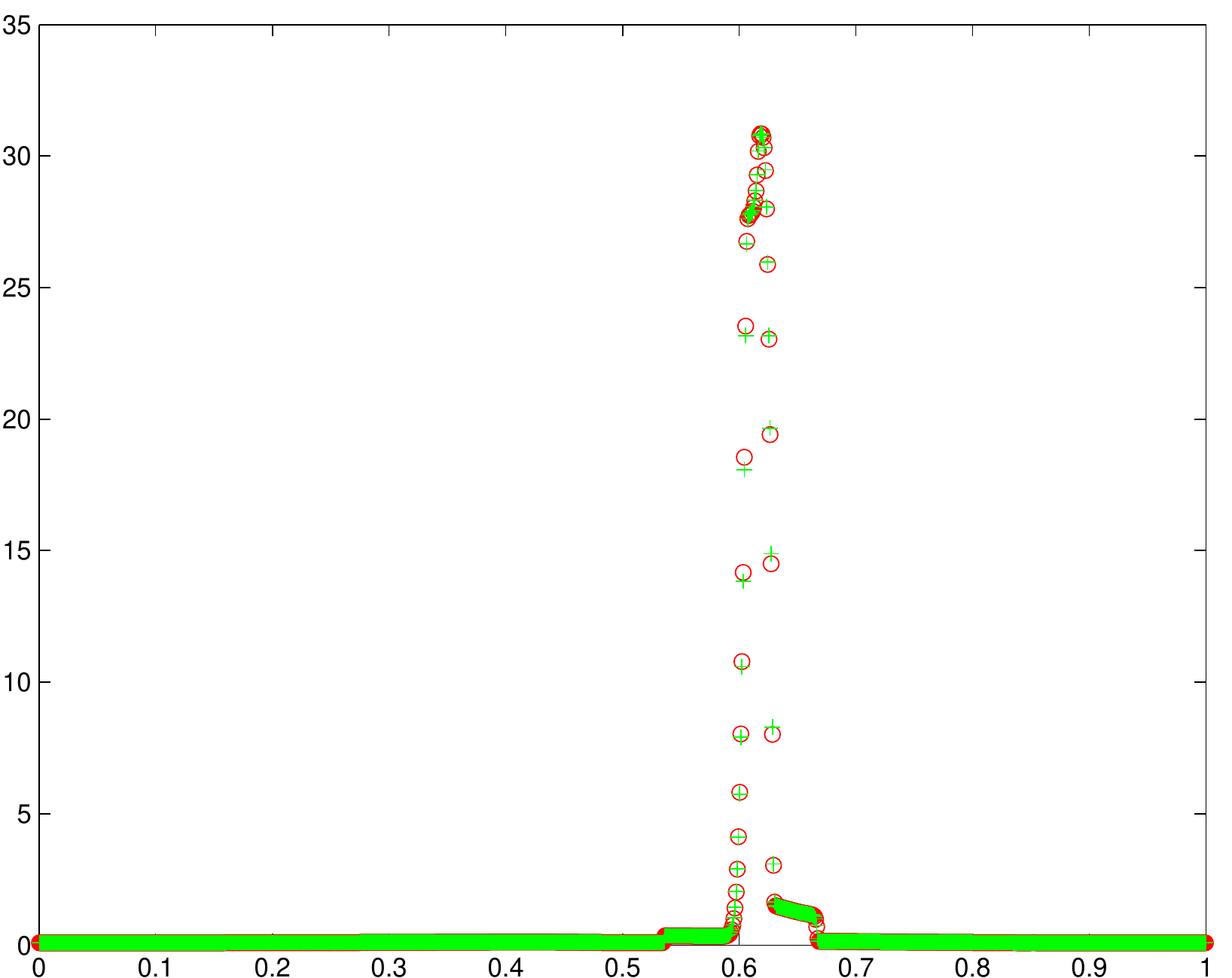}
 }
 \subfigure[$u_1$]{
 \includegraphics[width=0.3\textwidth]{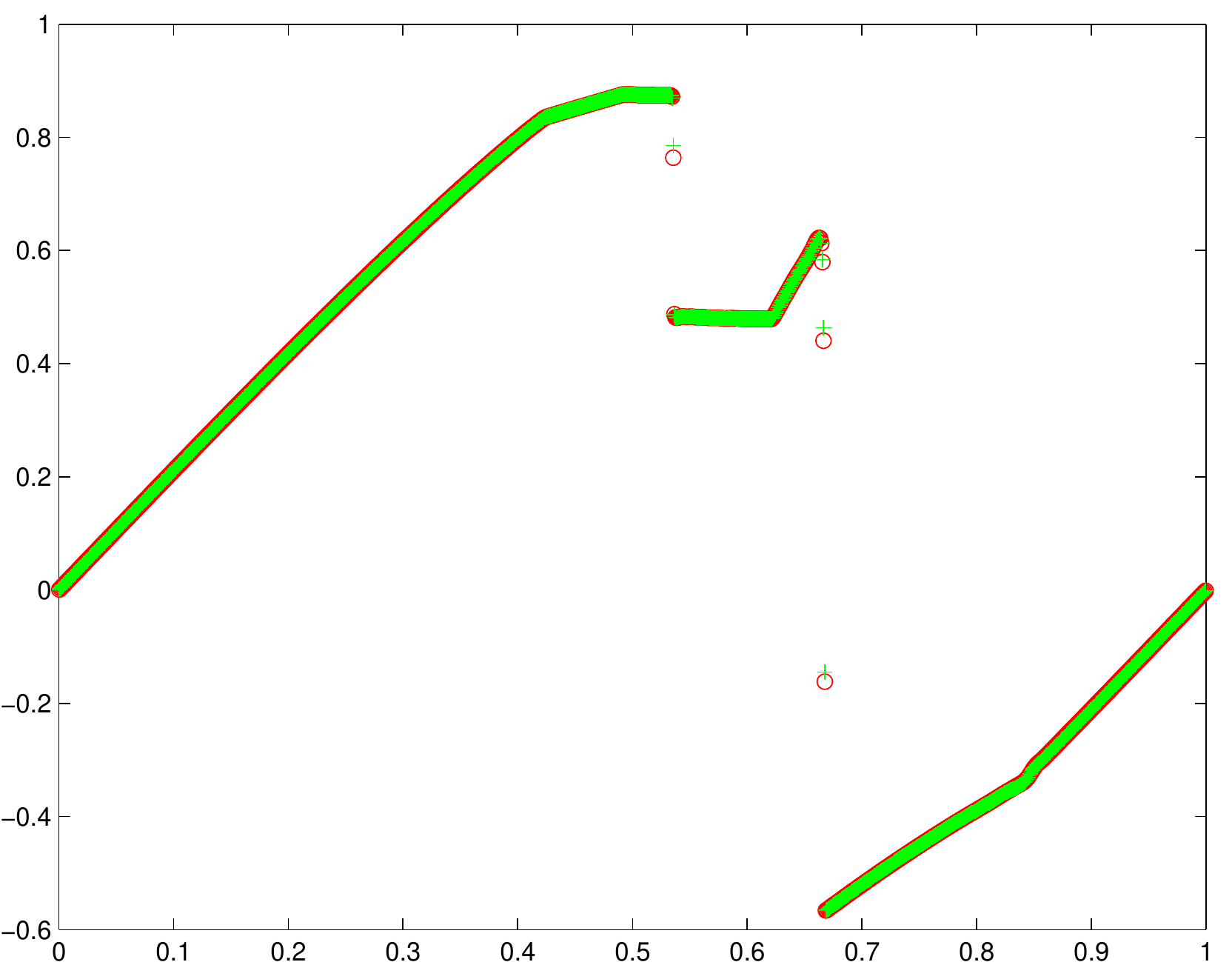}
 }
 \subfigure[$p$]{
 \includegraphics[width=0.3\textwidth]{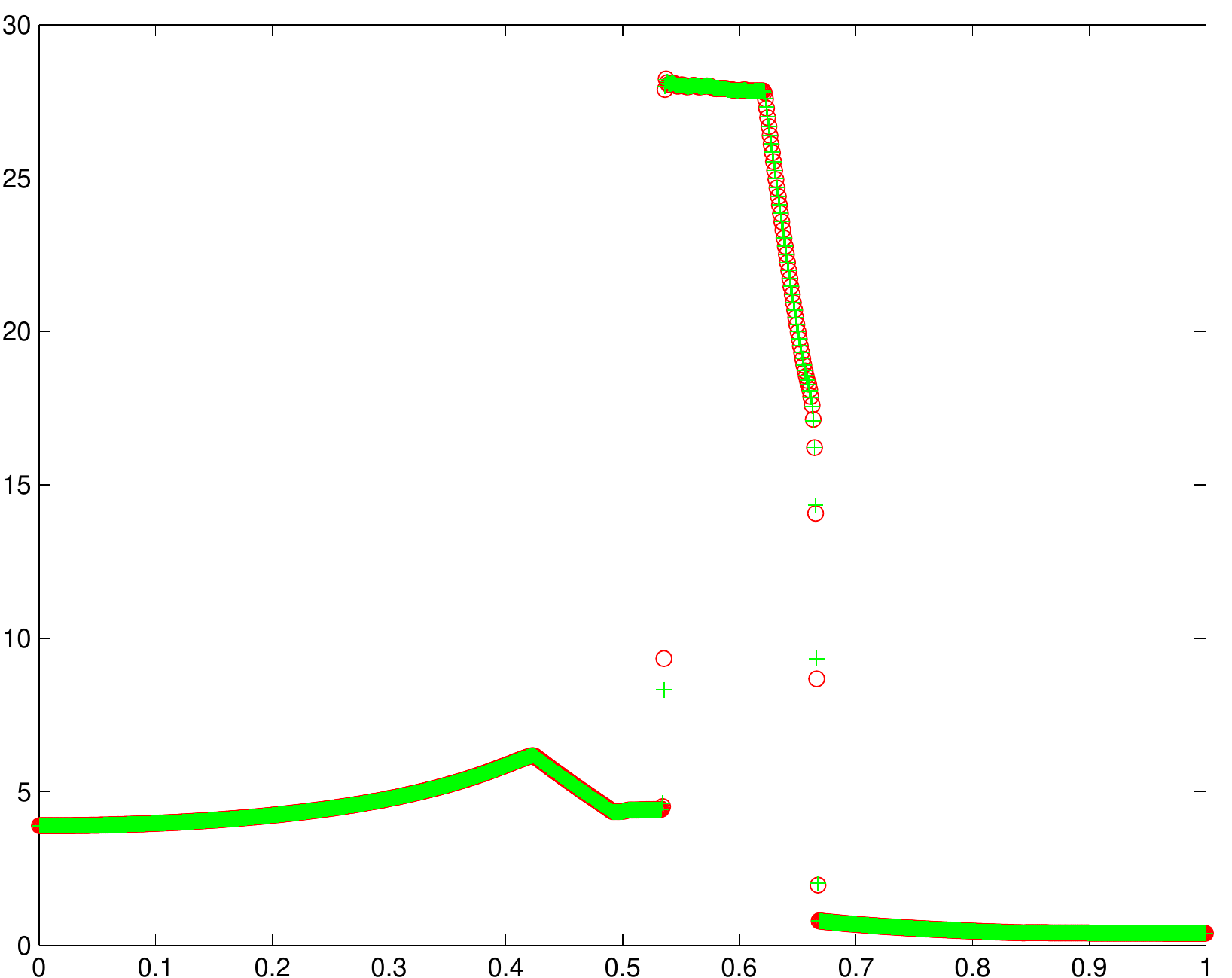}
 }
 \caption{Example \ref{ex:blast}: The results at $t=0.65$ obtained by the BGK scheme (``{$\circ$}") and the sBGK scheme (``{$+$}") with 1000  uniform cells.}
  \label{fig22:blast}
\end{figure}
Fig. \ref{fig:blast} gives the numerical results at $t=0.75$ in the domain $[0,1]$ obtained by the BGK scheme (``{$\circ$}"), the KFVS scheme (``{$*$}") and the BGK-type scheme (``{$+$}").
It is found that the solutions are bounded by two shock waves at $t=0.75$ because both initial discontinuities evolve and two blast waves collide with each other. Those schemes may well resolve those  discontinuities. However, the peak of the narrow structure in the density calculated by the KFVS scheme deviates from the results obtained by the BGK and BGK-type schemes. After refining the mesh with 2000 uniform cells, Fig. \ref{fig:blast2} shows that the dissipation of the KFVS scheme near the contact discontinuity decreases, and its   peak position of the density agrees with  the BGK and BGK-type schemes on 1000 meshes.
Fig. \ref{fig22:blast} shows that   resolving the complex wave structures by
the sBGK scheme is almost the same as the  BGK scheme.

\subsection{2D case} 
This section solves several 2D RHD problems by only using the sBGK scheme, because using the  BGK scheme to solve 2D problems is too expensive to be acceptable. Those problems are the smooth problem, the Riemann problems, the implosion in a box and the relativistic jet.

\begin{example}[Accuracy test]\label{ex:accurary2D}\rm
The  smooth problem with the exact solution
\[\rho(x,y,t)= 1+0.5\sin(2\pi (x-0.2t + y-0.2t)), u_1(x,y,t)=u_2(x,y,t)=0.2, p(x,y,t)=1,\]
is used to test accuracy of the numerical methods.
It describes a sine wave propagating periodically in the domain $\Omega=[0,1]\times[0,1]$ at an angle $\alpha={45}^{\circ}$ with the $x$-axis.
The computational domain $\Omega$ is divided into $N\times N$ uniform cells and the periodic boundary conditions are specified.
\end{example}
\begin{table}[H]
 \setlength{\abovecaptionskip}{0.cm}
 \setlength{\belowcaptionskip}{-0.cm}
 \caption{Example \ref{ex:accurary2D}: Numerical errors at $t = 0.2$  in $l^1, l^\infty$-norms and convergence rates with or without limiter.}\label{tab:accuracy2D}
 \begin{center}
   \resizebox{\textwidth}{28mm}{
   \begin{tabular}{*{9}{c}}
     \toprule
     \multirow{2}*{N} &\multicolumn{4}{c}{With limiter} &\multicolumn{4}{c}{Without limiter}\\
     \cmidrule(lr){2-5}\cmidrule(lr){6-9}
     & $l^1$ error  & $l^1$ order  &  $l^\infty$ error  &  $l^\infty$ order & $l^1$ error  & $l^1$ order  &  $l^\infty$ error  &  $l^\infty$ order\\
     \midrule
     25   & 3.5446e-03 &  -        & 1.1088e-02    & -            &9.0008e-04   &-       & 1.5882e-03 & -\\
     50   & 1.0431e-03 &  1.7647   & 4.3227e-03    & 1.3589       &2.3054e-04   &1.9650  & 3.9752e-04 & 1.9983\\
     100  & 2.6957e-04 &  1.9522   & 2.0557e-03    & 1.0723       &5.8127e-05   &1.9878  & 9.9022e-05 & 2.0052\\
     200  & 7.3381e-05 &  1.8772   & 8.7575e-04    & 1.2310       &1.45467e-05  &1.9985  & 2.4604e-05 & 2.0089\\
     400  & 1.8600e-05 &  1.9801   & 3.1622e-04    & 1.4696       &3.63818e-06  &1.9994  & 6.1313e-06 & 2.0046\\
     \bottomrule
   \end{tabular}}
 \end{center}
\end{table}
The errors and numerical orders of accuracy for the density $\rho$ by
using the present sBGK scheme are listed in Tables \ref{tab:accuracy2D}. The results show that second-order rates of convergence in the $l^1$
norm can be obtained, but the rate of convergence in $l^\infty$ is  little lower when the van Leer limiter is used.

\begin{example}[Riemann problem I]\rm \label{ex:2DRP1} The initial data
are given by
\begin{small}
 \begin{equation*}
   \label{exeq:SRHD2DRP1}
  (\rho,u_1,u_2,p)(x,y,0)=
   \begin{cases}
     (1,0,0,1),& x>0.5,y>0.5,\\
     (0.5121,-0.3548,0,0.4),& x<0.5,y>0.5,\\
     (1,-0.3548,-0.3548,1), &x<0.5,y<0.5,\\
     (0.5121,0,-0.3548,0.4),&x>0.5,y<0.5.
   \end{cases}
 \end{equation*}
 \end{small}
\end{example}
\begin{figure}[h]
 \centering
 \includegraphics[width=0.45\textwidth]{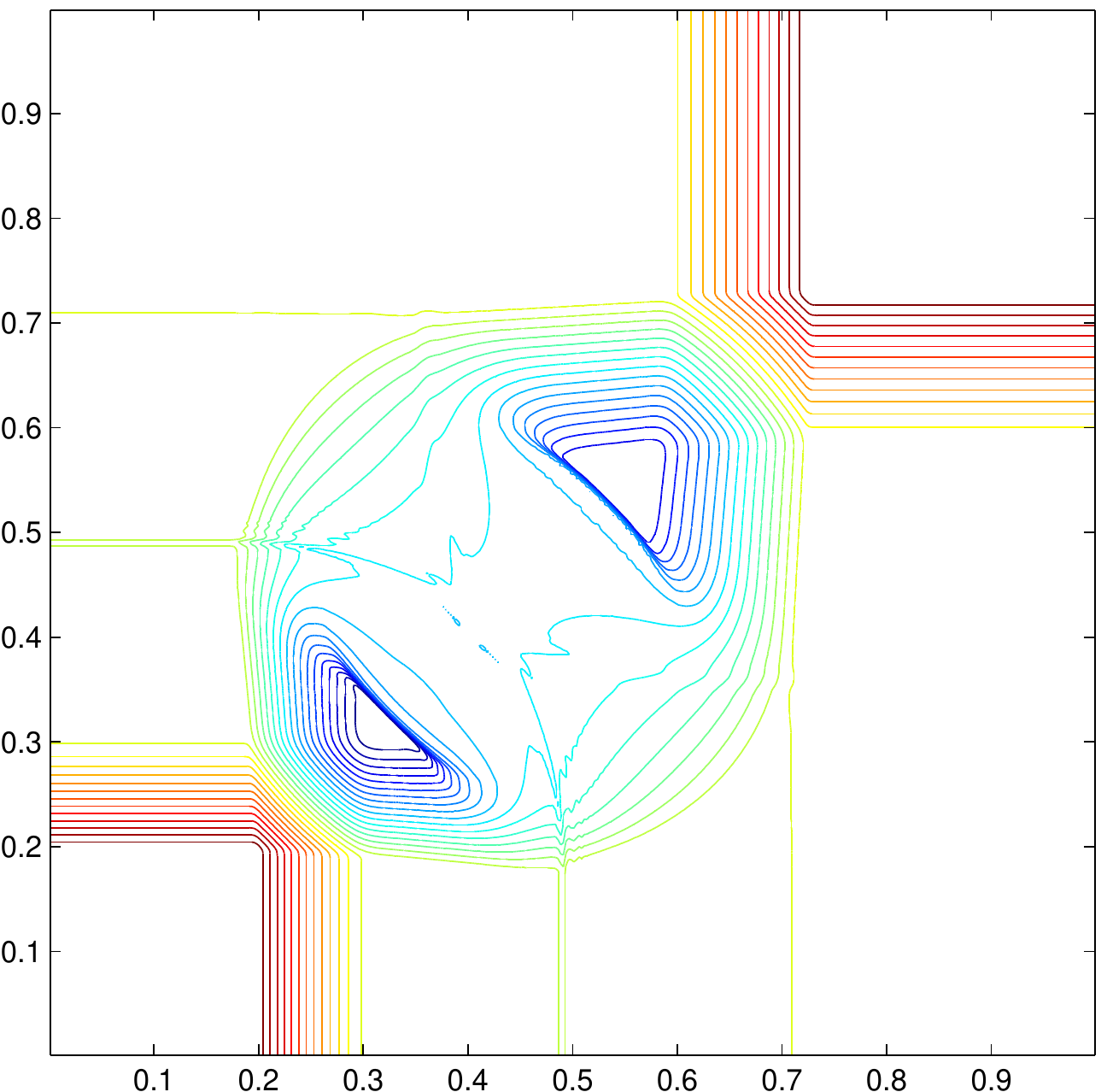}
 \caption{Example \ref{ex:2DRP1}: The contour of the density logarithm $\ln \rho$ at time $t = 0.4$ within the domain $[0,1] \times [0,1]$ obtained with the sBGK scheme and $400 \times 400$ uniform cells (30 equally spaced contour lines).}
 \label{fig:FourRare}
\end{figure}
The computational domain $\Omega$  $[0,1] \times [0,1]$ is
divided into $400 \times 400$ uniform cells.
Fig. \ref{fig:FourRare} shows clearly that four rarefaction waves are formed from those four initial discontinuities. As time goes on, the four rarefaction waves interact each other and form two curved shock waves perpendicular to the line $x = y$.

\begin{example}[Riemann problem II]\rm \label{ex:2DRP2}
The initial data  are taken as
 \begin{small}
  \begin{equation*}
    \label{exeq:2DRP1}
   (\rho,u_1,u_2,p)(x,y,0)=
    \begin{cases}
      (0.025510800277587, 0, 0, 0.142814727617575),& x>0.5,y>0.5,\\
      (0.1,0.7 , 0, 1),& x<0.5,y>0.5,\\
      (0.5, 0, 0, 1),  & x<0.5,y<0.5,\\
      (0.1, 0, 0.7, 1),& x>0.5,y<0.5,
    \end{cases}
  \end{equation*}
 \end{small}%
where the left and bottom discontinuities are contact discontinuities and the top and right ones are two shock waves with the speed of 0.855938.
\end{example}

\begin{figure}[h]
 \centering
 \includegraphics[width=0.45\textwidth]{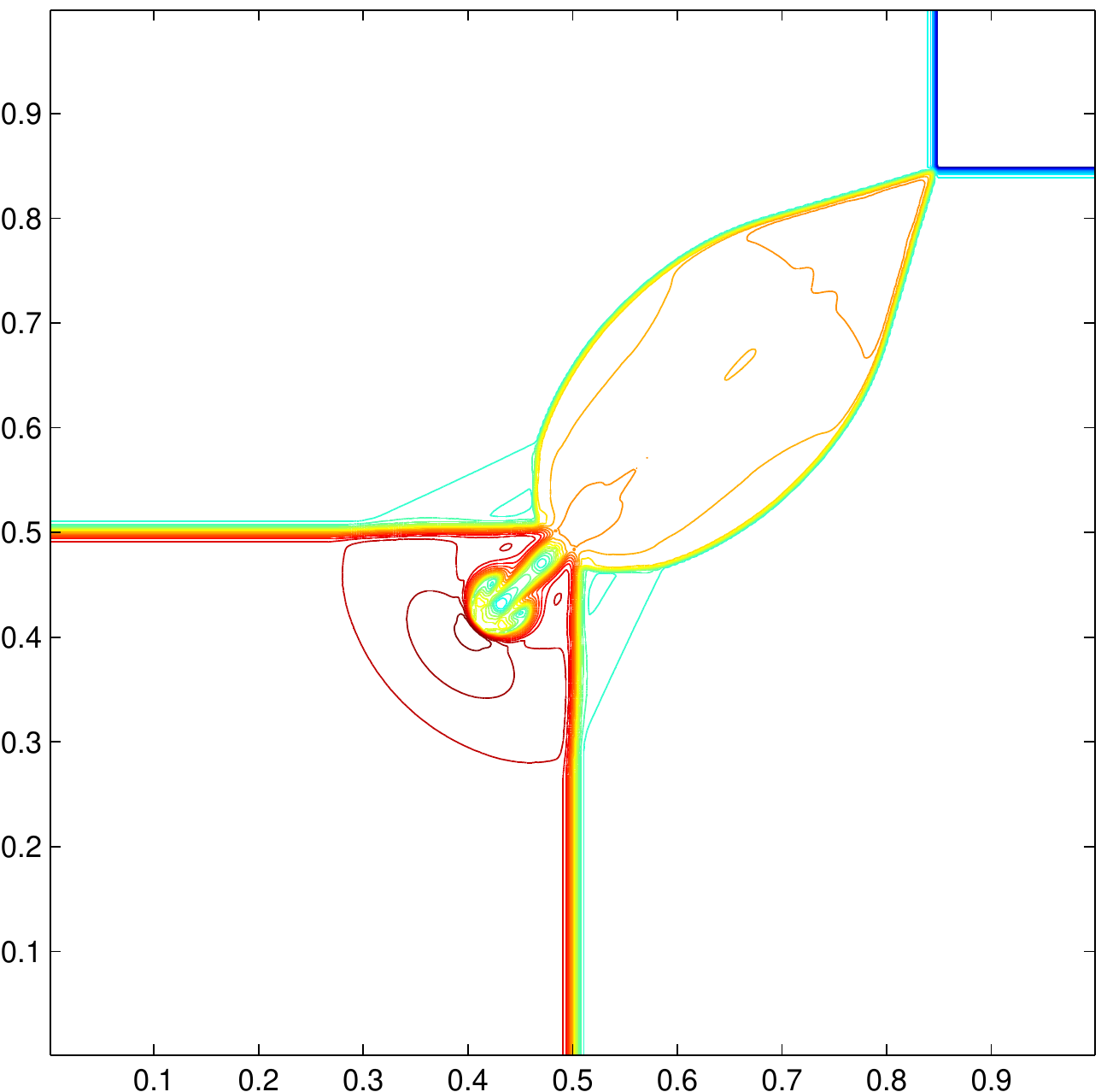}
 \caption{Example \ref{ex:2DRP2}: The contour of the density logarithm $\ln \rho$ and at time $t = 0.4$ within the domain $[0,1] \times [0,1]$
 obtained with the sBGK method and $400 \times 400$ uniform cells (30 equally spaced contour lines).}
\label{fig:2DRP2}
\end{figure}
Fig. \ref{fig:2DRP2} gives the contours of the density logarithm at
$t = 0.4$ obtained by the sBGK scheme with $400 \times 400$ uniform cells and $\alpha_1 = \alpha_2 = C_1=C_2=1$. We see that the four initial
discontinuities interact each other and form a mushroom cloud  around the point $(0.5,0.5)$ as time increases, and the sBGK scheme captures the contact discontinuities, shock waves and other complex structures well.

\begin{example}[Riemann problem III]\rm \label{ex:2DRP3}
 The initial data  are
 \begin{equation*}
   \label{exeq:2DRP3}
  (\rho,u_1,u_2,p)(x,y,0)=
  \begin{cases}
    (0.5,0.5,-0.5,5),& x>0.5,y>0.5,\\
    (1,0.5,0.5,5),& x<0.5,y>0.5,\\
    (3,-0.5,0.5,5), &x<0.5,y<0.5,\\
    (1.5,-0.5,-0.5,5),&x>0.5,y<0.5,
  \end{cases}
\end{equation*}
which are about the interaction of four vortex sheets (i.e. contact discontinuities for the perfect relativistic fluid), whose vorticity $\omega=\partial_x u_2(0,x,y)-\partial_y u_1(0,x,y)$  is negative.
\end{example}
\begin{figure}[h]
 \centering
 \includegraphics[width=0.45\textwidth]{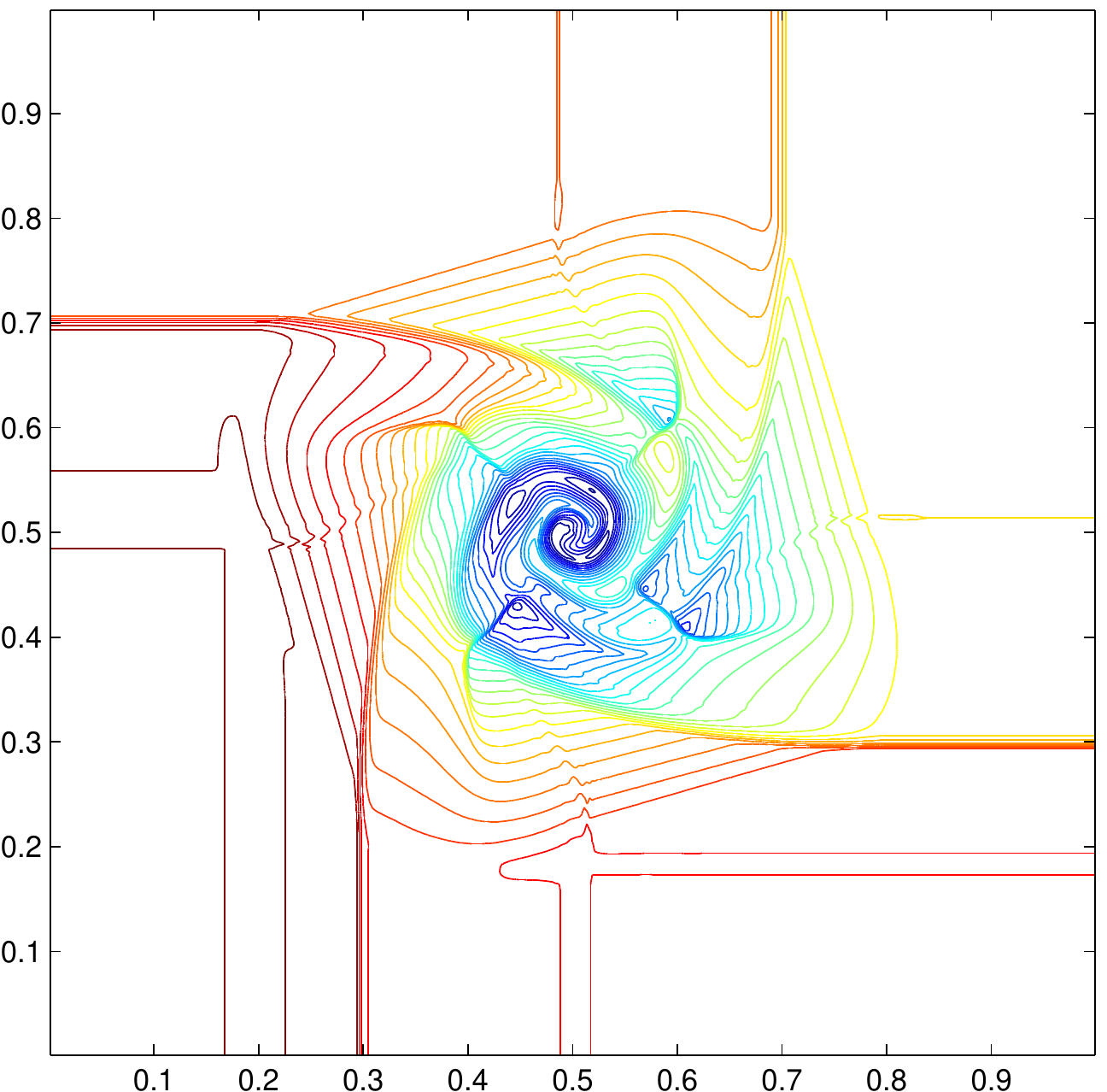}
 \caption{Example \ref{ex:2DRP3}: The contour of the density logarithm $\ln \rho$ at time $t = 0.4$ within the domain $[0,1] \times [0,1]$
 obtained with the sBGK method and $400 \times 400$ uniform cells (30 equally spaced contour lines).}
 \label{fig:2DRP3}
\end{figure}
Fig. \ref{fig:2DRP3} displays the contours of the density logarithm $\ln \rho$ at $t = 0.4$ obtained by using the sBGK scheme with $400 \times 400$ uniform cells. The results show that the four initial vortex sheets interact each other to form a spiral with the low density around the center of the domain as time increases. It is the typical cavitation phenomenon in gas dynamics.

\begin{example}[Implosion in a box]\rm\label{ex:implosion}
 This example considers the implosion inside a squared domain $[0,1]\times[0,1]$ with reflecting walls. 
Initially, the values of $(\rho,u_1,u_2,p)$ are specified as follows
 \begin{equation*}
   (\rho,u_1,u_2,p)(x,y,0)=\begin{cases}
                           (1,0,0,10),  & |x-1|\leq0.5,|y-1|\leq0.5,\\
                           (1,0,0,0.01), & \text{otherwise}.
                        \end{cases}
 \end{equation*}
\end{example}

\begin{figure}[h]
 \centering
 \subfigure[$ \rho $]{
   \includegraphics[width=4.5cm,height=4.5cm]{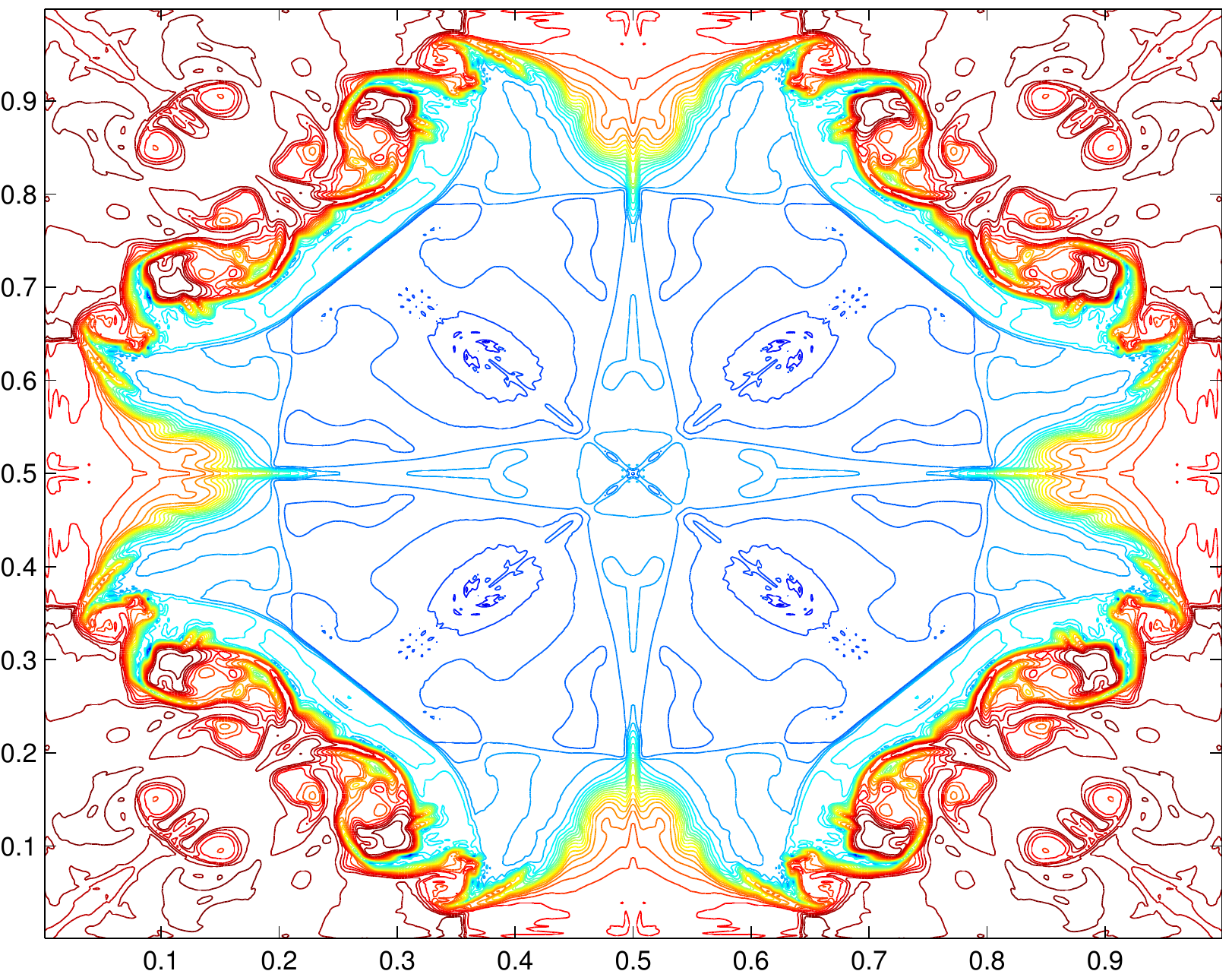}
 }
 \subfigure[$p$]{
   \includegraphics[width=4.5cm,height=4.5cm]{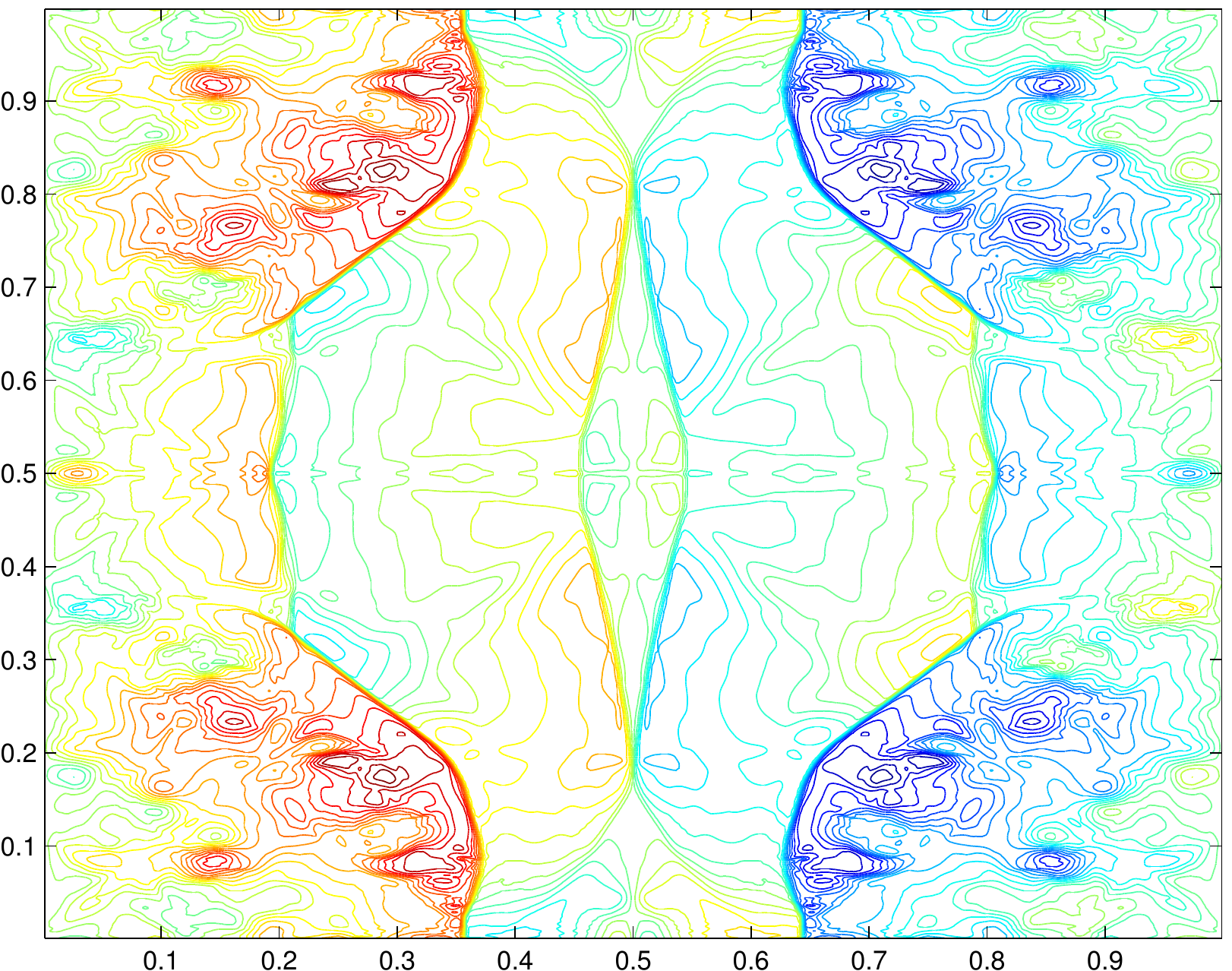}
 }

 \subfigure[$u_1$]{
   \includegraphics[width=4.5cm,height=4.5cm]{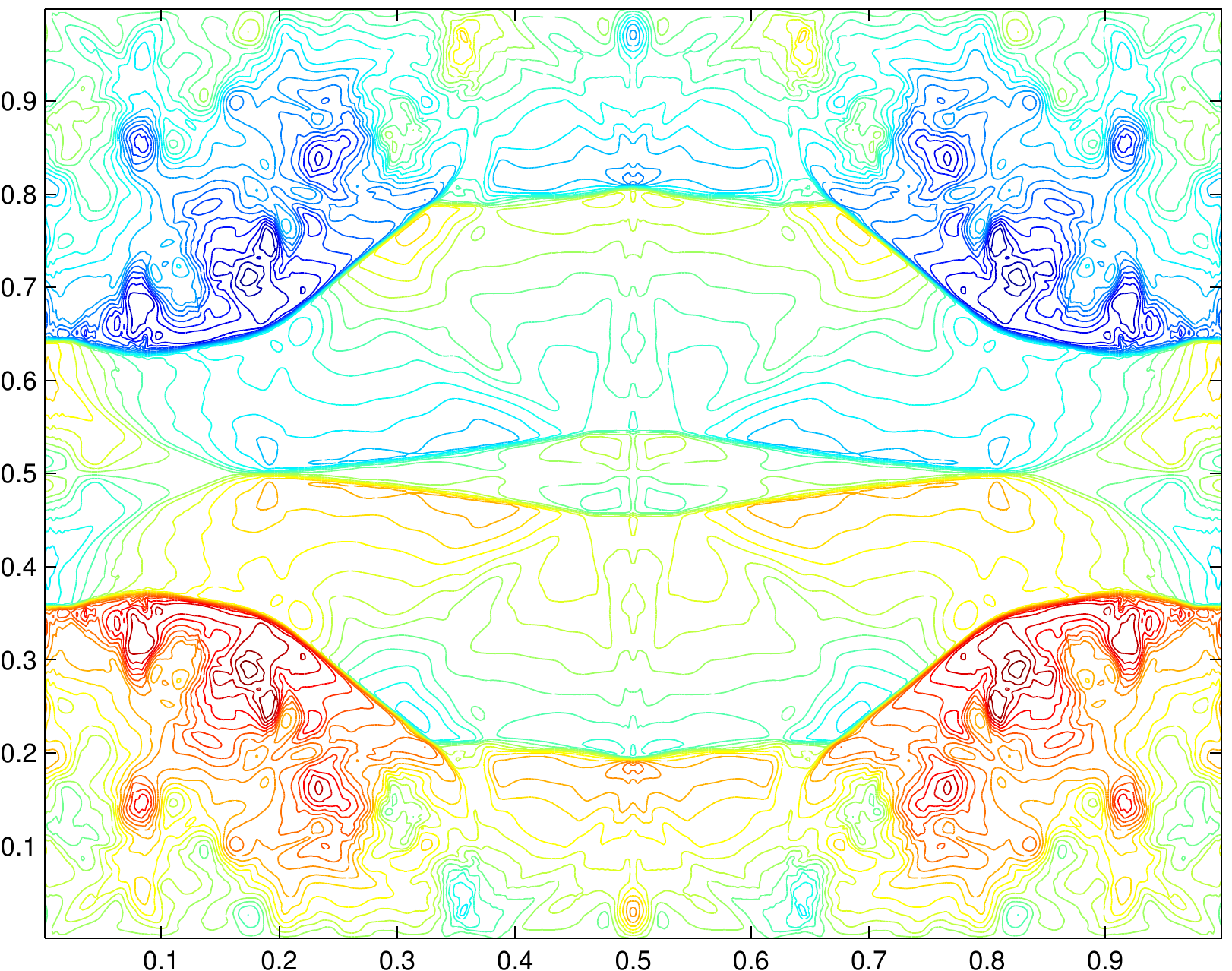}
 }
 \subfigure[$u_2$]{
   \includegraphics[width=4.5cm,height=4.5cm]{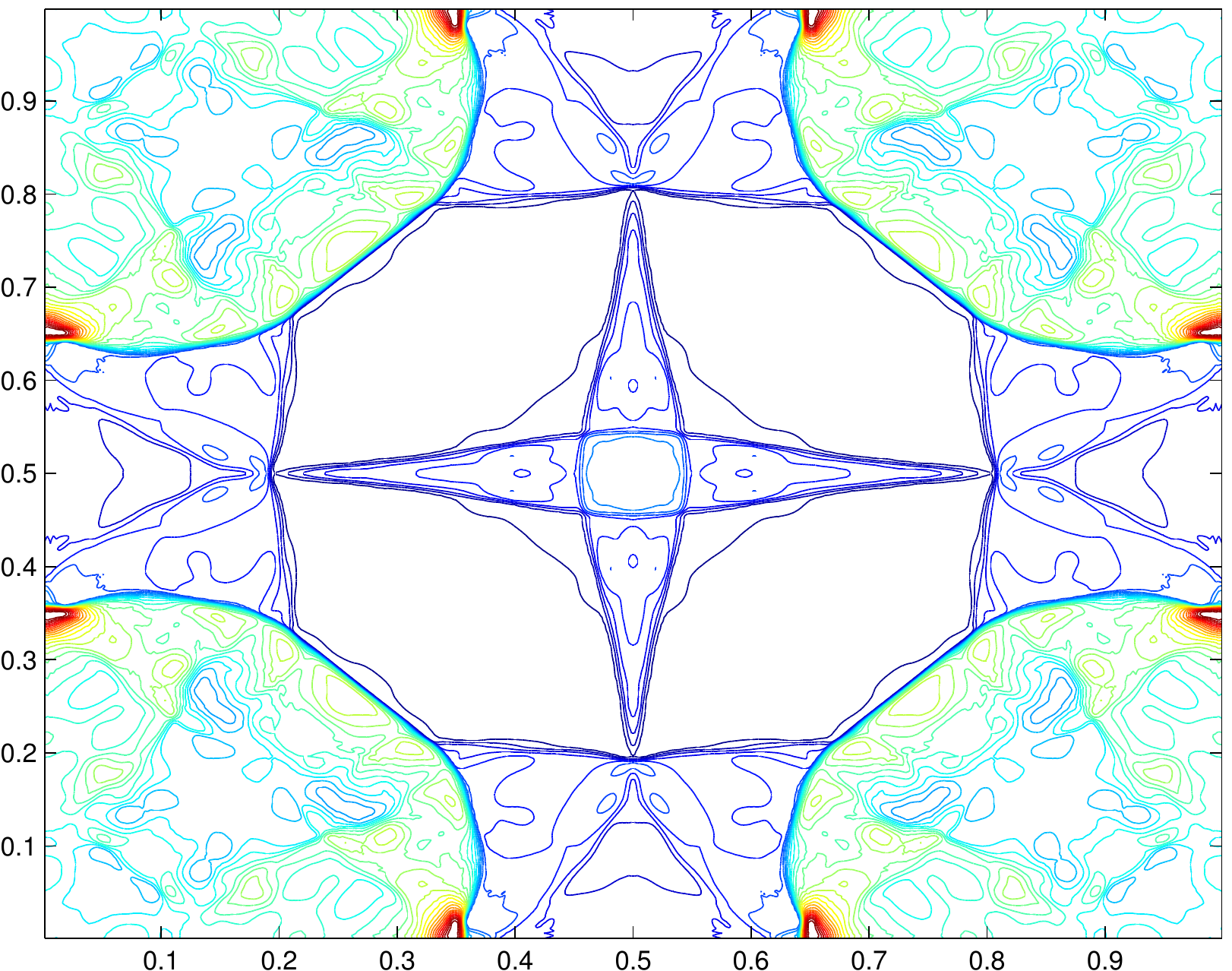}
 }
 \caption{Example \ref{ex:implosion}: The contours of the solutions at $t = 3$ obtained by the sBGK scheme with $400\times400$  uniform cells. 30 equally spaced contour lines are used.}
 \label{fig:implosion3}
\end{figure}
Fig. \ref{fig:implosion3} gives the contours of the density, the pressure and the velocities at time $t = 3$ obtained by our sBGK scheme on the uniform mesh of $400\times400$ cells. It can be seen   that four arc-shaped shock waves are formed at the four corners of the region, and the complex small wave structures are formed in the interior of the region due to the boundary reflections.

\begin{example}[Relativistic jet]\rm\label{ex:Jet}
 The dynamics of relativistic jet relevant in astrophysics has been  widely studied by numerical methods in the literature.
 This test  simulates a relativistic jet
 with  the computational region  $[0,12]\times[-3.5,3.5]$  and $\alpha_1 =\alpha_2= C_1=C_2=1$ using minmod limiter.
 The initial states for the relativistic jet beam are
 \begin{align*}
   ( \rho_b,u_{1,b},u_{2,b},p_b)=(0.01,0.99,0.0,0.1),\ \
   ( \rho_m,u_{1,m},u_{2,m},p_m)=(1.0,0.0,0.0,0.1),
 \end{align*}
 where the subscripts $b$ and $m$  correspond to the beam and  medium, {respectively}.
\end{example}
The initial relativistic jet is injected through a unit wide nozzle located at the middle of left boundary, while a reflecting boundary is specified outside of the nozzle.
Outflow boundary conditions with zero gradients of variables are imposed at the other part of the domain boundary. Fig. \ref{fig:Jet}  shows
the schlieren images of the rest-mass density logarithm $\ln \rho$ at $t=2,4,8,10$ obtained by our sBGK scheme
 on the mesh of $600\times350$ uniform cells.
 For a comparison,
 Fig. \ref{fig:JetLLF} displays the results at $t=10$ obtained by
  using the second-order high-resolution local Lax-Friedrich (LLF) scheme
  on the meshes of $600\times350$ and $1200\times700$ uniform cells, which is
 built on the local Lax-Friedrich flux, e.g. defined in the $x$-direction by
  \begin{equation*}\hat{\vec{F}}^1_{i+\frac12,j}
   = \frac12\left(\vec{F}^1(\vec{W}_L)+\vec{F}^1(\vec{W}_R)
-\alpha(\vec{W}_R-\vec{W}_L)\right),\ \alpha=\text{max}\{{\varrho}(\vec{W}_L),
{\varrho}(\vec{W}_R)\},
\end{equation*}
where $\vec{W}_L:=\vec{W}_h(x_{i+\frac12}-0,y_j,t_n)$,
$\vec{W}_R:=\vec{W}_h(x_{i+\frac12}+0,y_j,t_n)$,
 ${\varrho}(\vec{W})=\max_{1\leq k\leq 4}(|\lambda_1^{(k)}|)$ is the spectral radius of
$\frac{\partial \vec{F}^1}{\partial \vec{W}}$,
referred to Section \ref{sec:GovernEqns},
the same spatial reconstruction as that in the sBGK scheme,
and
the second-order explicit TVD Runge-Kutta time discretization.
 The results show that the time evolution of a light relativistic jet with large
 internal energy is well simulated by those schemes, and the shock wave at the jet head is  well captured during the whole simulation.
Moreover, the  sBGK scheme resolves the waves better than the high-resolution LLF scheme on the mesh of $600\times350$ cells, and is comparable to that obtained by using the latter the fine mesh
 of $1200\times700$ cells.
\begin{figure}[h]
 \centering
 \subfigure[$t=2$]{
   \includegraphics[width=0.45\textwidth]{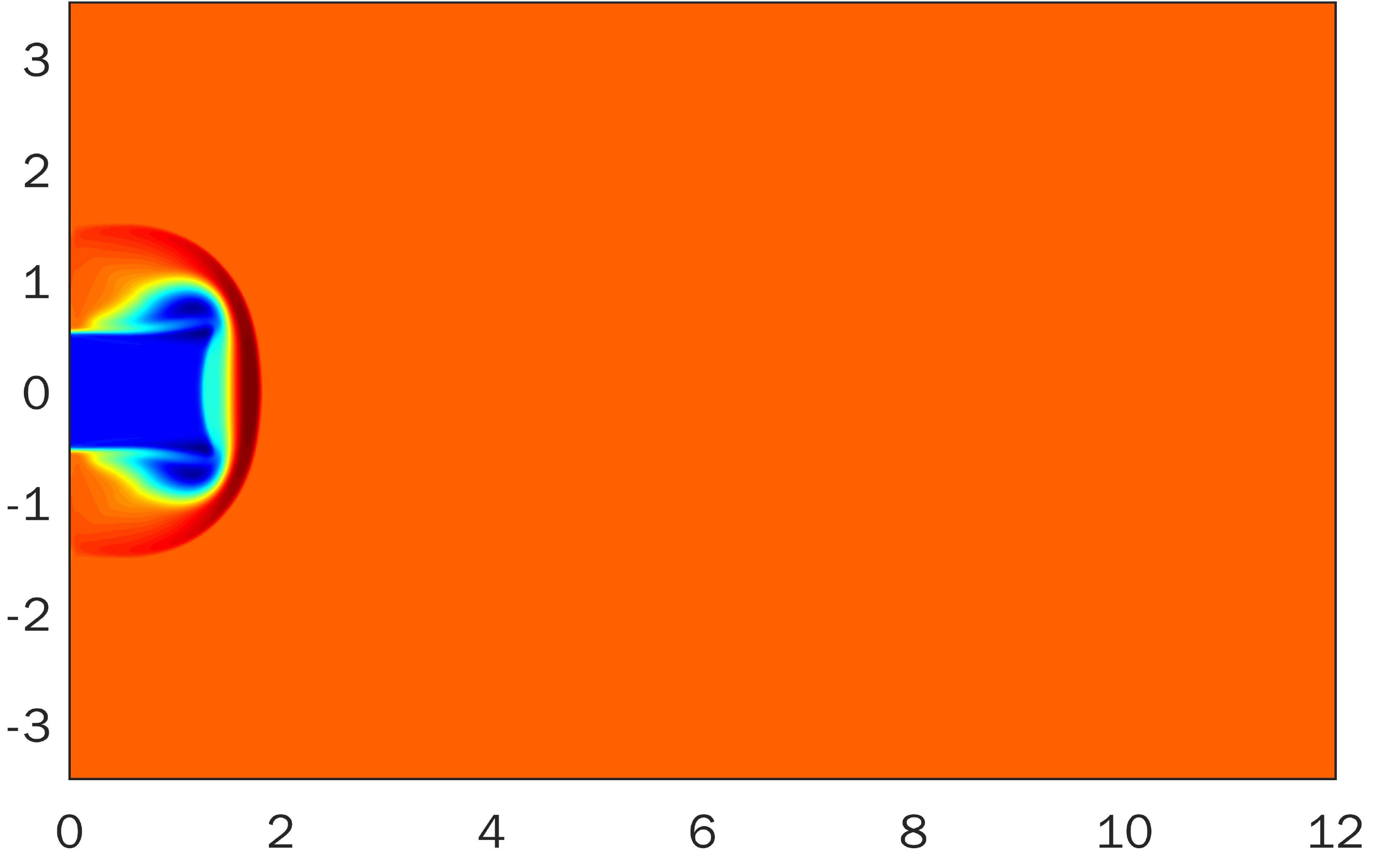}
 }
 \subfigure[$t=4$]{
   \includegraphics[width=0.45\textwidth]{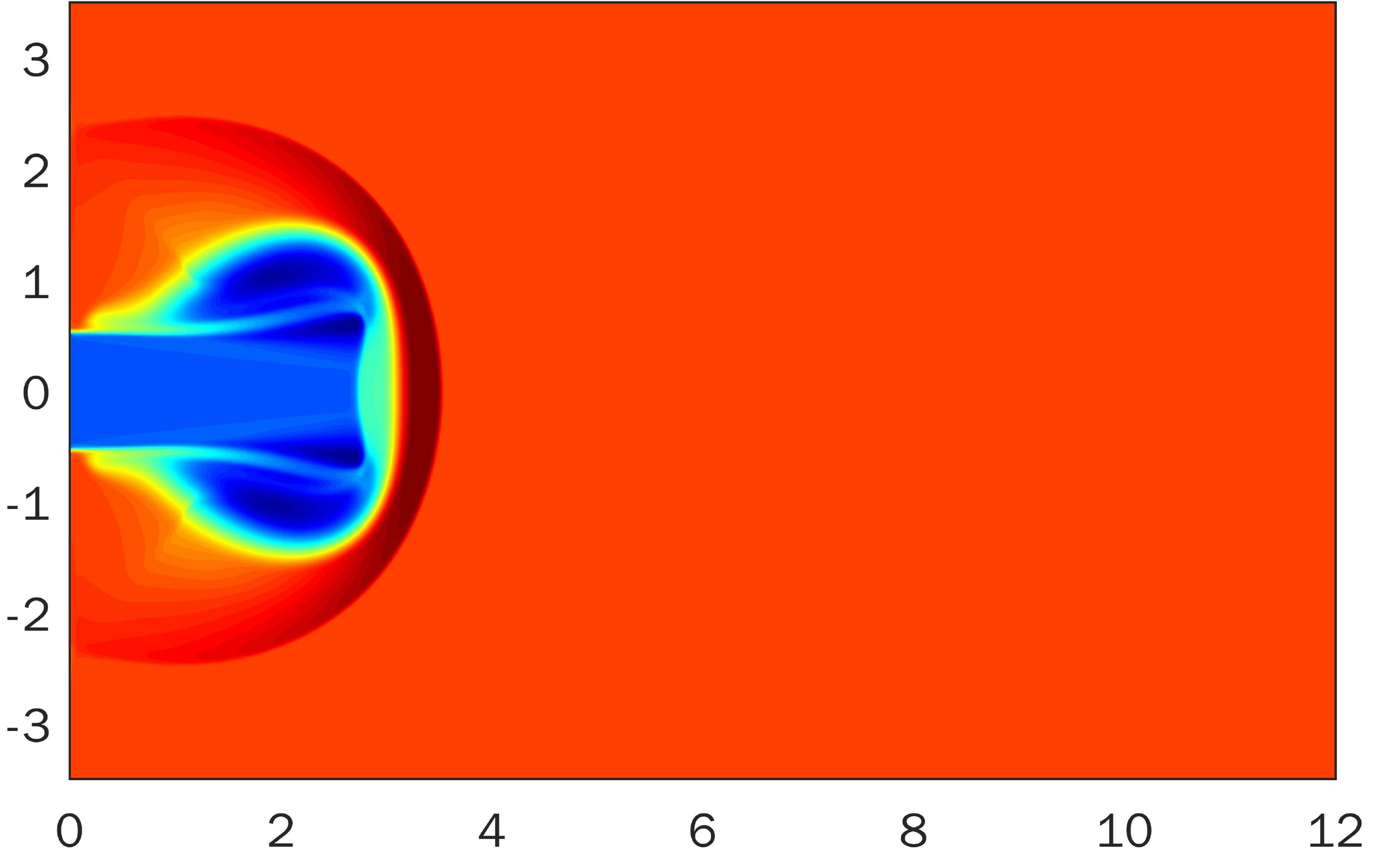}
 }
 \subfigure[$t=8$]{
   \includegraphics[width=0.45\textwidth]{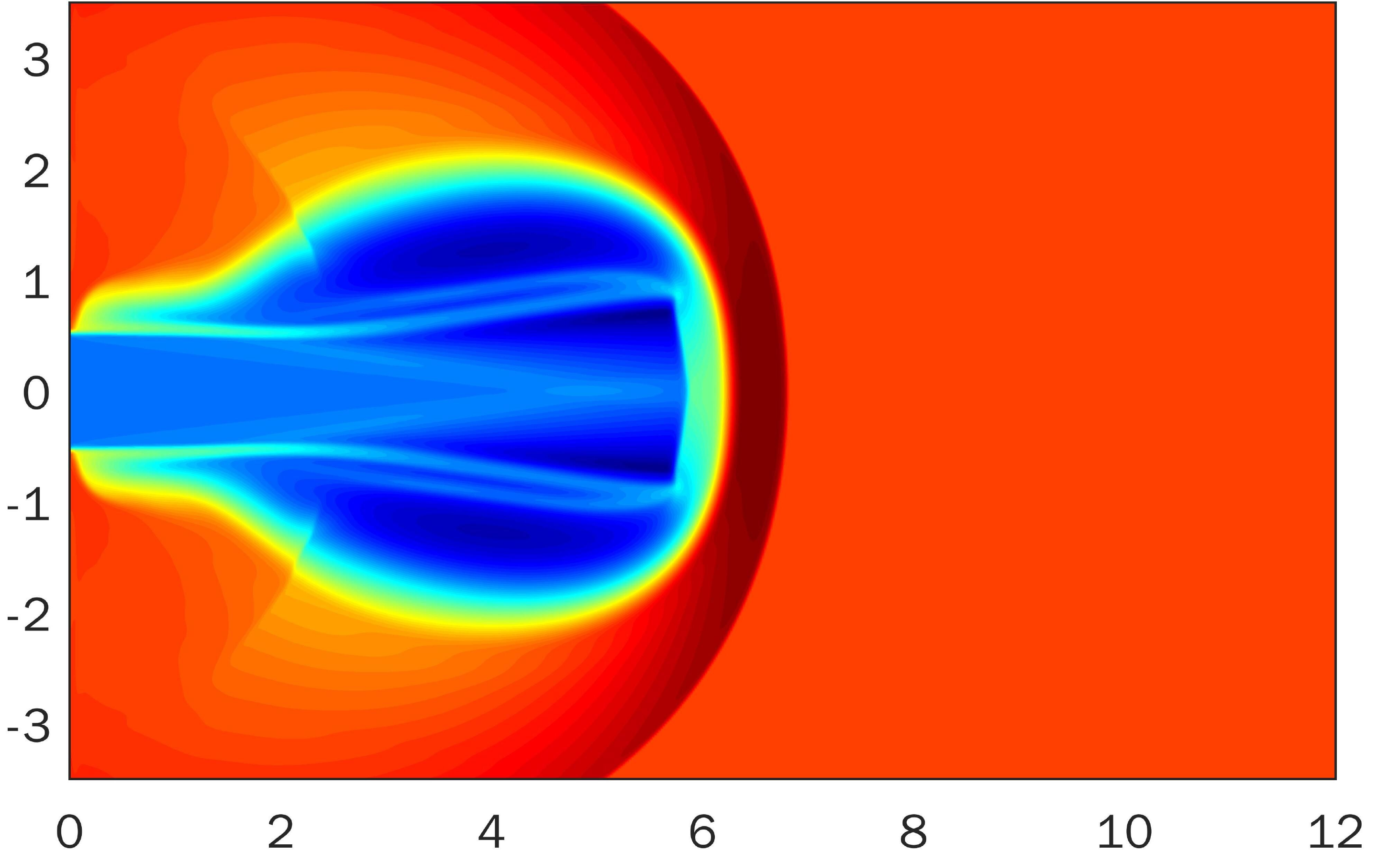}
 }
 \subfigure[$t=10$]{
   \includegraphics[width=0.45\textwidth]{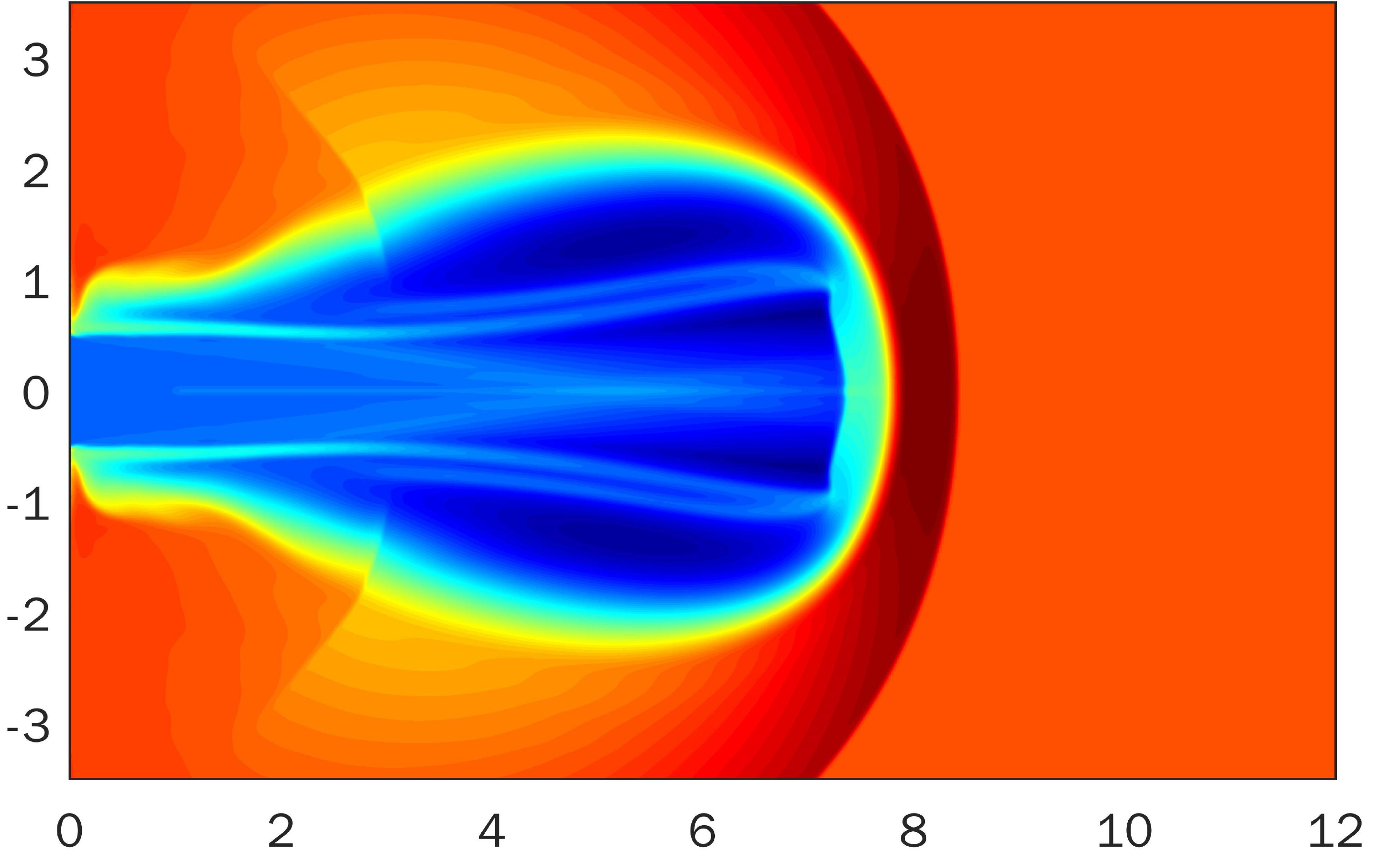}
 }
 \caption{Example \ref{ex:Jet}: Schlieren images of
$\ln \rho $  at several different times obtained by the sBGK
 scheme with $600\times350$ uniform cells in the domain $[0, 12]\times[-3.5, 3.5]$.}
 \label{fig:Jet}
\end{figure}

\begin{figure}[h]
 \centering
 \subfigure[$600\times350$ uniform cells]{
   \includegraphics[width=0.45\textwidth]{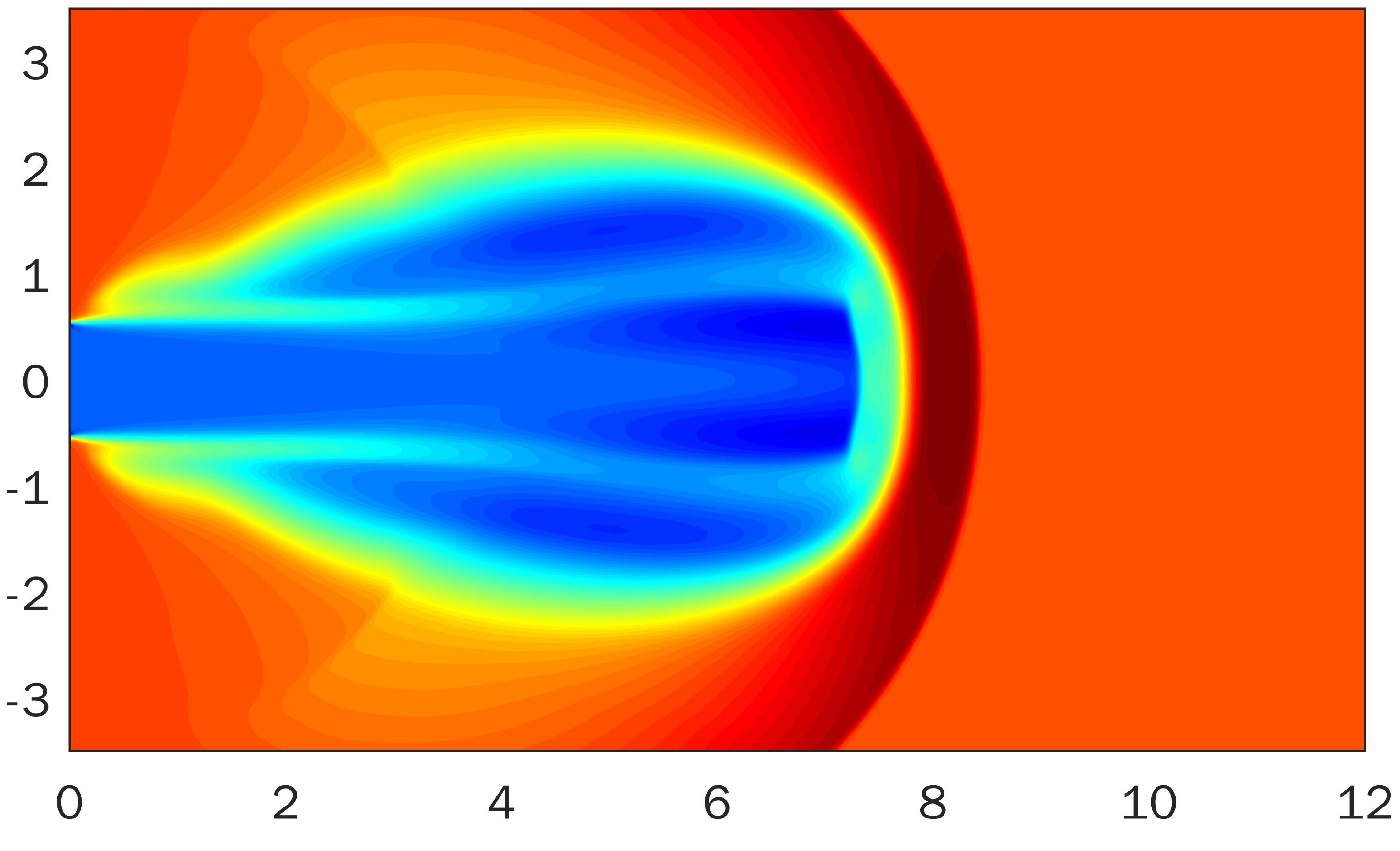}
 }
 \subfigure[$1200\times700$ uniform cells]{
   \includegraphics[width=0.45\textwidth]{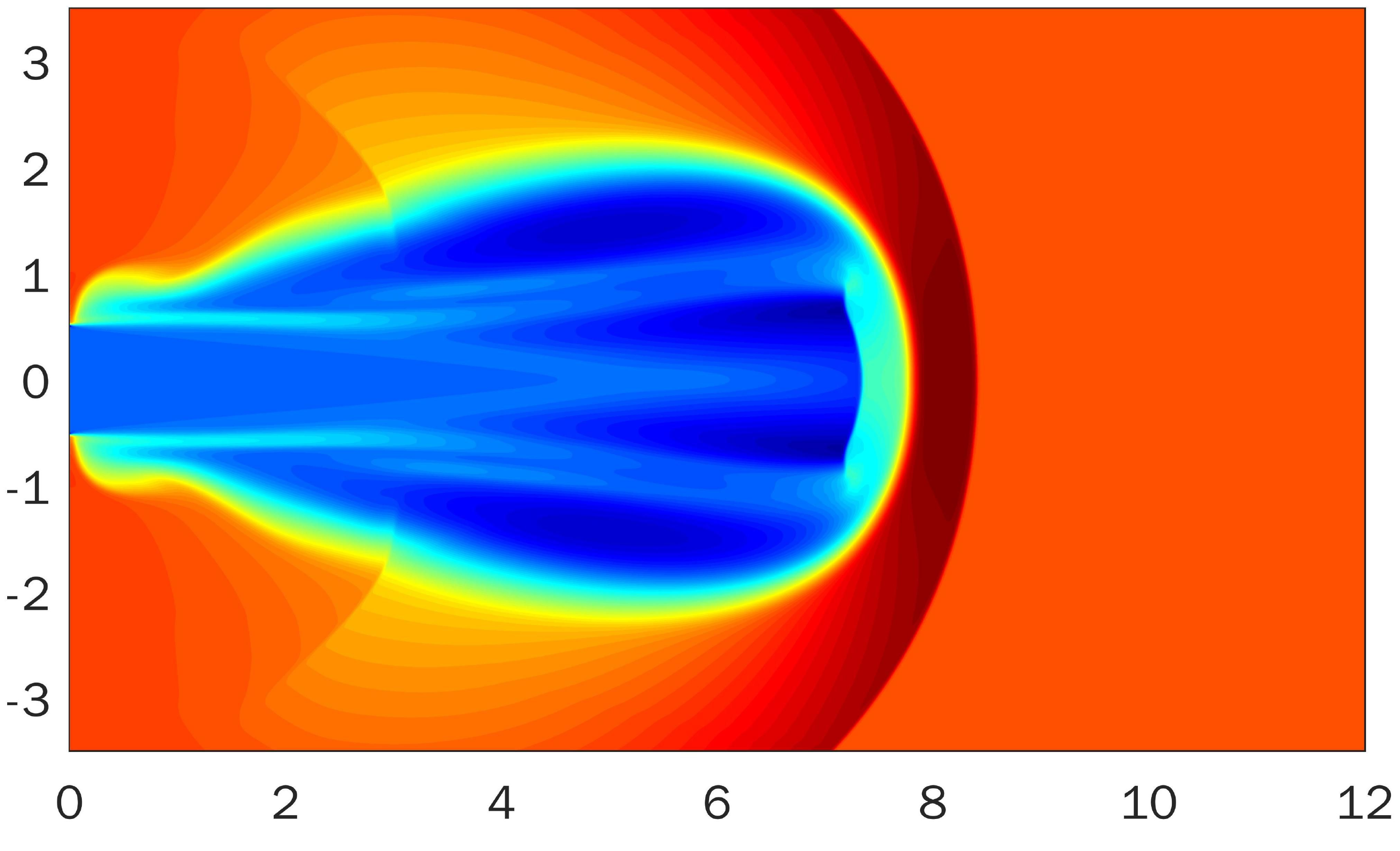}
 }
 \caption{Example \ref{ex:Jet}: Schlieren images of
$\ln \rho $  at $t=10$ obtained by using the high-resolution LLF scheme with the same spatial reconstruction as that in the sBGK scheme, the second-order TVD Runge-Kutta time discretization. 
}
 \label{fig:JetLLF}
\end{figure}

\section{Conclusions}
\label{sec:conclusion}
The correct equation of state (EOS) for the relativistic perfect gas has been recognized as being important.
For the relativistic perfect gases, Synge gave  the exact form of an EOS relating thermodynamic quantities of specific enthalpy and temperature, {which} is completely described in terms of modified Bessel functions \cite{synge1957}, also see \eqref{EQ:EOS}. However, such EOS does not seem to be welcome from the computational point of view since it involved the computation of Bessel functions. 
%
%
This paper extended the second-order accurate   BGK finite volume schemes
for  the ultra-relativistic flow simulations \cite{CHEN2017} to
 the 1D and 2D special relativistic hydrodynamics
with the Synge EOS. Unfortunately, such BGK  schemes
were very time-consuming thanks to calculating numerically
the triple moment integrals of the non-equilibrium part in the approximate distribution $\hat{f}$  for the macroscopic numerical flux at each time step
so that they were no longer  practical even though the the moment integrals in one dimension could be reduced to the double integrals.
 In view of this,  the
simplified BGK (sBGK) schemes were proposed 
 by removing some terms in the approximate nonequilibrium distribution at the
 cell interface for the  BGK schemes without loss of accuracy.
 They became practical because the triple moment integrals in them could be reduced to the single integrals by  using some  coordinate transformations.
Moreover, we also 
proved that the sound velocity was bounded by the speed of light and the relations between the left and right states of the shock wave, rarefaction wave, and contact discontinuity, so that the exact solution of the 1D Riemann problem could be derived.
Several 1D and 2D numerical experiments were conducted to demonstrate
the performance, accuracy and efficiency of the proposed schemes.
Besides the comparison of the sBGK scheme with
the high-resolution LLF scheme,
 the detailed comparisons of the sBGK scheme with the BGK scheme in one dimension showed that the former performed almost the same as the latter in terms of the accuracy and resolution, but was much more efficiency.


\begin{appendices}
 \section{The matrices $M_0$ and $M_1$ for 1D Euler equations}\label{M0M11D}
 If applying the transformation $\Lambda^{\alpha}_{\beta}$ with $u_2=0$ to the local rest values,
 then the matrices $M_0$ and $M_1$ for the 1D Euler equations can be explicitly given as follows
 \begin{align*}
  M_0 & =\int_{\mathbb{R}^{3}}p^{0}g\vec{\Psi}\vec{\Psi}^T d\varXi
  =\begin{pmatrix}
   \rho U^0               & G\rho U^0U^1                                      & \rho G(U^0)^2-\frac{\rho}{\zeta} \\
   G\rho U^0U^1                     & \frac{\rho((5G+\zeta) u_1^2 + G)}{\zeta /(U^0)^3}
                                    & \frac{\rho u_1(Gu_1^2+5G + \zeta)}{\zeta /(U^0)^3}                                    \\
   \rho G(U^0)^2-\frac{\rho}{\zeta} & \frac{\rho u_1(Gu_1^2+5G + \zeta)}{\zeta /(U^0)^3}
                                    & \frac{\rho (3Gu_1^2+3G + \zeta)}{\zeta/(U^0)^3}
  \end{pmatrix},
 \end{align*}
 and
 \begin{align*}
  M_1 & =\int_{\mathbb{R}^{3}}p^{1}g\vec{\Psi}\vec{\Psi}^T d\varXi
  =
  \begin{pmatrix}
   \rho U^1                 & \rho G(U^1)^2+\frac{\rho}{\zeta}                     & G\rho U^1 U^0                                    \\
   \rho G(U^1)^2+\frac{\rho}{\zeta} & \frac{\rho u_1 ((3G + \zeta)u_1^2+3G)}{\zeta/(U^0)^3}
                                    & \frac{\rho ((5G+\zeta)u^2_1  +G)}{\zeta/(U^0)^3}                                                         \\
   G\rho U^1 U^0                    & \frac{\rho ((5G+\zeta)u^2_1  +G)}{\zeta/(U^0)^3}      & \frac{\rho u_1(Gu_1^2+5G + \zeta)}{\zeta/(U^0)^3}
  \end{pmatrix}.
 \end{align*}

 \section{The matrices $M_0$, $M_1$ and $M_2$ for 2D Euler equations}\label{M2D}
 In the 2D case, the matrices $ M_{k}=\int_{\mathbb{R}^{3}}p^{k}g \vec{\Psi}\vec{\Psi}^Td\varXi$, $k=0,1,2$, have the explicit expressions
 \begin{align*}
  M_0 & =\int_{\mathbb{R}^{3}}p^{0}g\vec{\Psi}\vec{\Psi}^Td\varXi\notag \\
      & =
  \begin{small}
   \begin{pmatrix}
    \rho U^0                          & G\rho U^0 U^1
                                      & G\rho U^0 U^2                                          & \rho G (U^0)^2-\frac{\rho}{\zeta}                      \\
    G\rho U^0 U^1                     & \frac{\rho ((5G+\zeta)u_1^2-Gu_2^2+G)}{\zeta/(U^0)^3}
                                      & \frac{\rho u_1 u_2(6G+\zeta)}{\zeta/(U^0)^3}           & \frac{\rho u_1(Gu_1^2+Gu_2^2+5G+\zeta)}{\zeta/(U^0)^3} \\
    G\rho U^0 U^2                     & \frac{\rho u_1 u_2(6G+\zeta)}{\zeta/(U^0)^3}
                                      & \frac{\rho(-Gu_1^2+(5G+\zeta)u_2^2+G)}{\zeta/(U^0)^3}  & \frac{\rho u_2(Gu_1^2+Gu_2^2+5G+\zeta)}{\zeta/(U^0)^3} \\
    \rho G (U^0)^2-\frac{\rho}{\zeta} & \frac{\rho u_1(Gu_1^2+Gu_2^2+5G+\zeta)}{\zeta/(U^0)^3}
                                      & \frac{\rho u_2(Gu_1^2+Gu_2^2+5G+\zeta)}{\zeta/(U^0)^3} & \frac{\rho(3Gu_1^2+3Gu_2^2+3G+\zeta)}{\zeta/(U^0)^3}
   \end{pmatrix},
  \end{small}
 \end{align*}
 \begin{align*}
  M_1 & =\int_{\mathbb{R}^{3}}p^{1}g\vec{\Psi}\vec{\Psi}^Td\varXi\notag \\
      & =
  \begin{small}
   \begin{pmatrix}
    \rho U^1                         & \rho G(U^1)^2+\frac{\rho}{\zeta}
                                     & G\rho U^1U^2                                                 & G\rho U^0U^1                                           \\
    \rho G(U^1)^2+\frac{\rho}{\zeta} & \frac{\rho u_1((3G+\zeta)u_1^2 -3Gu_2^2 +3G)}{\zeta/(U^0)^3}
                                     & \frac{\rho u_2((5G+\zeta)u_1^2-Gu_2^2+G)}{\zeta/(U^0)^3}     & \frac{\rho((5G+\zeta)u_1^2-Gu_2^2+G)}{\zeta/(U^0)^3}   \\
    G\rho U^1U^2                     & \frac{\rho u_2((5G+\zeta)u_1^2-Gu_2^2+G)}{\zeta/(U^0)^3}
                                     & \frac{\rho u_1(-Gu_1^2+(5G+\zeta)u_2^2+G)}{\zeta/(U^0)^3}    & \frac{\rho u_1u_2(6G+\zeta)}{\zeta/(U^0)^3}            \\
    G\rho U^0U^1                     & \frac{\rho((5G+\zeta)u_1^2-Gu_2^2+G)}{\zeta/(U^0)^3}
                                     & \frac{\rho u_1u_2(6G+\zeta)}{\zeta/(U^0)^3}                  & \frac{\rho u_1(Gu_1^2+Gu_2^2+5G+\zeta)}{\zeta/(U^0)^3}
   \end{pmatrix},
  \end{small}
 \end{align*}
 and
 \begin{align*}
  M_2 & =\int_{\mathbb{R}^{3}}p^{2}g\vec{\Psi}\vec{\Psi}^Td\varXi\notag \\
      & =
  \begin{small}
   \begin{pmatrix}
    \rho U^2                         & G\rho U^1U^2
                                     & \rho G(U^2)^2+\frac{\rho}{\zeta}                             & G\rho U^0U^2                                                 \\
    G\rho U^1U^2                     & \frac{\rho u_2((5G+\zeta)u_1^2-Gu_2^2+G)}{\zeta/(U^0)^3}
                                     & \frac{\rho u_1(-Gu_1^2+(5G+\zeta)u_2^2+G)}{\zeta/(U^0)^3}    & \frac{\rho u_1u_2(6G+\zeta)}{\zeta/(U^0)^3}                  \\
    \rho G(U^2)^2+\frac{\rho}{\zeta} & \frac{\rho u_1(-Gu_1^2+(5G+\zeta)u_2^2+G)}{\zeta/(U^0)^3}
                                     & \frac{\rho u_2(-3Gu_1^2+(3G+\zeta)u_2^2 +3G)}{\zeta/(U^0)^3} & \frac{\rho(-Gu_1^2+(5G+\zeta)u_2^2+G)}{\zeta/(U^0)^3}        \\
    G\rho U^0U^2                     & \frac{\rho u_1u_2(6G+\zeta)}{\zeta/(U^0)^3}
                                     & \frac{\rho(-Gu_1^2+(5G+\zeta)u_2^2+G)}{\zeta/(U^0)^3}        & \frac{\rho u_2(Gu_1^2 + Gu_2^2 + 5G + \zeta)}{\zeta/(U^0)^3}
   \end{pmatrix}.
  \end{small}
 \end{align*}
 \section{1D Riemann problem}\label{sub:wavestructure}
For the  Riemann problem of the 1D special RHD equations,
three eigenvalues of the Jacobian matrix are $\lambda_-=\frac{u-c_s}{1-uc_s},\lambda_0=u,\lambda_+=\frac{u+c_s}{1+uc_s}$,
where $u=u_1$. The Riemann invariants, Rankine-Hugoniot conditions and the relations between the left and right states of the
elementary waves for the 1D RHD equations with the Synge EOS are given below.

 \subsection{Riemann invariants}
 The Riemann invariants associated with the characteristic field $\lambda_0$ are the pressure $p$ and velocity $u$ \cite{lanza1985},
  while the Riemann invariants associated with the characteristic field $\lambda_{\pm}$ are the entropy $S$ and $\psi_{\pm}$, which
 play a pivotal role in resolving the centered rarefaction waves. The concrete
 expressions are given as follows \cite{taub1948,lanza1985,synge1957}
   \begin{align}
     &S = -{\rm log}(\rho L(\zeta))+\mbox{const.}, \ \ L(\zeta)=\frac{\zeta}{K_2(\zeta)}\exp\left(-\frac{\zeta K_3(\zeta)}{K_2(\zeta)}\right),\\
     &\psi_{\pm} = \frac12{\rm ln}\left(\frac{1+u}{1-u}\right)\mp\int^{\rho}\frac{c_s(w,S)}{w}{\rm d} w \label{eq:Inviarant}.
   \end{align}
 The equation $\rho = p\zeta$ gives
 \begin{equation*}
   {\rm d}\rho = \frac{\partial \rho}{\partial p} {\rm d}p + \frac{\partial \rho}{\partial \zeta} {\rm d}\zeta
               = \zeta {\rm d}p + p {\rm d}\zeta.
 \end{equation*}
Using ${\rm d} S=0$ and the thermodynamic relation
 \begin{equation*}
   {\rm d}h = T{\rm d}S + \frac{1}{\rho} {\rm d}p,
 \end{equation*}
gives
 \begin{equation*}
   {\rm d} p = \rho {\rm d}h = \rho G'(\zeta){\rm d}\zeta.
 \end{equation*}
Hence, one has
 \begin{equation}\label{eq:drho}
   \frac{{\rm d\rho}}{\rho} = \zeta\left(G'+\frac{1}{\zeta^2}\right){\rm d}\zeta.
 \end{equation}
 Since $f(\zeta)=G(\zeta)-\frac{1}{\zeta}$ monotonically decreases
with respect to $\zeta$ \cite{synge1957}, we have $G'+\frac{1}{\zeta^2}<0$,
and the Riemann variants $\psi_{\pm}$ in \eqref{eq:Inviarant} can be rewritten as
 \begin{equation*}
   \psi_{\pm} = \frac12{\rm ln}\left(\frac{1+u}{1-u}\right)\pm\int^{\zeta}
   \left(\frac{\tilde\zeta(G'+1/\tilde\zeta^2)G'}{G}\right)^{1/2}{\rm d} \tilde\zeta.
 \end{equation*}

 \subsection{Rankine-Hugoniot conditions}
This section gives the jump conditions across the discontinuities for one-dimensional RHD equations of the perfect relativistic gas. Let the shock related to the characteristic field  $\lambda_{\pm}$ travel at speed $s$, $\vec{W}_a$ and $\vec{W}_b$ be the conservative variables in the wavefront and post-wave, respectively. Then the junction conditions across the shock satisfy
 \begin{equation*}
   s[\vec{W}]=[\vec{F}^1],
 \end{equation*}
where $[f]=f_b-f_a$ represents the discontinuity in the function involved.
If we choose our coordinate system such that the discontinuity is at rest, then the above equations become
 \begin{equation}\label{eq:RHconditions}
   \begin{aligned}
      &[\rho U^1] = \left[\frac{\rho u}{\sqrt{1-u^2}}\right] = 0,\\
      &[\rho hU^1U^1 + p] = \left[\frac{\rho h u^2}{1-u^2}+p\right] = 0,\\
      &[\rho hU^0U^1] = \left[\frac{\rho h u}{1-u^2}\right] = 0.
   \end{aligned}
 \end{equation}
It follows that the relativistic Rankine-Hugoniot equations are   \cite{synge1957}
 \begin{align}\nonumber
     &\frac{u_a-u_b}{1-u_au_b}=\mp\sqrt{\frac{(p_b-p_a)
     (\varepsilon_b-\varepsilon_a)}{(\varepsilon_a+p_b)(\varepsilon_b+p_a)}},\\
     &G_a^2-G_b^2=\left(\frac{G_b}{\rho_b}+\frac{G_a}
     {\rho_a}\right)(p_a-p_b)\label{eq:junction}.
 \end{align}
 Using the above results for the rarefaction wave and shock wave, we obtain the following theorem.
 \begin{theorem}
 For the wave associated with the characteristic field $\lambda_-$, one has
 \begin{align*}
   & p_l < p_r, ~ u_l > u_r,\ \text{for the shock wave},\\
   &p_l > p_r, ~ u_l < u_r,\  \text{for the rarefaction wave}.
 \end{align*}
 For the wave associated with the characteristic field $\lambda_+$, it holds
 \begin{align*}
   &p_l > p_r, ~ u_l > u_r,\ \text{for the shock wave},\\
   &p_l < p_r, ~ u_l < u_r,\ \text{for the rarefaction wave}.
 \end{align*}
 For the wave associated with the characteristic field $\lambda_0$,
 we have $p_l=p_r$ and $u_l=u_r$, where the subscripts $l$ and $r$ indicate the left and right states of the variables, respectively.
 \end{theorem}
 \begin{proof}
   (\romannumeral1) Since the Riemann invariants associated with the characteristic field $\lambda_0$ are the pressure $p$ and velocity $u$, it's easy to obtain that
   \begin{equation*}
     p_l=p_r,~ u_l=u_r.
   \end{equation*}
   (\romannumeral2) Suppose that the wave related to the characteristic field $\lambda_-$ is a rarefaction wave, the Lax entropy condition gives
   \begin{equation*}
     \frac{u_l-c^l_s}{1-u_lc^l_s}<\frac{u_r-c^r_s}{1-u_rc^r_s},
   \end{equation*}
   thus
   \begin{equation}\label{eq:u_lu_r}
     \frac{u_l-u_r}{1-u_lu_r}<\frac{c^l_s-c^r_s}{1-c^l_sc^r_s}.
   \end{equation}
If assuming $\zeta_r\leq \zeta_l$, then using the Riemann invariants $\psi_-$ gives
   \begin{equation*}
     \frac12\ln\left(\frac{1+u_l}{1-u_l}\right)-\frac12\ln\left(\frac{1+u_r}{1-u_r}\right)
     = \int_{\zeta_r}^{\zeta_l}\left(\frac{\zeta(G'+1/\zeta^2)G'}{G}\right)^{1/2}{\rm d} \zeta\geq 0,
   \end{equation*}
and then it's easy to obtain that
   \begin{equation*}
     u_l \geq u_r.
   \end{equation*}
Combining it with  \eqref{eq:u_lu_r}  deduces that $c_s^l-c_s^r>0$, which contradicts the fact that  $c_s$ is a monotonically decreasing function of $\zeta$.
   Therefore, for the rarefaction wave associated with $\lambda_-$, we have
   $$\zeta_r>\zeta_l,~u_l<u_r.$$
 From \eqref{eq:drho} and  $f(\zeta)=G(\zeta)-\frac{1}{\zeta}$ monotonically decreasing
with respect to $\zeta$, it holds that ${\rm d}\rho / {\rm d}\zeta<0$. Hence it's easy to get
   \begin{equation*}
     \rho_r<\rho_l.
   \end{equation*}
   By using $p=\rho/\zeta$, we obtain that
   $$p_r<p_l.$$
   (\romannumeral3) Suppose that the wave related to the characteristic field $\lambda_-$ is a shock wave, then one has \cite{synge1957}
   \begin{equation*}
      S_r > S_l,~ \zeta_l > \zeta_r.
   \end{equation*}
   Since $G(\zeta)$ is a monotonically decreasing function of $\zeta$, then $G_l<G_r$. Thus from \eqref{eq:junction} it's easy to get
   \begin{equation*}
     p_l < p_r.
   \end{equation*}
   Moreover, according to the third equation of \eqref{eq:RHconditions}, one has
   \begin{equation*}
     u_l > u_r.
   \end{equation*}

   (\romannumeral4) For the wave related to the characteristic field $\lambda_+$, the conclusion can be  similarly obtained.
 \end{proof}\qed

\end{appendices}

\section*{Acknowledgements}
This work was partially supported by
the Science Challenge Project, No. JCKY2016212A502
and
the National Natural Science
Foundation of China (Nos. 11901460,
11421101).

\bibliography{SRHD_BGKv3}
\bibliographystyle{plain}
\end{document}